\documentclass[leqno]{amsart}
\usepackage{a4wide}
\usepackage{graphicx}
\usepackage{latexsym}
\usepackage{amssymb,amsthm,amsmath,mathtools}
\usepackage[active]{srcltx}
\makeatletter

\usepackage{enumerate,url,amssymb,mathrsfs}
\usepackage{todonotes}

\calclayout

\usepackage{color}
\usepackage[pagebackref=true,colorlinks=false]{hyperref}

\definecolor{darkgreen}{rgb}{0.5,0.25,0}
\definecolor{darkblue}{rgb}{0,0,0.56}
\definecolor{answerblue}{rgb}{0,0,0.75}

\hypersetup{colorlinks,breaklinks,
            linkcolor=darkblue,urlcolor=darkblue,
            anchorcolor=darkblue,citecolor=darkblue}

\newtheorem{theorem}{Theorem}[section]
\newtheorem{lemma}[theorem]{Lemma}
\newtheorem*{lemma*}{Lemma}

\newtheorem{corollary}[theorem]{Corollary}

\theoremstyle{theorem}
\newtheorem{definition}[theorem]{Definition}

\theoremstyle{theorem}
\newtheorem{remark}[theorem]{Remark}

\numberwithin{equation}{section}

\makeatletter
\@namedef{subjclassname@2020}{%
  \textup{2020} Mathematics Subject Classification}
\makeatother

\def\mxi{\boldsymbol{\xi}}
\def\mPhi{\boldsymbol{\Phi}}
\def\malpha{\boldsymbol{\alpha}}
\def\cal{\mathcal}
\newcommand{\abs}[1]{\left\lvert#1\right\rvert}
\newcommand{\norm}[1]{\left \lVert#1 \right\rVert}
\newcommand{\inn}[1]{\left\langle#1\right\rangle}
\newcommand{\innb}[1]{\bigl\langle#1\bigr\rangle}

\newcommand{\pa}{\partial}

\newcommand{\R}{\mathbb{R}}

\newcommand{\cX}{\mathscr{X}}

\newcommand{\cD}{\mathcal{D}}
\newcommand{\cL}{\mathcal{L}}
\newcommand{\cA}{\mathcal{A}}
\newcommand{\cM}{\mathcal{M}}
\newcommand{\cB}{\mathcal{B}}
\newcommand{\cT}{\mathcal{T}}

\newcommand{\cI}{\mathcal{I}}
\newcommand{\cZ}{\mathcal{Z}}
\newcommand{\cY}{\mathcal{Y}}
\newcommand{\cC}{\mathcal{C}}
\newcommand{\cK}{\mathcal{K}}

\newcommand{\bS}{\mathbb{S}}
\newcommand{\prob}{\mathbf{P}}

\newcommand{\ddiv}{\operatorname{div}}
\newcommand{\Div}{\operatorname{div}}

\DeclareMathOperator{\sign}{sign\,}

\DeclareMathOperator{\Lip}{Lip}

\def\XXint#1#2#3{{\setbox0=\hbox{$#1{#2#3}{\int}$}
\vcenter{\hbox{$#2#3$}}\kern-.5\wd0}}

\def\le{\leqslant}
\def\ge{\geqslant}
\def\mx{{\bf x}}
\def\my{{\bf y}}

\def\mp{{\bf p}}
\def\mq{{\bf q}}
\def\mff{{\mathfrak f}}
\def\eps{\varepsilon}

\def\mF{{\mathfrak F}}
\def\F{{\cal F}}

\newcommand{\N}{\mathbb N} 
\newcommand{\cinfty}{C^\infty}
\newcommand{\grad}{{\,\mathrm{grad}}\,}

\usepackage[normalem]{ulem}

\newcommand{\tensoroo}{\mathcal{T}^1_1}
\newcommand{\supp}{\mathrm{supp}\,}
\newcommand{\mcf}{{\mathcal F}}
\newcommand{\lara}[2]{\left \langle #1, #2 \right\rangle}
\newcommand{\larab}[2]{\bigl \langle #1, #2 \bigr\rangle}

\newcommand{\X}{\mathfrak{X}}

\newcommand{\spann}{\mathrm{span}}
\newcommand{\Seq}[1]{\left\{#1\right\}}
\newcommand{\Seqb}[1]{\bigl\{#1\bigr\}}

\def\dscon{\relbar\joinrel\rightharpoonup}
\newcommand{\ton}{\xrightarrow{n\uparrow \infty}}
\newcommand{\tok}{\xrightarrow{k\uparrow \infty}}
\newcommand{\toL}{\xrightarrow{L\uparrow \infty}}
\newcommand{\tonweak}{\xrightharpoonup{n\uparrow \infty}}

\newcommand{\weak}{\rightharpoonup}
\newcommand{\weakstar}{\xrightharpoonup{\star}}
\newcommand{\simd}{\overset{d}{\sim}}
\newcommand{\loc}{\operatorname{loc}}
\newcommand{\akn}{\alpha_k^{(n)}}
\newcommand{\akm}{\alpha_k^{(m)}}
\newcommand{\ajn}{\alpha_j^{(n)}}
\newcommand{\set}[1]{\left\{#1\right\}}
\newcommand{\action}[2]{\left\langle #1, #2 \right\rangle}

\renewcommand{\ge}{\geqslant}
\renewcommand{\geq}{\geqslant}
\renewcommand{\le}{\leqslant}
\renewcommand{\leq}{\leqslant}

\begin{document}
\title[A stochastic vanishing dynamic capillarity limit]
{A dynamic capillarity equation with \\ 
stochastic forcing on manifolds: \\ 
a singular limit problem}

\author[K. H. Karlsen]{Kenneth H. Karlsen}
\address[Kenneth H. Karlsen]
{\newline Department of mathematics, University of Oslo, 
P.O. Box 1053,  Blindern, N--0316 Oslo, Norway} 
\email[]{kennethk@math.uio.no}

\author[M. Kunzinger]{Michael Kunzinger}
\address[M. Kunzinger]
{\newline University of Vienna, 
Faculty of Mathematics, Oskar Morgenstern--Platz 1, 
1090 Wien, Austria}
\email{michael.kunzinger@univie.ac.at}

\author[D. Mitrovic]{Darko Mitrovic}
\address[D. Mitrovic]
{\newline Faculty of Mathematics, University of
Vienna, Oskar Morgenstern--Platz 1, 1090 Wien, Austria 
\newline 
and 
\newline 
Faculty of Mathematics, 
University of Montenegro, 
Cetinjski put bb, 
Podgorica, Montenegro}
\email{darko.mitrovic@univie.ac.at}

% 60H15 Stochastic partial differential equations
% 35L02 First-order hyperbolic equations
% 35K70 Ultraparabolic equations, pseudoparabolic equations, etc.
% 42B37 Harmonic analysis and PDEs


\begin{abstract} 
We consider a dynamic capillarity equation with stochastic 
forcing on a compact Riemannian manifold $(M,g)$:
\begin{equation*}\tag{P}
	d \left(u_{\eps,\delta}
	-\delta  \Delta  u_{\eps,\delta}\right) 
	+\Div \mff_{\eps}(\mx, u_{\eps,\delta})\, dt
	=\eps \Delta u_{\eps,\delta}\, dt 
	+ \Phi(\mx, u_{\eps,\delta})\, dW_t,
\end{equation*} 
where $\mff_{\eps}$ is a sequence of smooth 
vector fields converging in $L^p(M\times \R)$ ($p>2$) 
as $\eps\downarrow 0$ towards a vector field 
$\mff\in L^p(M;C^1(\R))$, and $W_t$ is a Wiener process 
defined on a filtered probability space.
First, for fixed values of $\eps$ and $\delta$, 
we establish the existence and uniqueness 
of weak solutions to the Cauchy problem for (P). 
Assuming that $\mff$ is non-degenerate 
and that $\eps$ and $\delta$ 
tend to zero with $\delta/\eps^2$ bounded, we show 
that there exists a subsequence of 
solutions that strongly converges 
in $L^1_{\omega,t,\mx}$ to a martingale 
solution of the following stochastic 
conservation law with discontinuous flux:
$$
d u +\Div \mff(\mx, u)\,dt
=\Phi(u)\, dW_t.
$$ 
The proofs make use of Galerkin 
approximations, kinetic formulations as well as 
$H$-measures and new velocity averaging results for 
stochastic continuity equations. 
The analysis relies in an essential way on the 
use of a.s.~representations of random 
variables in some particular quasi-Polish spaces. 
The convergence framework developed here can be applied 
to other singular limit problems for 
stochastic conservation laws. 
\end{abstract}

\date{\today}

\maketitle

\tableofcontents

\section{Introduction and main results}\label{intro}

This paper develops a convergence framework 
for inspecting a vanishing diffusion-dynamic 
capillarity limit for stochastic scalar 
conservation laws on manifolds with discontinuous flux. 
The convergence framework is general enough to apply 
to many other convergence problems for stochastic conservation laws. 
We are starting from an equation describing the 
flow in a porous medium governed by convection, 
diffusion and dynamic capillarity effects. Convection and diffusion 
are widely known phenomena. Convection represents a motion 
induced by inertial forces and is modelled by a first order 
differential operator. According to Fick's law, diffusion 
results from Brownian motion, which is described 
by a second-order differential operator. 
In porous media, there is another vital phenomenon 
influencing fluid flow. 
This is the capillary pressure, the pressure 
difference across the interface between 
capillaries' wetting and non-wetting phases. 
If one assumes that the capillary pressure 
changes dynamically depending on 
the fluid concentration \cite{Hassanizadeh:1993oz,Hassanizadeh:1990ef}, 
after a linearization procedure, one obtains a pseudo-parabolic equation, 
which we considered in \cite{GKMV} 
(see also \cite{Duijn:2007lk}). 
Porous media flow phenomena often occur along a non-flat 
surface and are influenced by unpredictable (stochastic) 
sinks or sources. Therefore, it is natural to consider 
the following pseudo-parabolic stochastic partial differential 
equation (SPDE) on a manifold $M$, over a fixed time horizon $[0,T]$:
\begin{align}
	\label{PP-1}
	& d u_{k} +\Div \mff_k(\mx,u_k) \, dt
	=\eps_k \Delta u_k \, dt
	+\delta_k d \Delta  u_k 
	+\Phi(\mx, u_k)\, dW_t, \quad \mx \in M,
	\\
	\label{PP-ID}
	& u_k|_{t=0}=u_0(\mx) \in H^2(M), 
\end{align} 
where $\ddiv$ and $\Delta$ are the divergence and 
Laplace-Beltrami operators derived from the metric $g$ on 
the Riemannian manifold $M$. We assume $M$ to be 
compact and oriented. The pseudo-parabolic nature 
of \eqref{PP-1} comes from the term $\delta_k d \Delta  u_k$, 
where $d$ denotes the It\^{o} differential. 
As for the stochastic part of \eqref{PP-1}, we consider a real-valued 
Wiener process $W_t$ defined on a filtered 
probability space $\bigl(\Omega,\mcf,\prob,\{\F_t\}\bigr)$ 
satisfying the usual conditions \cite[p.~7]{Chow:2015aa}. 
We are ultimately interested in stochastic conservation laws 
with discontinuous flux $\mff$ (denoting by $\mff_k$ 
a regularisation of $\mff$). To simplify the 
presentation, the noise function $\Phi$ 
is assumed regular enough throughout 
the paper, so it is written without the index $k$, see 
below for precise conditions.

The diffusion $\eps_k$ and capillarity $\delta_k$ 
parameters in \eqref{PP-1} take small 
positive values that tend to zero as $k\to \infty$, and one 
of the main aims of the paper is 
to inspect the behaviour of the solutions $u_k$ 
as $\eps_k,\delta_k \to 0$ (i.e., as $k\to \infty$).  
As for the other objects appearing 
in \eqref{PP-1}, we assume:
\begin{itemize}
	\item[($C_f$--1)] For each $k\in \N$ and 
	$\lambda\in \R$, the mapping 
	$\mx\mapsto \mff_k(\mx,\lambda)$ 
	is a vector field on $M$, 
	and $\mff_k$ belongs to $C^1(M\times \R) 
	\cap L^{p_0}(M \times \R)$, for some $p_0>2$. 
	Moreover, $\mff_k':=\pa_{\lambda}\mff_k 
	\in L^{p_0}(M\times \R)$ and there is a limit 
	$\mff\in L^{p_0}(M;C^1(\R))$ such that\footnote{For brevity, given 
	a vector field $V$ on $M$ and $\mx\in M$, we write $\abs{V(\mx)}$ 
	to mean $\abs{V(\mx)}_g:=\action{V(\mx)}{V(\mx)}^{1/2}$.}
	\begin{equation}\label{conv}
		\norm{\sup\limits_{\lambda \in \R}
		\abs{\mff(\cdot,\lambda)
		-\mff_k(\cdot,\lambda)}}_{L^1(M)}+
		\norm{\sup\limits_{\lambda \in \R}
		\abs{\mff'(\cdot,\lambda)
		-\mff_k'(\cdot,\lambda)}}_{L^{p_0}(M)}\tok 0.
	\end{equation}

	\item[($C_f$--2)] $\mff_k'=\mff_k'(\mx,\lambda)$ 
	is bounded uniformly in $k\in \N$, 
	i.e., $\norm{\mff_k'}_{L^\infty(M\times\R)}\le C$, 
	for some constant $C$ independent of $k$. 
	Moreover, the limit $\mff$ satisfies 
	$\norm{\mff'}_{L^\infty(M\times\R)}\le C$. 
	
	\item[($C_f$--3)] The following growth condition holds 
	uniformly in $k\in \N$:
	\begin{equation*}%\label{GR}
		\left\|\mff_k(\cdot,\lambda)
		\right\|_{L^\infty(M)} 
		\leq C \bigl(1+|\lambda|\bigr),
	\end{equation*} 
	where $C>0$ is a constant.

	\item[($C_f$--4)] The geometry compatibility condition 
	\cite{Ben-Artzi:2007aa} holds for any $k\in \N$:
	\begin{equation}\label{gc}
		\Div \, \mff_k(\mx,\lambda)=0, 
		\quad \mx\in M,\  \lambda\in \R.
	\end{equation}

	\item[($C_\Phi$--1)] The noise function 
	$\Phi=\Phi(\mx,\lambda)\in C^1(M\times \R)$ is 
	at most linearly growing in $\lambda$: 
	\begin{equation}\label{assump-Phi}
		\left|\Phi(\mx,\lambda)\right|
		\le C\bigl(1+\left|\lambda\right|\bigr).
	\end{equation}

	\item[($C_\Phi$--2)] Bound on the 
	$\lambda$-derivative $\Phi'=\pa_{\lambda}\Phi$: 
	$\left\|\sup\limits_{\lambda \in \R}
	\bigl|\Phi'(\cdot,\lambda)\bigr|
	\right\|_{L^2(M)}<\infty$.

	\item[($C_\Phi$--3)] The noise function 
	$\Phi=\Phi(\mx,\lambda)$ is bounded 
	in $\lambda$ in the sense that
	$$
	\norm{\sup\limits_{\lambda\in \R}
	\abs{\Phi(\cdot,\lambda)}}_{L^2(M)}<\infty,
	$$
	which is a strengthening of ($C_\Phi$--1)
\end{itemize} 

These conditions reflect that we are 
working with $L^2$ solutions. In the case 
of higher integrability ($L^p$ solutions 
with $p>2$), one can weaken these conditions. 
Condition ($C_\Phi$--3) is only 
used to pass to the limit in the stochastic 
integral for solutions that converge strongly 
in $L^p$ with $p<2$ (see Lemma \ref{lem:weak-sol}).

We are first interested in the existence 
of a weak solution to \eqref{PP-1}, \eqref{PP-ID}, 
for any fixed value of $k$. Later we will return to 
the more difficult question of convergence 
of the solutions $u_k$ as $k\to \infty$.

\begin{theorem}[well-posedness of 
the pseudo-parabolic SPDE]\label{unique_sol_par}
Suppose the conditions ($C_f$--1)--($C_f$--3) 
and ($C_\Phi$--1)--($C_\Phi$--2) hold, and view 
$k\in \N$ as fixed, i.e., $\delta_k>0$ and $\eps_k>0$ 
are fixed numbers. There exists a unique 
adapted stochastic process 
\begin{equation}\label{eq:pseudo-spaces}
	u_k \in L_{\prob}^2\left(\Omega;
	L^\infty(0,T;H^1(M))\right)\cap 
	L_{\prob}^2\left(\Omega;C([0,T];L^2(M))\right),
\end{equation}
which solves problem \eqref{PP-1}, \eqref{PP-ID} in the sense that 
$\forall \varphi\in C^2(M)$ and $\forall t\in [0,T]$,
\begin{equation}\label{mild-ppde}
	\begin{split}
		& \int_M  u_k(t) \varphi \, dV 
		-\int_M u_0 \varphi \, dV 
		\\ & \quad 
		-\int_0^t \int_M \mff_k(\mx, u_k) 
		\nabla \varphi \, dV \, dt'
 		=\eps_k\int_0^t \int_M u_k 
 		\Delta \varphi \,dV\, dt'
 		\\ & \quad \quad
 		+\delta_k \left(\int_M u_k(t) 
 		\Delta \varphi \, dV
 		-\int_M u_0\Delta \varphi \, dV \right)
 		\\ & \quad \quad\quad
 		+\int_0^t\int_M \Phi(\mx, u_k)
 		\varphi \,dV\, dW_{t'}, 
 		\quad \text{almost surely}.
	\end{split}
\end{equation}
Finally, there exists a constant $C_T$, 
independent of $k$, such that 
\begin{equation}\label{eq:L2-uest-thm}
	E \left[\sup_{t\in [0,T]}
	\norm{u_k(t)}_{L^2(M)}^2
	\right]\le C_T,
\end{equation}
assuming $u_0\in L^2(M)$ and 
$\delta \left\| u_0 \right\|_{H^1(M)}^2\lesssim 1$.
\end{theorem}

The notations used in the previous theorem will 
be defined in the next section. 
We note that \eqref{mild-ppde} is well-defined 
if $u_k$ merely belongs to $L_{\prob}^2\bigl(\Omega; 
L^2((0,T)\times M)\bigr)$, but 
to prove uniqueness, we need solutions of 
higher ($H^1$) regularity. The proof of 
Theorem \ref{unique_sol_par} uses a stochastic variant of 
the Galerkin method.

Our second theorem deals with the 
vanishing diffusion-dynamic capillarity limit and 
requires a non-degeneracy condition on the flux function. 
Such conditions are standard in many 
results using velocity averaging lemmas (cf., e.g., 
\cite{LM5,Lions:1994qy,Pan,Perthame:1998if,Tadmor:2006vn}). 
Moreover, since an averaging lemma is a local 
result, i.e. we look for local convergence and may 
thus reduce considerations to local charts, 
it is enough to formulate it in Euclidean space. 
We have the following definition (for the flux 
in \eqref{PP-1}, take $m=1$).

\begin{definition}[non-degenerate flux]\label{def-non-deg}
We say that a function $\mff: \R^d\times \R^m 
\to \R^d$ ($d,m\in \N$) satisfies the non-degeneracy 
condition if, for any measurable set 
$K \times \Lambda \subset\subset 
\R^d \times \R^m$ and $C\in \R$, 
\begin{align}\label{non-deg}
	\int_{K} \, \sup\limits_{(\xi_0,\mxi')\in \bS^d} 
	\mathrm{meas}_{\lambda} \Bigl\{(\mx,\lambda)
	:\lambda \in \Lambda, \,\, 
	\xi_0+\bigl\langle \mff'(\mx,\lambda),\mxi'
	\bigr\rangle=C \Bigr\}\, d\mx =0,
\end{align} 
where $\mathrm{meas}_{\lambda}$ 
is the Lebesgue measure on $\R^m$ and 
$\bS^d$ is the unit sphere in $\R^{d+1}$. 
\end{definition} 

Many fluxes $\mff=\mff(\mx,\lambda)$ that are nonlinear 
with respect to $\lambda$ satisfy \eqref{non-deg}. 
For example (ignoring the growth condition), $\mff(\mx,\lambda)
=\bigl(x_1\lambda^2,x_2^2\lambda^3)$, if $m=1$ 
(scalar in $\lambda$) and $d=2$ ($\mx=(x_1,x_2)\in \R^2$).

We are interested in the singular limit 
$\eps_k,\delta_k \to 0$. Unlike Theorem \ref{unique_sol_par}, 
which supplies probabilistic strong 
solutions of \eqref{PP-1}, i.e., solutions defined relative 
to the original filtered probability space 
$\bigl(\Omega,\mcf,\prob,\Seqb{\mcf_t}\bigr)$, we now need 
to use the concept of martingale (or probabilistic weak) 
solutions for the limiting process, meaning that a new 
filtered probability space $\bigl(\tilde\Omega,\tilde\mcf,\tilde\prob,
\Seqb{\tilde{\mcf}_t}\bigr)$ is constructed along 
with a new Wiener process $\tilde W_t$ and a limit 
$\tilde{u}$ that together solve the stochastic conservation law. 
According to the famous Yamada--Watanabe principle, 
the constructed solution is probabilistic strong if 
the limit SPDE exhibits a pathwise uniqueness result, 
as for example in \cite{Debussche:2010fk,Galimberti:2018aa}.

The following theorem requires the 
additional assumption ($C_f$--4). 
This condition was first introduced 
in \cite{Ben-Artzi:2007aa}, where it 
was named the \emph{geometry compatibility condition}. 
Offhand it may appear purely technical, but, 
as noted in \cite{Graf:2017aa}, it 
represents a kind of incompressibility condition. Namely, 
\eqref{PP-1} serves as a model of fluid flow 
in porous media along non-flat surfaces and 
assumption ($C_f$--4) is related to 
incompressible fluids. 

\begin{theorem}[convergence to SPDE with 
discontinuous flux]\label{main-thm} 
Suppose conditions ($C_f$--1)--($C_f$--4) 
and ($C_\Phi$--1)--($C_\Phi$--2) hold.  
Suppose $\delta_k=o\left(\eps_k^2\right)$ in the sense that
\begin{equation}\label{neps}
	\text{$\eps_k\to 0$}
	\quad \text{and} \quad 
	\text{$\frac{\delta_k}{\eps^2_k}$}
	\quad \text{is uniformly bounded with respect to $k$}.
\end{equation} 
Moreover, suppose $\mff$ satisfies ($C_f$--2) 
and is non-degenerate in the sense 
of Definition \ref{def-non-deg}. 
Then there exist a filtered probability 
space $\bigl(\tilde\Omega,\tilde\mcf,\tilde\prob,
\Seqb{\tilde{\mcf}_t}\bigr)$, an $\tilde{\mcf}_t$ 
adapted Wiener process $\tilde{W}$ with 
$\tilde{W}\in C([0,T])$ $\tilde\prob$-a.s., 
and an $\tilde{\mcf}_t$ adapted process $\tilde{u}$, 
with $\tilde{u}\in C([0,T];H^{-1}(M))$ 
$\tilde\prob$-a.s.~and 
$$
\tilde{u} \in L^2_{\tilde{\prob}}\bigl(\tilde{\Omega}; 
L^\infty(0,T;L^2(M))\bigr),
$$ 
such that, modulo the extraction of a subsequence,
$$
\tilde{u}_k \tok \tilde{u} \quad 
\text{in $L^1_{\tilde{\prob}}
\bigl(\tilde\Omega;L^p((0,T)\times M)\bigr)$}, 
\quad \forall p\in [1,2),
$$
where $\Seq{\tilde{u}_k}$ is a Skorokhod-Jakubowski 
a.s.~representation of $\Seq{u_k}$, 
cf.~Section \ref{sec:prelim}. 

If, in addition 
to the conditions already listed, 
($C_\Phi$--3) holds, then $\tilde{u}$ is a 
weak solution to the following stochastic conservation 
law with discontinuous flux:
\begin{equation}\label{cl-1}
	d\tilde{u}+\ddiv \mff(\mx,\tilde{u})\, dt
	= \Phi(\mx, \tilde{u})\, d\tilde{W}_t,
	\quad \tilde{u}\big|_{t=0}=\tilde{u}_0
	\in L^2_{\tilde{\prob}}
	\bigl(\tilde{\Omega}; L^2(M)\bigr).
\end{equation}
The SPDE \eqref{cl-1} is interpreted weakly 
in $\mx$ in the sense that the following 
equation holds $\tilde\prob$-almost surely, for all 
$\varphi\in C^2(M)$ and for all $t\in [0,T]$,
\begin{equation}\label{mild-cl}
	\begin{split}
		\int_M \tilde u(t)\, \varphi\, dV dt
		& = \int_M \tilde u_0\, \varphi\, dV dt 
		+\int_0^t \int_M \mff(\mx, \tilde u)\cdot 
		\nabla \varphi \, dV \, dt' \\ & \qquad
		+\int_0^t\int_M \Phi(\mx,\tilde u) 
		\varphi(\mx) \, dV \, d\tilde W_{t'}.
	\end{split}
\end{equation}
\end{theorem}

To prove this theorem, we will use an 
approach that rewrites the SPDE under 
consideration in kinetic form and then apply 
a so-called velocity averaging result \cite{AMP,HKM,HTz}.
A significant difficulty hampering a naive adaptation of 
this ``deterministic" approach to the stochastic 
situation is that the probability 
and temporal variables $\omega,t$ cannot 
be incorporated in the same 
way as the other variables (more on this below). 
Besides, an appropriate stochastic 
velocity averaging result is not available in 
the literature.

Roughly speaking, 
$h_k=\sign\bigl(u_k(\omega,t,\mx)-\lambda\bigr)$, 
where $u_k$ solves \eqref{PP-1}, satisfies a 
kinetic SPDE in $\cD_{t,\mx,\lambda}'$, $\prob$--a.s., 
of the form
\begin{equation}\label{kinetic-intro}
	\begin{split}
 		& dh_k+\Div_{\mx} 
 		\bigl(\mff'(\mx,\lambda)h_k\bigr)\,dt 
 		\\ & \quad 
 		= \Bigl[\Div_{\mx} \left( \overline G_k^{(1)}
 		+ G_k^{(1)}\right)
 		+ \partial_{\lambda} \left( \overline G_k^{(2)}
 		+ G_k^{(2)}\right)
 		+ \partial_\lambda\Div_{\mx} \left( \overline G_k^{(3)}
 		+ G_k^{(3)}\right)\Bigr]\,dt
 		+\Bigl[\overline\Phi_k+\partial_\lambda \Phi_k 
 		\Bigr]\, dW_t,
	\end{split}
\end{equation} 
see Lemma \ref{kinetic-L} for the definition 
of the terms on the right-hand side.
We note that this equation is not entirely straightforward to 
derive because of It\^{o}'s formula and the appearance of 
two temporal differentials in \eqref{PP-1}. 
The terms on the right-hand side of \eqref{kinetic-intro} 
are partial derivatives of processes that are either bounded in 
$L^1_{\omega}\bigl(\mathcal{M}_{t,x,\lambda}\bigr)$, bounded or 
converge to zero in $L^2_{\omega}\bigl(L^2_{t,x,\lambda}\bigr)$, 
or bounded or converge to zero in 
$L^2_{\omega}\bigl(L^2_{t,x,\lambda,\loc}\bigr)$. 
Regarding stochastic velocity averaging results, only 
a few are available, see 
\cite{Gess:2018ab,Lions:2013ab,Nariyoshi-vel-av:2020,Punshon-Smith:2018aa}, 
neither of which can be applied to the problem at hand.

One central part of the paper consists of deriving 
stochastic velocity averaging results for kinetic SPDEs---like 
\eqref{kinetic-intro}---with (non-smooth) spatially dependent coefficients. 
These results are of independent interest as they can be applied to numerous 
other singular limit problems for stochastic conservation laws. 
The primary tool for achieving this goal will be a stochastic variant of 
$H$-measures (see Theorem \ref{thm:H-measure}) and the 
Skorokhod-Jakubowski representation theorem \cite{Jakubowski:1997aa} 
applied to some particular quasi-Polish spaces. 
We briefly recall that $H$-measures were 
introduced independently by Tartar in \cite{Tartar:1990mq} 
and by Gerard in \cite{Gerard:1991kl}, with many different 
variants appearing later on, cf.~\cite{AEL,AM-1,LM2,MM,Pan,rind}.
Here, relying on the most up to date result \cite{LM5}, 
we shall derive a version adapted to 
the stochastic situation.  

The $H$-measures rely crucially on Fourier 
transformation techniques, which are inherently problematic 
to apply in the probability and temporal variables 
given the stochastic structure of \eqref{kinetic-intro}.
At the moment, there is no deterministic $H$-measure 
theorem that applies to a general measure space. 
The natural a priori estimates supply weak 
convergence in all variables $(\omega,t,\mx)$, including 
the probability variable $\omega$. First, to sidestep 
the problem with the probability variable in the construction of the 
$H$-measure, we will replace the processes in 
\eqref{kinetic-intro} by almost surely convergent 
versions, defined on a different probability space 
$\bigl(\tilde \Omega,\tilde{\mathcal{F}},\tilde{\prob}\bigr)$. 
Such constructions, linked to tightness and 
weak compactness of the probability laws,  date 
back to the work of Skorokhod 
(see e.g.~\cite[Theorem 2.4]{DaPrato:2014aa}), 
for processes taking values in a Polish (separable 
completely metrizable) space. In our context, the 
Skorokhod theorem is not directly applicable 
because we have to work in specific spaces 
equipped with a weak topology, including 
$L^2$ and the space $\cM$ of Radon measures.  
Therefore we use a recent variant of the Skorokhod 
theorem that applies to so-called quasi-Polish 
spaces, see Jakubowki \cite{Jakubowski:1997aa}, where 
quasi-Polish refers to a Hausdorff space $\cX$ 
that exhibits a continuous injection 
into a  Polish space. Separable Banach spaces endowed 
with the weak topology and dual 
spaces of separable Banach spaces (with the 
weak-$\star$ topology) are quasi-Polish.  For other 
SPDE applications of the Skorokhod-Jakubowki theorem, we refer to 
\cite{Breit:2018aa,Brzezniak:2013aa,Brzezniak:2011aa,
Brzezniak:2013ab,Ondrejat:2010aa,Punshon-Smith:2018aa}, 
to mention a few examples.

However, the above manoeuvre does not resolve the predicament  
with the time variable $t$. Indeed, applying the Skorokhod-Jakubowski
procedure we arrive at the weak convergence of $h_k$ in $L^2_{t,x}$, 
almost surely, and the weak convergence in $t$ does 
not allow us to adapt deterministic $H$-measure techniques. 
The idea put forward here is to replace the 
space $(L^2_{t,\mx,\lambda})_w$ with the locally convex space 
$L^2_t\bigl((L^2_{\mx,\lambda})_w\bigr)$, where $(L^2_{\mx,\lambda})_w$ 
refers to the $L^2$ space in the variables $\mx,\lambda$ 
with the weak topology. Note carefully that the topology 
in $t$, on the other hand, is strong. One can verify 
that $L^2_t\bigl((L^2_{\mx,\lambda})_w\bigr)$ is quasi-Polish, thus 
the Skorokhod-Jakubowki theorem can be applied 
in this space. Let us mention here that the solutions 
$h_k$ of the kinetic SPDE \eqref{kinetic-intro} are 
not uniformly (H\"older) continuous in time, even if 
the topology in $\mx,\lambda$ is weak. This fact, which is due to 
the singular nature of the source terms of \eqref{kinetic-intro}, prevents 
the use of a more traditional approach based 
on the quasi-Polish space $C_t\bigl((L^2_{\mx,\lambda})_w\bigr)$, 
see, e.g., \cite{Brzezniak:2011aa,Ondrejat:2010aa}.

Of course, this leaves the question 
of how to establish the tightness of the probability 
laws $\prob\circ h_k^{-1}$ on $L^2_t\bigl((L^2_{\mx,\lambda})_w\bigr)$, 
which is no longer straightforward. Our main observation here is 
that this follows from the fact that $h_k$ is pointwise 
bounded ($-1\leq h_k\leq 1$) and that the kinetic 
SPDE \eqref{kinetic-intro} can be used to 
derive the temporal translation estimate
$$
E \left[\sup_{\tau\in (0,\vartheta)}
\norm{h_k(\cdot,\cdot+\tau,\cdot,\cdot)
-h_k}_{L^1(0,T-\tau;H^{-N}(M\times L))}\right]
\lesssim_L \vartheta^{1/2}, 
\quad L\subset\subset \R,
$$
where the right-hand side does not depend on $k$, provided $N$ 
is chosen sufficiently large. Denote by $\tilde h_k$ 
the Skorokhod-Jakubowki a.s.~representation of 
$h_k$ on $L^2_t\bigl((L^2_{\mx,\lambda})_w\bigr)$. 
Using the ideas just outlined, it follows that there is 
a limit $\tilde h$ such that $\tilde h_k
\tok \tilde h$ in $L^2_t\bigl((L^2_{\mx,\lambda})_w\bigr)$, 
almost surely. We can now adapt $H$-measure ideas 
from the deterministic situation to deduce 
that, passing if necessary to a subsequence,
$$
\tilde h_k \tok \tilde h 
\quad \text{in $L^2_{\tilde \omega,t,\mx}
\bigl((L^2_{\lambda})_w\bigr)$},
$$
where the $(L^2_{\lambda})_w$ convergence happens 
locally on $\R$. Hence, as required, the velocity averages 
of $\tilde h_k$ converge strongly in $L^2_{\tilde \omega,t,\mx}$.

The paper is organized as follows. 
Section \ref{sec:prelim} is devoted to the 
gathering of relevant background material. 
We establish a variant of the $H$-measure 
in Section \ref{sec:H-measure}. In Section \ref{sec:well-posed} 
we prove Theorem \ref{unique_sol_par}. 
In Section \ref{sec:velocity-averaging} 
we supply the required velocity 
averaging result, and in 
Section \ref{sec:singular-limit} we 
conclude the proof of Theorem \ref{main-thm}.
 
\section{Preliminary material}\label{sec:prelim}

This section briefly reviews some basic aspects of 
stochastic analysis \cite{DaPrato:2014aa,Chow:2015aa}, 
harmonic analysis \cite{Grafakos:2008,Stein:1970pr}, 
and differential geometry \cite{Lee:2003aa}. 

\subsection{Stochastics}

We refer to \cite[Chapter 1]{Chow:2015aa} and 
\cite[Chapter 2]{Breit:2018aa} for basic background 
material on stochastic analysis, including stochastic 
integrals, the It\^o chain rule, and martingale inequalities. 
An example of the latter is the 
Burkholder-Davis-Gundy (BDG) inequality, which will be used 
to bound moments of It\^{o} integrals in terms 
of their quadratic variation. Let $Y=\left(Y_t\right)_{t\in [0,T]}$ be a 
continuous local martingale (taking values in $\R$) with $Y_0=0$. 
Then, for any stopping time $\tau \leq T$,
\begin{equation}\label{eq:BDG}
	E\left[\left(\sup_{t\in[0,\tau]}\abs{Y_t}\right)^p\,\right] 
	\leq C_p\,E \left[
	\left\langle Y,Y\right\rangle_\tau^{\frac{p}{2}}
	\right],
	\qquad p\in (0,\infty),
\end{equation}
where $C_p$ is a universal constant.

We shall frequently use Bochner spaces like 
$L^r_{\prob}\bigl(\Omega;L^p(0,T;B)\bigr)$, 
for $p,r\in [1,\infty]$, where $B$ is a (separable) Banach space, for 
example $B=C(M)$, $B=L^q(M)$, $B=W^{k,r}(M)$. We will often 
drop the subscript $\prob$ on the outer space. 
For basic properties of Bochner spaces, we 
refer to \cite[Chapters 1 \& 2]{BanachI:2016} 
and also \cite{Ewards:Book65}.

We note that $M$ is a compact orientable 
Riemannian manifold, so that $C(M)=C_0(M)=C_c(M)$,
where $C_0(M)$ is the space of continuous 
functions vanishing at the boundary, and $C_c(M)$ 
the space of continuous functions with compact support.
Moreover, we use the notation $\mathcal{D}'(M)=\mathcal{D}(M)'$, 
$\mathcal{D}(M)=C^\infty(M)$. Note that 
$L^1(M)\subset \mathcal{D}'(M)$ via the identification 
$L^1(M)\ni u\mapsto \left[\mathcal{D}(M)\ni \varphi \mapsto 
\int_M f \varphi\, dV \right]$, where $dV=dV_g$ 
denotes the Riemannian volume density $\sqrt{\det g_{ij}}$.

\medskip

In this paper, we work with kinetic SPDEs 
containing stochastic objects that exhibit 
limited temporal regularity. These objects 
can be viewed as random variables $X:\Omega \to \cX$ 
taking values in spaces like $\cX=L^2$ 
equipped with the weak topology or $\cX=\cM$ 
equipped with the weak-$\star$ topology. 
The lack of temporal regularity makes interpreting 
these objects' measurability (adaptivity) subtle. 
Therefore, often without explicitly mentioning it, we will 
analyse these objects as random distributions, 
a concept recently proposed and 
developed \cite[Section 2.2]{Breit:2018aa} 
to handle problems with limited temporal regularity.

\begin{definition}[random distribution]\label{2.1}
Let $\bigl(\Omega,{\cal F},\prob\bigr)$ be 
a complete probability space. A random mapping 
$$
X:\Omega \to \cD'((0,T)\times M)
$$
is called a random distribution if 
$\innb{X,\varphi}:\Omega \to \R$ 
is a measurable function for any 
test function $\varphi \in \cD((0,T)\times M)
:=C^\infty_c((0,T)\times M)$.
\end{definition} 

Although being very general objects, random distributions come 
with most of the important concepts from stochastic calculus. 
We will make use of the following definition.

\begin{definition}
Let $X$ be a random distribution in $\cD'((0,T)\times M)$. Then
\begin{itemize}
	\item $X$ is adapted to a 
	filtration $\Seqb{\F_t}_{t\in (0,T)}$ if 
	$\innb{X,\varphi}$ is $\F_t$-measurable 
	for any $\varphi \in \cD((-\infty,t)\times M)$, 
	for all $t\in (0,T)$;
	
	\item the family of $\sigma$-algebras 
	$\Seqb{\sigma_t(X)}_{t\in (0,T)}$ given as
	$$
	\sigma_t(X):= \bigcap_{s<t} 
	\sigma\Bigl(\, \bigcup_{\varphi 
	\in \cD((-\infty,s)\times M)} 
	\Seqb{\innb{X,\varphi}<1}
	\cup \Seqb{N\in \F, \; \prob(N)=0}\Bigr)
	$$ 
	is called the history of $X$.
\end{itemize}
\end{definition} 

For a random distribution $X$, the history 
of $X$ is a natural filtration on 
the probability space linked to $X$. 
Consider a random mapping $X:\Omega\to \cX$ 
taking values in a topological vector space $\cX$ that 
is continuously embedded in $\cD'((0,T)\times M)$, 
such that the law of $X$ is tight (see next subsection). 
Then \cite[Theorem 2.2.3]{Breit:2018aa} shows
that the measurability of $X$ is determined 
by its measurability in the sense of distributions.

Equality in law between two 
random distributions is defined as follows:

\begin{definition}
We say that two random distributions $X$ 
and $\tilde X$, defined respectively on 
$\bigl(\Omega,\F,\prob\bigr)$ and 
$\bigl(\tilde\Omega,\tilde\F,\tilde \prob\bigr)$, 
coincide in law, denoted by 
$X\simd \tilde X$ (or $X\sim \tilde X$) if 
for the laws $\cL_{\R^k}$ we have
\begin{equation}\label{equiv-prob}
	\cL_{\R^k} \bigl(
	\inn{X,\varphi_1}, \ldots, 
	\inn{X,\varphi_k}\bigr)
	=\cL_{\R^k} \left(
	\innb{\tilde X,\varphi_1}, \ldots, 
	\innb{\tilde X,\varphi_k}  \right),
\end{equation} 
for any $k$-tuple of test functions 
$\varphi_1,\ldots,\varphi_k \in 
\cD((0,T)\times M)$. Equivalently,
$$
E\Bigl[F\bigl( \inn{X,\varphi_1}, 
\ldots,\inn{X,\varphi_k}\bigr)\Bigr]
= \tilde E \Bigl[F\bigl( \inn{X,\varphi_1}, 
\ldots,\inn{X,\varphi_k}\bigr)\Bigr], 
\quad F\in C_b(\R^k).
$$
\end{definition} 
We will also write $\cL_{\prob}\sim \cL_{\tilde{\prob}}$ 
instead of \eqref{equiv-prob}. According 
to \cite[Theorem 2.2.12.]{Breit:2018aa}, for random variables $X$ 
as above with laws that are Radon measures, it 
is irrelevant which topology we use to compare their laws.

\subsection{Quasi-Polish spaces and a.s.~representations}
Jakubowski \cite{Jakubowski:1997aa} 
defines a quasi-Polish space to be 
a topological space $\bigl(\cX,\tau\bigr)$ for which there 
is a countable family $\Seq{f_\ell:\cX
\to [-1,1]}_{\ell\in \N}$ of $\tau$-continuous 
functions $f_\ell$ that separate points of $X$. 
This implies that $\bigl(\cX,\tau_f\bigr)$ and 
thus $\bigl(\cX,\tau\bigr)$ are Hausdorff (the latter 
not necessarily regular). 
Both topologies coincide on $\tau$-compact 
sets, and thus $\tau$-compacts are metrisable.

Denote by $\cB(\cX)$ the $\sigma$-algebra of 
Borel subsets of $\bigl(\cX,\tau\bigr)$ 
and by $\cB_f$ the $\sigma$-algebra generated 
by the separating sequence $\Seq{f_\ell}_{\ell\in \N}$.
Then, every compact subset $K\subset \cX$ 
is $\cB_f$-measurable and thus metrisable. 
In general, $\cB_f$ may be smaller than 
the Borel $\sigma$-algebra, $\cB_f\subset \cB(\cX)$. 
However, as is shown in \cite{Jakubowski:1997aa}, 
this is not a problem as long as we work with random 
mappings with tight laws. 

\begin{definition}[tight probability measures]
\label{def:tight-measures}
We say that a sequence $\Seq{\mu_n}$ 
of probability measures on $\bigl(\cX,\cB\bigr)$, 
where $\cB$ is a $\sigma$-algebra (e.g.~$\cB(\cX)$ 
or $\cB_f$), is tight if, for 
any $\kappa>0$, there exists 
a compact set $K_\kappa \subset \cX$ such that
$$
\mu_n(K_\kappa)>1-\kappa, 
\quad \text{uniformly in $n$}.
$$
\end{definition}

Consider a random variable $X:\bigl(\Omega,\prob,\F\bigr)\to 
\bigl(\cX,\cB\bigr)$, $\Omega\ni \omega \mapsto 
X(\omega)\in \cX$, where $\bigl(\Omega,\prob,\F\bigr)$ 
is a probability space and $\cB$ is a 
$\sigma$-algebra (e.g.~$\cB(\cX)$ or $\cB_f$). 
On $\bigl(\cX,\cB\bigr)$ define 
the probability measure $\cL$ by
\begin{equation*}%\label{eq:law-of-X}
	\cL(A)=P(X^{-1}(A)), \quad A\in \cB,
\end{equation*} 
i.e., $\cL$ is the pushforward of $\prob$ 
by $X$---often denoted by $\cL=\prob\circ X^{-1}$. 
We refer to $\cL$ as the law of $X$ (on $\cB$). 

If $X:\Omega\to \cX$ 
is $\cB_f$-measurable, if the law $\cL$ of $X$ 
is tight as a probability measure on $\cB_f$,
and if $F:\bigl(\cX,\tau\bigr)\to \R$ is continuous, 
then $F(X)$ is Borel measurable if we 
replace $\F$ by its $\prob$-completion 
$\overline{\F}$ \cite{Jakubowski:1997aa}. 
The law $\cL$, viewed as a probability measure 
on $\bigl(\cX,\cB_f\bigr)$, has a unique 
Radon extension to $\cB(\cX)\supset \cB_f$ 
\cite{Jakubowski:1997aa}. 

The following theorem is the 
main result of Jakubowski \cite{Jakubowski:1997aa} 
(see \cite{Breit:2018aa,Brzezniak:2013aa,Brzezniak:2011aa,
Brzezniak:2013ab,Ondrejat:2010aa} for 
the first applications of this theorem to SPDEs).

\begin{theorem}[Skorokhod--Jakubowski]\label{thm:Jak-Skor}
Let $\Seq{\mu_n}$ be a tight sequence of probability 
measures on $\bigl(\cX,\cB_f\bigr)$, where 
$\bigl(\cX,\tau\bigr)$ is a quasi-Polish space with 
separating sequence $\Seq{f_\ell}$ and 
$\sigma$-algebra $\cB_f\subset \cB(\cX)$ generated 
by $\Seq{f_\ell}_{\ell\in \N}$. Then there exist 
a subsequence $\Seq{n_k}$ and $\cX$-valued 
Borel measurable random variables $\Seq{X_k}$, 
$X$---all defined on the standard probability space 
$\bigl([0,1],\cB([0,1]),\operatorname{Leb}\bigr)$---such 
that the law of $X_k$ is $\mu_{n_k}$ and 
$X_k \to X$ almost surely, in the topology of $\cX$. 
\end{theorem}

In this work, we will need to use 
non-metric spaces like
\begin{equation}\label{eq:qp-spaces-tmp1}
	\bigl(\cX,\tau\bigr)
	=\bigl(L^2((0,T)\times M\times L),\tau_w\bigr),
	\, 
	\bigl(\cM((0,T)\times M\times L),\tau_{w\star}\bigr),
\end{equation}
where $L\subset \subset \R$ and $\tau_w$ ($\tau_{w\star})$ 
refers to the weak (weak--$\star$) topology. 
These are not Polish spaces, but quasi-Polish 
(so that Theorem \ref{thm:Jak-Skor} applies). 
Indeed, for a given separable Banach space $\bS$, the 
dual space $\bigl(\bS^\star,\tau_{w\star}\bigr)$---equipped 
with weak--$\star$ topology $\tau_{w\star}$---is quasi-Polish.  
Applying this claim, it follows that 
$\bigl(L^2,\tau_w\bigr)$ and $\bigl(\cM,\tau_{w\star}\bigr)$ 
are both quasi-Polish spaces. 

For completeness, let us demonstrate the claim that 
$\bigl(\bS^\star,\tau_{w\star}\bigr)$ is 
quasi-Polish (if $\bS$ is a separable Banach space). 
There are several ways one can do this. 
Set $\Bbb{B}:=\bS^\star$ and denote 
by $\norm{\cdot}_{\Bbb{B}}$ the norm on 
the Banach space $\Bbb{B}$. We must 
supply a sequence of separating functions 
on $\Bbb{B}$. We start by expressing 
$\Bbb{B}= \bigcup_{j=1}^\infty \Bbb{B}_j$, where
each $\Bbb{B}_j:= \Seq{L\in \Bbb{B}: 
\norm{L}_{\Bbb{B}}\le j}$ is a separable 
metric space that is compact (closed) 
in the subspace topology induced by $\tau_{w\star}$ 
(by the Banach-Alaoglu theorem). 
For each $j$, there is a sequence 
$f_{\ell}^{(j)}:\Bbb{B}_j\to \R$, $\ell\in\N$, 
of continuous functions that separate points. 
Each function $f_{\ell}^{(j)}:\Bbb{B}_j \to \R$ 
can be extended to a continuous function 
$\overline{f}_{\ell}^{(j)}:\Bbb{B}\to \R$ 
by the Tietze-Urysohn-Brouwer extension theorem, which states 
that a continuous function on a closed 
subset of a normal topological space 
can be extended to the entire space (preserving boundedness).
The countable collection of 
continuous functions (separating points on $\Bbb{B}$) 
is then provided by $\Seq{\overline{f}_{\ell}^{(j)}}_{(\ell,j) \in \N^2}$.

For later use, note in particular that $L^1$ 
belongs to the $\sigma$-algebra $\cB_f$ for 
$\bigl(\cM,\tau_{w\star}\bigr)$. 
Indeed, we can write $L^1=\bigcup_{j\in \N} \Bbb{B}_j$, 
where $\Bbb{B}_j$ is the closed ball in $L^1$ 
with radius $j$. Each $\Bbb{B}_j$ is a compact subset 
of $\bigl(\cM,\tau_{w\star}\bigr)$. 
the continuity of the injection $f(u):=\Seq{f_\ell(u)}_{\ell\in \N}$ 
into $[-1,1]^{\N}$ implies that $f(\Bbb{B}_j)$ is compact (and thus 
Borel measurable) in $[-1,1]^{\N}$. 
As a result, $\Bbb{B}_j=f^{-1}\bigl(f(\Bbb{B}_j)\bigr) 
\in \cB_f$. This implies $L^1\in \cB_f$.

\medskip

To simplify the notation, we will write 
$L^2_w([0,T]\times M\times L)$ instead of \eqref{eq:qp-spaces-tmp1}. 
We will also use the local space $L^2_{\operatorname{loc},w}
([0,T]\times M\times \R)$, which is defined 
by $z\in L^2_{\operatorname{loc},w}([0,T]\times M\times \R)$ 
if and only if $z\in L^2_w((0,T)\times M\times L)$ 
for all $L\subset\subset \R$. The local space is 
also quasi-Polish. Local spaces are 
systematically used in, e.g., 
\cite{Brzezniak:2013aa,Brzezniak:2011aa}.

In this paper, we are going to make crucial use 
of the (topological) space
\begin{equation}\label{eq:L2tL2x-weak}
	L^2_t\bigl(L^2_w(M\times L)\bigr)
	=L^2\bigl(0,T;L^2_w(M\times L)\bigr), 
	\quad L\subset \subset \R.
\end{equation}
To define this space consider 
the classical Bochner space of equivalence 
classes of measurable functions $z:[0,T]\to L^2(M\times L)$ 
for which $\norm{z}_{L^2(M\times L)}\in L^2(0,T)$, 
denoted by $L^2_t(L^2(M\times L))
=L^2(0,T;L^2(M\times L))$. We equip this 
space with the locally 
convex topology generated by the seminorms
$$
L^2_t(L^2(M\times L))\ni 
z\mapsto \left(\, \int_0^T \abs{\, \iint_{M\times L}
z(t,\mx,\lambda)\phi(\mx,\lambda) \,d \mx\, d\lambda}^2 
\,d t\right)^{\frac{1}{2}},
\quad \phi \in L^2(M\times L).
$$ 
We denote the resulting topological 
space by $L^2_t\bigl(L^2_w(M\times L)\bigr)$, 
cf.~\eqref{eq:L2tL2x-weak}. Arguing as for \eqref{eq:qp-spaces-tmp1}, 
we can verify that $L^2_t\bigl(L^2_w(M\times L)\bigr)$ 
is a quasi-Polish space. We will also utilise the local variant 
$L^2_t\bigl(L^2_{\operatorname{loc},w}(M\times \R)\bigr)$. 
Note carefully that the topology of 
$L^2_t\bigl(L^2_{\operatorname{loc},w}(M\times \R)\bigr)$ 
is strong in time $t$ and weak in $(\mx,\lambda)$.
Finally, since many of the upcoming calculations 
take place in ``local charts", we will also use global (in $\mx$) 
versions of the spaces appearing in this subsection, for example 
$L^2_t\bigl(L^2_{\operatorname{loc},w}(\R^d\times \R)\bigr)$.

\medskip

Polish spaces are quasi-Polish, and (countable) 
products of quasi-Polish spaces are quasi-Polish.
Herein, we will work with a path space of the 
form $\cX=\cX_1\times\ldots 
\times \cX_N$, where each factor space 
$\cX_i$, $i=1,\ldots,N$, is 
either Polish or quasi-Polish. The space $\cX$ is 
equipped with the product topology. Note 
that for a quasi-Polish space $\cX$ the Borel 
$\sigma$-algebra $\cB(\cX)$ for the 
product topology is likely to differ from 
the product of the individual Borel $\sigma$-algebras 
$\cB(\cX_i)$, $i=1,\ldots,N$ (although they coincide 
if $\cX$ is Polish). Indeed, we have $\cB_f\subset 
\bigotimes_{i=1}^N \cB(\cX_i)\subset \cB(\cX)$ (where 
$f$ refers to the separating sequence of $\cX$).

\medskip

Finally, we recall the Lusin-Suslin theorem 
stating that for a continuous injection 
$F:\cX\to \cY$ between Polish spaces, 
one knows that $F(\cX) \subset \cY$ is Borel.
A variant of this result for quasi-Polish 
spaces $\cX$ can be found in 
\cite[Corollary A.2]{Ondrejat:2010aa} 
and \cite[Proposition C.2]{Brzezniak:2013ab}: 
If $\cY$ is a Polish space and $F:\cY\to \cX$ 
is a continuous injection, then $F(B)$
is a Borel set whenever $B$ is Borel in $\cY$. 

\subsection{Harmonic analysis}
Let us recall the notion of Fourier 
multiplier operator and the Marcinkiewicz 
multiplier theorem. It is possible to use 
the H\"ormander-Mikhlin theorem as well, but 
the notation is slightly simpler with 
the Marcinkiewicz theorem.

\begin{definition}[multiplier operator \cite{Stein:1970pr}]
\label{multiplier} 
A multiplier operator $\cA_\psi:L^2(\R^d)\to L^2(\R^d)$ 
associated to a function $\psi\in C_b(\R^d)$ 
is a mapping given by
$$
\cA_\psi(u)=\F^{-1}
\bigl(\psi \F(u)\bigr),
$$
where $\F(u)(\mxi)=\hat{u}(\mxi)
=\int_{\R^d}e^{-2\pi i \mx\cdot
\mxi}\, u(x)\, dx$ is the Fourier transform of $u$, 
while $\F^{-1}$ is the inverse Fourier transform. 
If the multiplier operator $\cA_\psi$ satisfies
$$
\norm{\cA_\psi (u)}_{L^p(\R^d)} 
\leq C \norm{u}_{L^p(\R^d)}, 
\qquad u\in L^2(\R^d)\cap L^p(\R^d), 
\quad \ p>1,
$$ 
where $C$ is a positive constant, then
the function $\psi$ is called 
an $L^p(\R^d)$-multiplier.
\end{definition}

The next result is taken from 
\cite[Corollary 6.2.5]{Grafakos:2008}.
\begin{theorem}[Marcinkiewicz multipliers]\label{m1} 
Suppose $\psi\in C^{d}\bigl(\R^d\setminus\bigcup_{j=1}^d
\{\xi_j= 0\}\bigr)$ is a bounded function such that 
for some constant $C>0$,
\begin{equation*}%\label{c-mar}
	\abs{\mxi^{\malpha} 
	\partial^{\malpha}\psi(\mxi)}\leq C,
	\quad\mxi\in \R^d \backslash 
	\bigcup_{j=1}^d\{\xi_j= 0\},
\end{equation*}
for every multi-index 
$\malpha=(\alpha_1,\dots,\alpha_d) \in \N_0^d$ 
such that $|\malpha|=\alpha_1+\dots+\alpha_d\leq d$. 
Then $\psi$ is an $L^p$-multiplier for 
any $p\in (1,\infty)$, and the operator 
norm of $\mathcal{A}_\psi$ equals $C_{d,p} \cdot C$, 
where $C_{d,p}$ depends only on $p$ and $d$. 
In particular, for any $\psi\in C^d(\bS^{d-1})$, 
the function $\psi\left(\mxi/\abs{\mxi}\right)$ satisfies 
the conditions of the theorem.
\end{theorem}

We refer to \cite[Chapter 5]{Stein:1970pr} for 
background material on Riesz potentials.

\begin{theorem}[Riesz potentials]\label{riesz} 
The Riesz potential, denoted by $\cT_{-s}$, $s>0$, 
is the Fourier multiplier operator 
with the symbol $\frac{1}{\abs{\mxi}^s}$, $\mxi \in \R^d$, 
which is a continuous mapping 
$L^p(\R^d)\to W_{\loc}^{s,p}(\R^d)$, 
$p\in (1,\infty)$. In particular, for 
any $u_n \tonweak 0$ in $L^p(\R^d)$, we 
have $\cT_{-s}(u_n) \ton 0$ 
strongly in $L^p_{\loc}(\R^d)$ 
along a subsequence.
\end{theorem}

\subsection{Differential geometry}\label{subsec:geometry}
We denote by $\X(M)\equiv\mathcal{T}^1_0(M)$ 
the set of vector fields over $M$, 
by $\Omega^1(M)\equiv\mathcal{T}_1^0(M)$ the 
set of one-forms over $M$, and by 
$\mathcal{T}^r_s(M)$ the set of $(r,s)$-tensors 
over $M$ (whose regularity is to 
be specified individually).

If $X\in \mathfrak{X}(M)$ is a $C^1$ vector field on
$M$ with local representation $X=X^i\frac{\pa}{\pa x^i}$, 
then its divergence $\Div X\in C(M)$ 
is given locally by
\begin{equation*}%\label{divx}
	\Div X = \frac{\pa X^k}{\pa x^k}
	+\Gamma^j_{kj}X^k,
\end{equation*}
where $\Gamma^i_{kj}$ denote the Christoffel symbols 
of $g$ (Einstein's summation convention is used here and below). 
The same expression holds for a distributional vector 
field $X$, and similarly for the formulae given below, which  
we formulate in the smooth case with the 
understanding that they carry over 
by continuous extension also to the 
distributional setting \cite{GKOS,Mar}. 

Using abstract index notation \cite{Lee:2003aa} 
for a vector field $X^a$,
$$
\Div X = \nabla_a X^a.
$$
If $\omega\in \Omega^1(M)$ is a $C^1$ one-form 
that is locally given by $\omega=\omega_i\, dx^i$, then its 
divergence is defined as the metric contraction 
of its covariant differential $\nabla \omega\in \mathcal{T}^0_2(M)$, 
so that $\Div \omega$ becomes $\nabla^b\omega_b$ 
in the abstract index notation, or
\begin{equation*}%\label{divom}
	\Div \omega=g^{ij} \pa_i\omega_j
	-\Gamma^k_{il} g^{il}\omega_k.
\end{equation*} 
in local coordinates (where $g^{ij}$ are the components 
of the inverse metric). If $T={T^a}_b\in \tensoroo(M)$, then 
$\Div T=\nabla_a  {T^a}_b\in \Omega^1(M)$. Locally, 
if $T=T^k_i \frac{\pa}{\pa x^k}\otimes dx^i$, then 
$\Div T=(\Div T)_i \, dx^i$, where
\begin{equation*}%\label{div11}
	(\Div T)_i=\pa_j T^j_i+\Gamma^j_{jl} T^l_i
	-\Gamma^l_{ji} T^j_l.
\end{equation*}
If $T\in \tensoroo(M)$ is $C^2$, the explicit 
form of $\Div(\Div(T))\in C(M)$ in terms of 
local coordinates can be found in \cite{Graf:2017aa} 
(it will not be needed here).

The metric $g$ induces scalar products 
on any of the tensor spaces $(T_{\mx}M)^r_s$, 
which in the abstract index notation are given by
$$
\innb{S,T}=S^{a_1\dots a_r}{}_{b_1\dots b_s} 
T_{a_1\dots a_r}{}^{b_1\dots b_s}.
$$
We will denote the corresponding norms 
uniformly by $\|\,\|_g$, irrespective 
of $r$ and $s$. Furthermore, for 
any tensor field $T\in \mathcal{T}^r_s(M)$, we set
\begin{equation*}%\label{eq:tensor_linfty}
	\norm{T}_{L^\infty(M)}:=\sup_{\mx \in M} 
	\norm{T(\mx)}_g,
\end{equation*}
and
\begin{equation*}%\label{eq:tensor_lp}
	\norm{T}_{L^p(M)} := \left(\,\int_M 
	\norm{T(\mx)}^p_g\, dV(\mx) \right)^{\frac{1}{p}},
	\qquad p\in [1,\infty),
\end{equation*}
where $dV$ is the oriented 
Riemannian volume measure, which, in any 
chart of the oriented atlas, is given by 
$dV=\sqrt{\det\, (g_{ij})} 
\,dx^1\wedge \dots \wedge dx^n$.

\medskip

We also need to fix the space where 
we will look for the Galerkin approximations. 
To this end, let us recall some fundamentals 
on Sobolev spaces on manifolds that we 
shall need in the sequel. As above, let $(M,g)$ 
be a compact oriented Riemannian manifold 
of dimension $d$, and set
$$
\Delta u:= \ddiv \grad(u), 
\quad 
u\in \mathcal{D}'(M).
$$ 
Defining $\Lambda^s:=(I-{\Delta})^{s/2}$, $s>0$, we 
have $(\Lambda^s)^{-1}=\Lambda^{-s}$ 
and by \cite[Theorem 8.5]{CP-PDE-1982}, 
$$
\Lambda^s: H^r(M)\to H^{r-s}(M), 
\quad r\in \R,
$$ 
is a linear isomorphism with 
inverse $\Lambda^{-s}$.

Since $\Lambda^{-s}:L^2(M)\to H^s(M)$ is an 
isomorphism, we may introduce on $H^s(M)$ 
an equivalent scalar product given by 
\begin{equation*}%\label{SP}
	\inn{u,v}_s :=\langle \Lambda^s u, 
	\Lambda^s v \rangle_{L^2}.
\end{equation*}
As $\Lambda^{-1}: L^2(M)\to H^1(M)
\hookrightarrow L^2(M)$ is compact (as well 
as positive and self-adjoint), the spectral 
theorem implies the existence of a countable 
orthonormal basis $\{e_k\}_{k\in \N}$ of $L^2(M)$ 
consisting of eigenfunctions of $\Lambda^{-1}$, and 
we denote the eigenvalue of $e_k$ by 
$\lambda_k^{-1} \in \R$. 
Then $\Lambda e_k=\lambda_k e_k$ and, due 
to the ellipticity of $\Lambda$, we obtain 
$e_k \in C^\infty(M)$. Furthermore, 
$\lambda_k\to \infty$ as $k\to\infty$. 
With respect to $\langle\cdot,\cdot \rangle_s$, 
$\{e_k\}_{k\in \N}$ is an orthogonal system:
\begin{align*}
	\inn{e_k,e_l}_s 
	=\inn{\Lambda^s e_k,\Lambda^s e_l}_{L^2}
	=\inn{\Lambda^{2s} e_k,e_l}_{L^2}
	=\inn{\lambda_k^{2s} e_k,e_l}_{L^2}
	=\lambda_k^{2s}\delta_{kl}, 
	\quad k,l\in \N,
\end{align*} 
with corresponding orthonormal 
system $\left\{e^{(s)}_k\right\}_{k\in \N} 
:=\left\{\lambda_k^{-s}e_k\right\}_{k\in\N}$. 
Moreover,
\begin{equation*}%\label{Hs-norm}
	\begin{split}
		& u\in H^s(M) \, 
		\Longleftrightarrow \Lambda^s(M) u\in L^2(M) 
		\\ & \quad \Longleftrightarrow 
		\sum\limits_{k\in \N} 
		\left| \left\langle \Lambda^s u,e_k 
		\right \rangle_{L^2} \right|^2
		=\sum\limits_{k\in \N} 
		\left|\langle u,e_k \rangle_{L^2}\right|^2
		\lambda_k^{2s}<\infty.
	\end{split}
\end{equation*}
A convenient consequence is that projections 
onto $\operatorname{span}\left\{e_k\right\}_{k=1}^n$ do 
not depend on the Sobolev-index:
\begin{equation*}%\label{eq:projection_independent_of_s}
	\sum_{k=1}^n \left\langle u,
	e_k^{(s)} \right \rangle_{s} e^{(s)}_k
	=\sum_{k=1}^n \left\langle u,
	e_k \right \rangle_{L^2} e_k.
\end{equation*}

\section{A tool for representing weak limits}
\label{sec:H-measure}

We are going to introduce a variant of the $H$-measure adapted 
to kinetic SPDEs, see \cite{Tartar:1990mq} and 
\cite{AEL,AM-1,Gerard:1991kl,LM2,LM5,MM,Pan,rind} for 
some background material. First, however, we will gather 
some notations and present a version of the first 
commutation lemma \cite{Tartar:1990mq}.

We denote by $\chi_K$ the characteristic function of 
a measurable set $K$. We write  $L^p_c(\R^d)$ for 
the collection of compactly supported $L^p$ functions, whereas 
$L^\infty_{\omega,t,\mx}$ is short-hand for 
$L^\infty(\Omega\times (0,T)\times\R^d)$, and similarly 
for other spaces. By $\bar{z}$, we denote the complex conjugate of $z$. 
For subsets $X$ and $Y$ of the Euclidean spaces $\R^d$ 
and $\R^k$, respectively, we use the Bochner space 
$L^2\bigl(X;C_0(Y)\bigr)$. By \cite[p.~695]{Ewards:Book65}, 
the dual is $L^2_{w\star}(X;\cM(Y))$, where $\cM(Y)$ is the 
space of Radon measure on $Y$, and ``$w\star$" 
refers to weak-$\star$ measurable mappings. 

In what follows, we denote by $\F_{t,\mx}$, $\F_t$ 
and $\F_{\mx}$ the Fourier transforms 
with respect to $(t,\mx)$, $t$, and $\mx$, 
respectively, although $\F_t$ conflicts with the notation used 
elsewhere in the paper for the stochastic filtration. 
The strong-weak space $L^2_t\bigl(L^2_w(\R^d)\bigr)$ used 
below is defined as in \eqref{eq:L2tL2x-weak} 
with $M\times L$ replaced by $\R^d$.

The following is a version of Tartar's first commutation 
lemma (see \cite[Lemma 1.7]{Tartar:1990mq} and \cite{AMM}).

\begin{lemma}[first commutation lemma]\label{com-lemma}
Let $\Seq{u_n}_{n\in \N}$ be a sequence 
of functions $u_n:\Omega\times (0,T)\times \R^d\to \R$ satisfying 
\begin{equation}\label{convergence}
	u_n \ton 0 \quad 
	\text{in $L^2_t\bigl(L^2_w(\R^d)\bigr)$, a.s.},
	\quad 
	\norm{u_n}_{L^\infty_{\omega,t,\mx}}\lesssim 1,
\end{equation}
that is, for every $\varphi \in L^2(\R^d)$,
$$
\norm{\int_{\R^d} u_n(\omega,\cdot,\mx)\varphi(\mx) 
\,d\mx}_{L^2([0,T])} \ton 0 \quad 
\text{for $\prob$-a.e.~$\omega\in \Omega$},
$$
specifically, the convergence is strong in $t$ and weak 
in $\mx$, see \eqref{eq:L2tL2x-weak}.
 
Let $\psi \in C^d(\R^{d})$ be such that
\begin{equation}\label{marcin}
	\forall \malpha \in \N^{d}_0, 
	\, \abs{\malpha} \leq d \implies
	\abs{\pa^{\malpha} \psi(\mxi)} 
	\lesssim \abs{\mxi}^{-\abs{\malpha}}.
\end{equation} 
Then, for every $\varphi \in L^p_c(\R^d)$, 
$p>1$, and $K\subset\subset \R^d$,
\begin{equation}\label{commutation}
	\lim\limits_{n\to \infty}
	\int_0^T \norm{\varphi
	\cA_{\psi}\bigl(\chi_K u_n(t)\bigr)
	-\cA_{\psi}\bigl(\varphi\chi_K u_n(t)\bigr)}_{L^2(\R^d)}^2
	\,dt=0, \quad \text{a.s.},
\end{equation} 
where $\cA_{\psi}: L^p(\R^d) \to L^p(\R^d)$ is the
multiplier operator with respect to the variable $\mx$, 
see Definition \ref{multiplier}.
\end{lemma}

\begin{proof}
By an approximation argument, we may assume 
that $\varphi \in C^1_c(\R^d)$. Indeed, any $\varphi \in L^p_c(\R^d)$ 
can be approximated by a function $\bar \varphi \in C^1_c(\R^d)$ (in 
the $L^p(\R^d)$--norm). Since $\psi$ satisfies the conditions 
of the Marcinkiewicz multiplier Theorem \ref{m1} and $u_n$ 
is uniformly bounded \eqref{convergence}, 
replacing $\varphi$ by $\bar \varphi$ in \eqref{commutation} 
induces an error that is small uniformly with respect to $n$. 
Therefore, in the sequel, we assume that $\varphi\in C^1_c(\R^d)$.  

By Plancherel's identity,
\begin{equation}\label{step11}
	\begin{split}
		& I:=\lim\limits_{n\to \infty}
		\norm{\varphi
		\cA_{\psi}\bigl(\chi_K u_n \bigr)
		-\cA_{\psi}\bigl(\varphi 
		\chi_K u_n \bigr)}^2_{L^2((0,T)\times \R^d)}
		\\ & \quad 
		=\lim\limits_{n\to \infty}
		\norm{\F_{\mx}(\varphi)\star \bigl(
		\psi\F_{\mx}\left(\chi_K u_n\right)\bigr)
		-\psi\F_{\mx}\left(\varphi 
		\chi_K u_n\right)}^2_{L^2([0,T]\times\R^d_{\mxi})}
		\\ & \quad
		= \lim\limits_{n\to \infty}
		\Biggl(\norm{\F_{\mx}(\varphi)
		\star \bigl({\psi}
		\F_{\mx}\left(\chi_K u_n\right)\bigr)-\psi 
		\F_{\mx}\left(\varphi \chi_K u_n\right)}^2_{L^2([0,T] \times B_R)}
		\\ & \quad \qquad\qquad\quad 
		+\norm{\F_{\mx}(\varphi)
		\star \bigl(\psi\F_{\mx}\left(\chi_K u_n \right)\bigr)
		-\psi\F_{\mx}\left(\varphi \chi_K u_n(t,\cdot)
		\right)}^2_{L^2([0,T]\times B_R^c)} \Biggr)
		\\ & \quad = \lim\limits_{n\to \infty}
		\Bigl( I^{(1)}_{R,n}+I^{(2)}_{R,n}\Bigr),
	\end{split}
\end{equation} 
where $B_R\subset\R^d$ is the ball centered 
at the origin with radius $R$ and $B_R^c$ is its complement. 
In this proof, the symbol $\mxi$ denotes the Fourier 
variable corresponding to $\mx$ ($t$ is not included).
According to the assumption \eqref{convergence} 
and the definition of the Fourier transform, we have
\begin{equation}\label{B-M}
	\lim\limits_{n\to \infty} I^{(1)}_{R,n}=0, 
	\quad \text{a.s., for each fixed $R$.}
\end{equation}
To justify this claim, notice that
\begin{equation}\label{comm-1}
	\begin{split}
		I^{(1,2)}_{R,n} & :=\norm{\psi
		\F_{\mx}\bigl(\varphi\chi_K 
		u_n\bigr)}^2_{L^2([0,T]\times B_R)}
		\\ & 
		=\int_{B_R}\underbrace{\int_{0}^T
		\abs{\psi(\mxi)\int_{\R^d} e^{2\pi i \mx\cdot \mxi}
		\varphi \chi_K u_n(t,\mx)
		\, d\mx }^2\, dt}_{=:\overline{I}^{(1,2)}_n(\mxi)}
		\, d\mxi \ton 0, \quad \text{a.s.},
	\end{split}
\end{equation} 
Indeed, $\overline{I}^{(1,2)}_n(\mxi)\ton 0$ for every 
$\mxi \in \R^{d}$ (and thus in $B_R$) due to the 
strong-weak convergence assumption \eqref{convergence}. 
Besides, $\overline{I}^{(1,2)}_n(\mxi)$ is bounded by a 
constant $C$ depending on the (uniform) bounds of $u_n$, $\psi$ 
and $\varphi$, as well as the set $K$. 
Thus, by Lebesgue's dominated convergence 
theorem, the claim \eqref{comm-1} follows. 
Similarly, using also Young's inequality for convolutions,
\begin{equation*}%\label{comm-2}
	\begin{split}
		I^{(1,1)}_{R,n} &:=\norm{\F_{\mx}(\varphi)\star
		\bigl(\psi
		\F_{\mx}\left(\chi_K u_n\right)
		\bigr)}_{L^2([0,T]\times B_R)}^2
		\\ & 
		\leq \norm{\F_{\mx}(\varphi)}_{L^1(B_R)}
		\norm{\psi \F_{\mx}
		\left(\chi_K u_n\right)}_{L^2([0,T]\times B_R)}^2
		\ton 0, \quad \text{a.s.}
	\end{split}
\end{equation*}
As $I^{(1)}_{R,n}\leq I^{(1,1)}_{R,n}+I^{(1,2)}_{R,n}$ by 
the triangle inequality, we conclude that \eqref{B-M} holds.

By \eqref{step11} and \eqref{B-M}, 
$I=\lim\limits_{n\to \infty}I^{(2)}_{R,n}$, for 
any $R>0$. Using \eqref{marcin} and arguing as in 
\cite[Lemma 1.7]{Tartar:1990mq}, a.s.,
$$
I^{(2)}_{R,n}\lesssim 1/R 
\overset{R\uparrow\infty}{\longrightarrow} 
0, \quad \text{uniformly in $n$}.
$$ 
It is in this last step that the assumption $\varphi \in C^1_c(\R^d)$ 
is used. This concludes the proof of \eqref{commutation}.
\end{proof}

The following theorem is the main 
result of this section.

\begin{theorem}[a variant of the $H$-measure]
\label{thm:H-measure}
Let $\Seq{u_n}_{n\in\N}=\Seq{u_n(\omega,t,\mx,\lambda)}_{n\in\N}$ 
be a sequence of functions $u_n:\Omega\times 
(0,T)\times \R^{d}\times \R^m\to \R$ satisfying 
\begin{equation}\label{convergence-incl-lambda}
	u_n \ton 0 \quad 
	\text{in $L^2_t\bigl(L^2_w(\R^d\times \R^m)\bigr)$, a.s.},
	\quad
	\norm{u_n}_{L^\infty_{\omega,t,\mx}}\lesssim 1.
\end{equation}
Moreover, for some $K\subset\subset \R^d$, 
$\supp\bigl(u_n(\omega,t,\cdot,\lambda)\bigr)\subset K$,
uniformly in $n$ and $\omega,t,\lambda $.
Consider four functions $\phi\in L^1(\R)$, 
$\phi_1 \in L^r(\R^d\times\R^m)$ (with $r>2$), 
$\phi_2 \in C_c(\R^d\times \R^m)$, and $\psi\in C^d(\bS^d)$, 
where $\bS^d$ denotes the unit sphere in $\R^{d+1}$. 
There exists a subsequence of $\Seq{u_n}$ (not relabeled) 
and a deterministic, time-independent functional
$$
\mu=\mu(\mp,\mq,\mx,\mxi)
\in \left( L^1_{\xi_0}(\R)\times 
L^2_{\mp,\mq,\mxi'}(\R^{2m+d})\times  
C_{\mxi}(\bS^d)) \right)',
$$
where $\mxi=(\xi_0,\dots,\xi_d)
=(\xi_0,\mxi')\in \bS^d$, such that
\begin{align}\label{rev11}
	\begin{split}
		& \int_{\R^{2m}}   
		\Bigl\langle \mu(\mp,\mq,\cdot,\cdot),
		\phi \otimes\phi_1(\cdot,\mp) \overline{\phi_2(\cdot,\mq)} 
		\otimes \psi \Bigr\rangle 
		\,d\mp\, d\mq 
		\\ & \quad
		=\lim\limits_{n\to\infty}
		E \Biggl[ \,\,\,  \int\limits_{\R^{2m+d+1}} 
		\psi\left(\frac{\mxi}{\abs{\mxi}}\right) 
		\, \phi(\xi_0) 
		\\ & \quad \qquad \qquad \qquad \times
		\F_t\Bigl( \F_{\mx}\bigl(\phi_1(\cdot, \mp)
		u_n(\cdot,\cdot,\cdot,\mp)\bigr)
		\overline{ \F_{\mx}\bigl(\phi_2(\cdot,\mq) 
		u_n(\cdot,\cdot,\cdot,\mq)
		\bigr)}\Bigr)(\mxi)\, d\mxi \, d\mp \, 
		d\mq \,\Biggr].
	\end{split}	
\end{align}
We call $\mu$ the generalised 
$H$-measure generated by $\Seq{u_n}$.
\end{theorem}

\begin{remark}
Although the variables $\mp$, $\mq$ take part in the duality 
bracket involving the functional $\mu$, it makes sense to 
additionally integrate over $\mp$, $\mq$ 
because, as we shall see in Corollary \ref{cor:extension-1}, we have 
that $\action{\mu}{\phi}=f(\mp,\mq,\mx,\mxi)\, d\nu(\mx,\mxi)$, 
for an integrable function $f$ and a Radon measure 
$\nu \in L_{\operatorname{loc}}^1(\R^d;\cM(\bS^d))$ 
that is independent of $\mp$ and $\mq$.
\end{remark}

\begin{remark}\label{rem:motivate}
It is a priori unclear whether the assumption \eqref{convergence} 
can be fulfilled. Indeed, \eqref{convergence} asks for 
strong convergence in the temporal variable $t$ 
and a.e.~convergence in the probability variable $\omega$. 
In applications, the  natural a priori estimates imply the 
uniform boundedness  of $u_n$ in $L^2_{\omega,t,\mx,\lambda}$.  
Consequently, we may assume that $u_n\weak u$ in 
$L^2_{\omega,t,\mx,\lambda}$. Upon replacing $u_n$ 
by $u_n-u$, we may as well assume that  $u_n\weak 0$ 
in $L^2_{\omega,t,\mx,\lambda}$.  Unfortunately, no 
$H$-measure theorem applies to all the variables $(\omega,t,x,\lambda)$, 
which jointly belong to a general measure space. 
One key contribution of this paper is to improve this 
weak convergence in all the variables to the mixed 
strong-weak convergence in \eqref{convergence}, 
which will allow us to adapt deterministic $H$-measure 
techniques. We refer to Section \ref{sec:singular-limit} 
for details.
\end{remark}

\begin{remark}\label{rem-all}
In \eqref{rev11}, the test functions $\phi_1$ and $\phi_2$ 
may depend on $\mp,\mq$ and $\mx$. However, since we aim to prove 
that $\mu\equiv 0$, we will eventually 
specify $\phi_1=\phi_2=\varphi(\mx)\rho(\mp)$ and $\psi\equiv 1$. 
Thus, by Plancherel's theorem,
\begin{align*}
	0 & =\lim\limits_{n\to\infty}
	E\left[ \,\,\, \int\limits_{\R^{2m+d+1}} 
	\phi(\xi_0) \, \F_t\left( 
	\F_{\mx}\bigl(\varphi(\cdot)\rho(\mp)
	u_n(\cdot,\cdot,\cdot,\mp)\bigr)
	\,\overline{\F_{\mx} \big(\varphi(\cdot)\rho(\mq) 
	u_{n}(\cdot,\cdot,\cdot,\mq)\big)}\right)\!(\mxi) 
	\, d\mxi \, d\mp\, d\mq\,\right]
	\\ & =\lim\limits_{n\to\infty}
	E \left[\, \, \, \int\limits_{\R^{d+1}} 
	\F_t\left(\F_t^{-1}\bigl(\overline{\phi}\bigr)\right)(\xi_0)
	\F_t\left(\abs{\, \int_{\R^m}\F_{\mx}\bigl(\varphi(\cdot)\rho(\mp)
	u_{n}(\cdot,\cdot,\cdot,\mp)\bigr)(t,\mxi')\, d\mp}^2
	\, \right)\, d\mxi \,\right]
	\\ & = \lim\limits_{n\to\infty}
	E \left[\, \, \,  \int\limits_{0}^T 
	\overline{\F_t^{-1}(\phi)(t)} \, 
	\norm{\, \varphi(\cdot)\int_{\R^m}\rho(\mp)
	u_{n}(\cdot,t,\cdot,\mp)
	\, d\mp}^2_{L^2(\R^{d})}\, dt\, \right].
\end{align*}
By the arbitrariness of $\phi$ and $\varphi$, this yields the 
$L^2(\Omega\times (0,T)\times K)$ convergence 
of $\int_{\R^m}\rho(\mp)u_n(\omega,t,\mx,\mp)\, d\mp$, 
for any $K\subset \subset \R^d$, which is the 
velocity averaging result of 
Section \ref{sec:velocity-averaging}.
\end{remark}

\begin{proof}
We may assume that $\phi\in L^1(\R)$ 
is compactly supported (by the density of such 
functions in $L^1(\R)$). We will divide 
the proof into two main steps:

\medskip

\noindent \textit{Step 1 (application of commutation lemma).}

\medskip

By Lemma \ref{com-lemma}, keeping in mind 
the assumption 
$\supp \bigl(u_n(\omega,t,\cdot,\lambda)\bigr)\subset K$,
\begin{equation}\label{step1}
	\begin{split}
		& I:=\lim\limits_{n\to \infty}
		\int\limits_{\Omega}\int\limits_{\R^{2m+d+1}}
		\psi\left(\frac{\mxi}{\abs{\mxi}}\right) 
		\, \phi(\xi_0) \, 
		\F_t\Bigl( \F_{\mx}\big(\phi_1(\cdot,\mp)
		u_n(\omega,\cdot,\cdot,\mp)\big)
		\\ & \qquad \qquad \qquad \qquad \qquad \qquad \times
		\overline{\F_{\mx}\bigl(\phi_2(\cdot,\mq) 
		u_n(\omega,\cdot,\cdot,\mq)\bigr)}
		\Bigr)(\mxi) \, d\mxi \, d\mp \, d\mq \, d\prob(\omega)
		\\ & \qquad = \lim\limits_{n\to \infty}
		\int\limits_{\Omega}
		\int\limits_{\R^{2m+d+1}} 
		\psi\left(\frac{\mxi}{\abs{\mxi}}\right) \, \phi(\xi_0) 
		\F_t\Bigl( \F_{\mx}\big(\phi_2(\cdot,\mq)
		\phi_1(\cdot,\mp)u_n(\omega,\cdot,\cdot,\mp)\bigr)
		\\& \qquad\qquad\qquad\qquad\qquad\qquad 
		\times \overline{\F_{\mx} 
		\bigl( u_n(\omega,\cdot,\cdot,\mq)\bigr)}\Bigr)(\mxi)
		\, d\mxi \, d\mp\, d\mq \, d\prob(\omega)
		=: \lim_{n\to \infty}\innb{\mu_n, 
		\phi \otimes \phi_2\, \phi_1 \otimes \psi}.
	\end{split}
\end{equation}
To verify the equality in \eqref{step1}, 
notice that the Plancherel theorem implies
\begin{equation*}
	\begin{split}
		I & = \lim\limits_{n\to \infty}
		\int\limits_{\Omega}\int\limits_{\R^{2m+d+1}}\phi(\xi_0) \, 
		\F_t\Bigl( \mathcal{A}_\psi\bigl(\phi_1(\cdot,\mp)
		u_n(\omega,\cdot,\cdot,\mp)\bigr)(\mx)
		\\ & \quad\quad \qquad \qquad \qquad 
		\qquad \qquad \qquad \times 
		\phi_2(\mx,\mq)u_n(\omega,\cdot,\mx,\mq)
		\Bigr)(\xi_0)
		\, d\mx \, d\mp \, d\mq 
		\, d\xi_0\, d\prob(\omega),
	\end{split}
\end{equation*} 
where $\mathcal{A}_\psi=
\mathcal{A}_{\psi\left(\frac{(\xi_0,\cdot)}{\abs{(\xi_0,\cdot)}}\right)}$ 
is an $L^p(\R^d_{\mx})$-multiplier operator, so that in the 
Fourier variable $\mxi=(\xi_0,\mxi')$ 
the temporal part $\xi_0$ is kept fixed. 
Adding and subtracting equal terms gives
\begin{align*}
	& \mathcal{A}_\psi\bigl(\phi_1(\cdot,\mp)
	u_n(\omega,\cdot,\cdot,\mp)\bigr)(\mx)
	\phi_2(\mx,\mq)u_n(\omega,\cdot,\mx,\mq)
	\\ & \qquad= \mathcal{A}_\psi\bigl(\phi_2(\cdot,\mq)
	\phi_1(\cdot,\mp)u_n(\omega,\cdot,\cdot,\mp)\bigr)(\mx)
	u_n(\omega,\cdot,\mx,\mq)+\Delta_n,
\end{align*}
where
\begin{align*}
	\Delta_n & =
	\Bigl(\mathcal{A}_\psi\bigl(\phi_1(\cdot,\mp)u_n(\omega,\cdot,\cdot,\mp)
	\bigr)(\mx)\phi_2(\mx,\mq)
	\\ & \qquad -\mathcal{A}_\psi\bigl(\phi_2(\cdot,\mq)
	\phi_1(\cdot,\mp)u_n(\omega,\cdot,\cdot,\mp)\bigr)(\mx)\Bigr)
	u_n(\omega,\cdot,\mx,\mq).
\end{align*}
Consequently, referring to \eqref{step1}, we obtain 
$I = \lim_{n\to \infty}\innb{\mu_n, 
\phi \otimes \phi_2\, \phi_1 \otimes \psi}
+\mathrm{error}$, where the claim is that
\begin{equation}\label{step1-inbetween}
	\mathrm{error} :=\lim\limits_{n\to \infty}
	\int\limits_{\Omega}\int\limits_\R
	\left[\,\,\,\, \underbrace{\int_{\R^{2m+d}}
	\phi(\xi_0) \, \F_t\bigl(\Delta_n\bigr)
	(\omega,\xi_0) \, d\mx \, d\mp \, d\mq}_{=:I_n(\omega,\xi_0)} 
	\, \right] \, d\xi_0\, d\prob(\omega)=0.
\end{equation}
Indeed, we have 
\begin{equation}\label{ptw-cv}
	I_n(\omega,\xi_0)\ton 0
	\quad \text{for a.e.~$(\omega,\xi_0)$},
\end{equation}
and then \eqref{step1-inbetween} follows from Lebesgue's dominated 
convergence theorem. To see that \eqref{ptw-cv} holds (and also 
to supply a dominating function), note that the symbol 
$\psi\left(\frac{(\xi_0,\mxi')}{\abs{(\xi_0,\mxi')}}\right)$
satisfies the assumption \eqref{marcin} of 
the first commutation lemma, for every fixed $\xi_0$. 
Exploiting the Hausdorf--Young 
and Cauchy--Schwarz inequalities,
\begin{align}
	&\abs{I_n(\xi_0,\omega)} 
	\leq \abs{\phi(\xi_0)}
	\int_0^T \Biggl| \, \int_{\R^{m+d}} 
	\Biggl( \mathcal{A}_{\psi}\left(\phi_2(\cdot,\mq)
	\int_{\R^m} \phi_1(\cdot,\mp)u_n(\omega,\cdot,\cdot,\mp)
	\, d\mp\right)(\mx)
	\notag \\ & \qquad\qquad
	-\phi_2(\mx,\mq)\mathcal{A}_{\psi}\left(\int_{\R^m}\phi_1(\cdot,\mp)
	u_n(\omega,\cdot,\cdot,\mp) d\mp \right)(\mx) \Biggr)
	u_n(\omega,t,\mx,\mq)\Biggr| \, d\mx \, d\mq \, dt
	\notag \\ & \qquad \leq \abs{\phi(\xi_0)}
	\int_0^T \Biggl\|
	\mathcal{A}_{\psi}\left(\phi_2(\cdot,\mq)
	\int_{\R^m} \phi_1(\cdot,\mp)u_n(\omega,\cdot,\cdot,\mp)
	\, d\mp\right)(\mx)
	\notag \\ & \qquad\qquad
	-\phi_2(\mx,\mq)\mathcal{A}_{\psi}\left(\int_{\R^m}\phi_1(\cdot,\mp)
	u_n(\omega,\cdot,\cdot,\mp) d\mp \right)(\mx) 
	\Biggr\|_{L^2(\R^{m+d}_{\mx,\mq})}
	\norm{u_n(\omega,t,\mx,\mq)}_{L^2(\R^{m+d}_{\mx,\mq})} \, dt
	\notag \\ & \qquad \lesssim_{T,K} 
	\abs{\phi(\xi_0)}\Biggl(\int_0^T \Biggl\|
	\mathcal{A}_{\psi}\left(\phi_2(\cdot,\mq)
	\int_{\R^m} \phi_1(\cdot,\mp)u_n(\omega,\cdot,\cdot,\mp)
	\, d\mp\right)(\mx)
	\notag\\ & \qquad\qquad\qquad\qquad\qquad
	-\phi_2(\mx,\mq)\mathcal{A}_{\psi}\left(\int_{\R^m}\phi_1(\cdot,\mp)
	u_n(\omega,\cdot,\cdot,\mp) d\mp \right)(\mx) 
	\Biggr\|_{L^2(\R^{m+d}_{\mx,\mq})}^2\, dt\Biggr)^{\frac12}
	\label{eq:conv-rhs},
\end{align}
where we have used the second part of \eqref{convergence-incl-lambda} 
and the compact support assumption to uniformly 
bound $u_n$ in $L^2$. Observe that, by \eqref{convergence-incl-lambda}, 
$(t,\mx)\mapsto \int_{\R^m} \phi_1(\mx,\mp)u_n(\omega,t,\mx,\mp)\, d\mp$ 
satisfies the strong-weak convergence assumption 
\eqref{convergence} of the commutation lemma. Hence, 
by Lemma \ref{com-lemma} and the dominated convergence 
theorem (in $\mq$), the term \eqref{eq:conv-rhs} 
converges to zero as $n\to\infty$.

The previous calculation also provides 
a dominating function for $\abs{I_n(\omega,\xi_0)}$ of the form 
$C\phi(\xi_0)\in L^1(\Omega\times \R)$, 
where the constant C bounds the final term in 
\eqref{eq:conv-rhs}; it depends on the $L^2$ 
norms of $\phi_1$ and $\phi_2$ as 
well as on the $L^\infty$ norms of 
$\psi$ and $u_n$, cf.~\eqref{convergence-incl-lambda}, 
and the support set $K$ of $u_n$.

\medskip

\noindent \textit{Step 2 (compactness and construction 
of $H$-measure functional).}

\medskip

The functionals $\mu_n$, $n\in \N$ are clearly linear 
on $L^2(0,T)\otimes L^2(\R^{2m+d}) \otimes C(\bS^d)$. 
They are also uniformly bounded on the same space. 
Indeed, by the Cauchy-Schwarz inequality and 
Plancherel's theorem,
\begin{align*}
	& \abs{\innb{\mu_n, \phi \otimes \phi_2
	\, \phi_1 \otimes \psi}}
	\\ & \quad \leq \norm{\psi}_{C(\bS^d)}
	\int_{\Omega} \int_{\R^{2m+d}}
	\norm{\phi}_{L^1(\R)}
	\\ & \qquad \quad \times 
	\norm{\F_t \left(\F_{\mx}
	\bigl(\phi_2(\cdot,\mq)\phi_1(\cdot,\mp)
	u_n(\omega,\cdot,\cdot,\mp)\bigr)\,
	\overline{\F_{\mx}
	\bigl(u_n(\omega,\cdot,\cdot,\mq)\bigr)}
	\right)(\cdot,\mxi')}_{L^\infty(\R)} 
	\, d\mxi'\, d\mp\, d\mq \, d\prob(\omega)
	\\ & \quad 
	\leq \norm{\psi}_{C(\bS^d)} \norm{\phi}_{L^1(\R)}
	\int_{\Omega}
	\int_{\R^{2m}}\int_0^T\int\limits_{\R^d}
	\Bigl| \F_{\mx}\bigl(\phi_2(\cdot,\mq)
	\phi_1(\cdot,\mp) u_n(\omega,t,\cdot,\mp)\bigr)(\mxi')
	\\ & \qquad \qquad\qquad\qquad \qquad \qquad
	\qquad\qquad \times 
	\overline{\F_{\mx}\bigl(u_n(\omega,t,\cdot,\mq)\bigr)(\mxi')}
	\Bigr| \, d\mxi' \, dt \, d\mp\, d\mq \, d\prob(\omega)
	\\ & \quad\leq 
	\norm{\psi}_{C(\bS^d)} \norm{{\phi}}_{L^1(\R)} 
	\int_{\Omega} \int_{\R^{2m}}\int_0^T
	\norm{\F_{\mx} \bigl(\phi_2(\cdot,\mq)
	\phi_1(\cdot,\mp)u_n(\omega,t,\cdot,\mp)\bigr)}_{L^2(\R^d)}
	\\ & \qquad \qquad \qquad\qquad\qquad 
	\qquad \qquad \qquad \times 
	\norm{\F_{\mx} \bigl( u_{n}(\omega,t,\cdot,\mq)
	\bigr)}_{L^2(\R^d)} \, dt\, d\mp\, d\mq \, d\prob(\omega)
	\\ & \quad = \norm{\psi}_{C(\bS^d)}\norm{\phi}_{L^1(\R)}
	\int_{\Omega}\int_{\R^{2m}}   
	\int_0^T\norm{\phi_2(\cdot,\mq)\phi_1(\cdot,\mp)
	u_n(\omega,t,\cdot,\mp)}_{L^2(\R^d)} 
	\\ & \qquad \qquad \qquad\qquad\qquad 
	\qquad \qquad \qquad \times
	\norm{u_{n}(\omega,t,\,\cdot\,,\mq)}_{L^2(\R^d)} 
	\, dt\, d\mp\, d\mq \, d\prob(\omega)
	\\ & \quad \leq C_T \norm{\psi}_{C(\bS^d)} 
	\norm{\phi}_{L^1(\R)} 
	\norm{\phi_2 \phi_1}_{L^2(\R^{2m+d})},
\end{align*}
where $C_T$ is a constant depending on $T$ 
and $\norm{u_n}_{L^\infty_{\omega,t,x,\lambda}}\lesssim 1$.

Thus, the sequence $\Seq{\mu_n}\subset  
\bigl(L^1(\R)\otimes L^2(\R^{2m+d})\otimes
C( \bS^d)\bigr)'$ of linear functionals 
is bounded independently of $n$. 
By the Banach-Alaoglu theorem, there exist a subsequence 
of $\Seq{\mu_n}$ (not relabeled) and 
a limit functional $\mu$ such that 
$\mu_n \weakstar \mu$, which---via \eqref{step1}---implies 
that $\mu$ satisfies \eqref{rev11}.
\end{proof}

We will need to show that the limit
functional $\mu$ can be extended to the space 
$\bigl(L^1(\R)\otimes L^2(\R^{2m+d};C_0( \bS^d))\bigr)'$, 
where by $L^1(\R) \otimes L^2(\R^{2m+d};C_0(\bS^d))$ 
we understand the linear space spanned by functions of the 
form $\phi(\xi_0)\varphi(\lambda,\mx,\mxi)$, with
$\phi \in L^1(\R)$, $\varphi \in L^2(\R^{2m+d};C_0(\bS^d))$.  
In other words, for any $\phi \in L^1(\R)$, we have
$$
\innb{\mu,\phi} \in 
\big(L^2(\R^{2m+d};C_0(\bS^d))\big)'
=L^2_{w\star}(\R^{2m+d};\cM(\bS^d)).
$$ 
Unlike the situation with the standard H-measure, here 
we do not know the non-negativity of the 
multilinear functional $\mu$ (so that Schwartz's 
lemma on nonnegative distributions cannot be applied). 
Hence, we need to prove directly that $\mu$ can be 
extended t<<<o a larger space.

\begin{corollary}[extension of $H$-measure functional $\mu$]
\label{cor:extension} 
Suppose the assumptions of Theorem \ref{thm:H-measure} hold. 
The deterministic linear functional $\mu$---defined 
by \eqref{rev11}---can be continuously extended to the 
space $\bigl(L^1(\R)\otimes 
L^2(\R^{2m+d};C_0(\bS^d))\bigr)'$. 
\end{corollary}

\begin{proof}
Any $\varphi \in L^2(\R^{2m+d};C_0( \bS^d))$ 
can be approximated by functions of the form
$$
\sum\limits_{k=1}^N \psi^N_{k}(\mxi)
\chi^N_k(\mx,\mp,\mq), \quad N=1,2,\ldots,
$$ 
where $\psi^N_{k} \in C_0(\bS^d)$, and 
$\chi^N_k(\mx,\mp,\mq)$, $k=1,\ldots,N$, 
are characteristic functions of mutually disjoint sets 
in $\R^{2m+d}$. If we prove that the functional $\mu$ 
is uniformly bounded in 
$\bigl(L^1(\R)\otimes L^2(\R^{2m+d};C_0(\bS^d))\bigr)'$ 
with respect to functions of the 
form $\phi(\xi_0)\sum\limits_{k=1}^N 
\psi^N_{k}(\mxi) \chi^N_k(\mx,\mp,\mq)$, $\phi \in L^1(\R)$, 
the conclusion of the theorem follows.

Using that $(\chi_k^N)^2=\chi_k^N$ and arguing 
along the lines of Remark \ref{rem-all}, we obtain 
\begin{align*}
	&\int_{\R^{2m}}\inn{\mu(\mp,\mq,\cdot,\cdot),
	\phi \otimes \left(\,\sum\limits_{k=1}^N \psi^N_{k}(\mxi)
	\chi^N_k(\mx,\mp,\mq) \right)} \,d\mp\, d\mq 
	\\ & \quad = \lim\limits_{n\to\infty}
	E \Biggl[ \,\,\,\int_{\R^{2m+d+1}} 
	\sum\limits_{k=1}^N \psi_k^N
	\left(\frac{\mxi}{\abs{\mxi}}\right) \phi(\xi_0) 
	\F_t\Bigl( \F_{\mx}\bigl(\chi^N_k(\cdot,\mp,\mq)
	u_n(\cdot,\cdot,\cdot,\mp)\bigr)
	\\&\quad\qquad\qquad
	\qquad\qquad\qquad \,\times
	\overline{\F_{\mx} \bigl(\chi^N_k(\cdot,\mp,\mq)
	u_n(\cdot,\cdot,\cdot,\mq)\bigr)}\Bigr)(\mxi)
	\, d\mxi\, d\mp\, d\mq\Biggr]
	\\ & \quad \leq \lim\limits_{n\to\infty}
	E\Biggl[\norm{\phi}_{L^1(\R)}
	\int_{\R^{2m+d}}\int_0^T\sum\limits_{k=1}^N 
	\norm{\psi_k^N}_{C(\bS^d)}
	\, \Bigl| \F_{\mx}\bigl(\chi^N_k(\cdot,\mp,\mq)
	u_n(\cdot,t,\cdot,\mp)\bigr)(\mxi')
	\\&\quad\qquad\qquad
	\qquad\qquad \qquad \times
	\overline{\F_{\mx}\bigl(\chi^N_k(\cdot,\mp,\mq)
	u_n(\cdot,t,\cdot,\mq)\bigr)(\mxi')}
	\Bigr| \,dt\,d\mxi' \, d\mp\, d\mq\Biggr]
	\\ & \quad \leq \lim\limits_{n\to\infty}
	E \Biggl[ \norm{\phi}_{L^1(\R)} 
	\int\limits_{\R^{2m+d}} \int_0^T  
	\left(\sum\limits_{k=1}^N \norm{\psi_k^N}_{C(\bS^d)}^2
	\abs{\F_{\mx}\bigl(\chi^N_k(\cdot,\mp,\mq)
	u_n(\cdot,t,\cdot,\mp)\bigr)(\mxi')}^2\right)^{1/2}
	\\ & \quad\qquad\qquad\qquad\qquad\qquad
	\times \left(\sum\limits_{k=1}^N
	\abs{\F_{\mx}\bigl(\chi^N_k(\cdot,\mp,\mq)
	u_n(\cdot,t,\cdot,\mq)\bigr)(\mxi')}^2 
	\right)^{1/2}\, dt\, d\mxi' \, d\mp\, d\mq\Biggr],
\end{align*}
where, in the last step, we have used the discrete 
Cauchy-Schwarz inequality. To proceed, we apply the Cauchy-Schwarz 
inequality in the $L^2$-setting and Plancherel's theorem to obtain
\begin{align*}
	&\left|\int_{\R^{2m}}\inn{\mu( 
	\mp,\mq,\cdot,\cdot),\phi \otimes
	\left(\sum\limits_{k=1}^N \psi^N_{k}(\mxi) 
	\chi^N_k(\mx,\mp,\mq)\right)} \,d\mp\, d\mq \right|
	\\ & \quad \leq \lim\limits_{n\to\infty}
	E\left[ \norm{\phi}_{L^1(\R)}  
	\norm{\sum\limits_{j=1}^N
	\norm{\psi_k^N}_{C(\bS^d)} \, 
	\chi^N_k \, u_n}_{L^2((0,T)\times \R^{2m+d})}
	\norm{\sum\limits_{k=1}^N 
	\chi^N_k \, u_n}_{L^2((0,T)\times \R^{2m+d})} \right]
	\\ & \quad \leq C \norm{\phi}_{L^1(\R)}  
	\norm{\sum\limits_{k=1}^N\psi_k^N
	\, \chi^N_k}_{L^2(\R^{2m+d};C(\bS^d))},
\end{align*}
where $C$ depends on $T$ and 
$\norm{u_n}_{L^\infty_{\omega,t,x,\lambda}}
\lesssim 1$. The obtained bound is independent of $N$, 
which concludes the proof of the theorem. 
\end{proof} 

Based on the proof of Corollary \ref{cor:extension}, we 
may characterise the functional $\mu$ more precisely 
as follows:

\begin{corollary}[final characterisation of 
$H$-measure functional $\mu$]\label{cor:extension-1}
Suppose that the assumptions of Theorem \ref{thm:H-measure} 
hold. For each fixed $\phi \in L^1(\R)$, there exists 
a functional $\mu_\phi \in 
\left(L^2(\R^{2m};C_0(\R^d\times \bS^d))\right)'$ such that
\begin{equation}\label{rev11-1}
	\begin{split}
		&\int_{\R^{2m}}   
		\inn{\mu_\phi(\mp,\mq,\cdot,\cdot), 
		\phi_1(\cdot,\mp) \overline{\phi_2(\cdot,\mq)} 
		\otimes \psi} \,d\mp\, d\mq 
		\\ & \qquad 
		= \lim\limits_{n\to\infty}E \Biggl[ \,\,\, 
		\int\limits_{\R^{2m+d+1}}
		\psi\left(\frac{\mxi}{\abs{\mxi}}\right)
		\phi(\xi_0)
		\\ & \qquad \qquad \qquad \qquad\times
		\F_t\Bigl( \F_{\mx}\bigl(\phi_1(\cdot, \mp)
		u_n(\cdot,\cdot,\cdot,\mp)\bigr) 
		\overline{\F_{\mx}\bigl(\phi_2(\cdot,\mq)
		u_n(\cdot,\cdot,\cdot,\mq)\bigr)}
		\Bigr)(\mxi)\, d\mxi \, d\mp \, d\mq
		\, \Biggr].
	\end{split}
\end{equation} 
Moreover, $\mu_\phi$ has the form 
\begin{equation}\label{representation}
	\mu_\phi=f(\mp,\mq,\mx,\mxi)\, d\nu(\mx,\mxi),
\end{equation} 
where $\nu \in L_{\operatorname{loc}}^1(\R^d;\cM(\bS^d))$, 
so $\nu$ is regular with respect 
to $\mx\in \R^d$, and $f\in L^2(\R^{2m};L^1_{\nu}(\R^d\times \bS^d))$, 
i.e., $\int_{\R^{2m}}\abs{\int_{\R^d\times \bS^d}
\abs{f(\mp,\mq,\mx,\mxi)}\, d\nu(\mx,\mxi)}^2 
\, d\mp \, d\mq<\infty$.
\end{corollary}

\begin{proof}
The proof of the first part of the theorem is 
the same as the proof of Corollary \ref{cor:extension}. 
We note that being fixed, the function $\phi$ does not 
influence the estimates leading to the conclusion. 
The proof of \eqref{representation} 
is now the same as the one of \cite[Proposition 12]{LM2}.
\end{proof}

\section{Well-posed pseudo-parabolic SPDE} 
\label{sec:well-posed} 

This section contains a proof of  Theorem \ref{unique_sol_par}, 
divided into three subsections. The existence part of the 
theorem is based on establishing convergence of 
suitably constructed  Galerkin approximations, 
relying on a  Cauchy convergence argument. 
The uniqueness part makes use of $L^2$ 
stability techniques. 

In what follows, we drop the $k$-index, writing $\delta, 
\eps, u$ instead of $\delta_k,\eps_k,u_k$. 
This is justified as Theorem \ref{unique_sol_par} addresses
the question of well-posedness of the 
pseudo-parabolic SPDE for a fixed value 
of $k$ (so that $\delta_k,\eps_k>0$). 
In other words, we wish to solve 
the initial value problem
\begin{align}\label{PP-1'}
	d\left(u-\delta \Delta u\right)
	+\Div {\mathfrak f}(\mx, u)\, dt 
	& =\eps \Delta u dt 
	+\Phi(\mx, u)\, dW_t,\\
	\label{PP-ID'}
	u|_{t=0}&=u_0\in H^2(M), 
	\ \ \mx \in M,
\end{align} 
where we assume that 
\begin{equation}\label{assump-ed}
	\text{$1/2\geq \eps>\eps_0>0$ 
	and $1/2\geq \delta>\delta_0>0$.}
\end{equation} 

\subsection{Approximate solutions}

The approximate solution has the form
\begin{equation*}%\label{appr}
	u_n(t,\mx)=\sum\limits_{k=1}^n 
	\akn(t) e_k(\mx),
\end{equation*} 
where $\{e_k\}$ is the orthonormal 
basis of $L^2(M)$ introduced earlier, 
and the coefficients $\akn=\akn(\omega,t)$, $k=1,\dots,n$, 
are adapted continuous stochastic processes to 
be determined. Note that
\begin{align*}
	\norm{u_n(t)}^2_{H^1(M)} 
	& = \sum\limits_{k=1}^n 
	\lambda_k^2 \left|\akn(t)\right|^2
	\\ \quad \text{and} \quad 
	\Delta e_k & =\left(1-\lambda_k^2\right) e_k 
	\ \ \text{for every $k\in \N$}.
\end{align*}
We insert the approximation into \eqref{PP-1'}, 
multiply the obtained relation by a 
basis function $e_k$ and integrate over $M$, 
eventually ariving at
\begin{equation}\label{alpha-sys}
	\begin{split}
		& d \left( \left(1-\delta+\delta \lambda^2_k\right)
		\akn\right) +\eps \left(\lambda^2_k-1\right)
		\akn\, dt =
		\int_M  \mff(\mx,u_n)\cdot \nabla e_k
		\, dV(\mx) \, dt
		\\ & \qquad \quad
		+\int_M \Phi(\mx,u_n) e_k\, dV(\mx) \, dW_t,
		\quad k=1,\ldots,n.
	\end{split}
\end{equation} 
We supplement the SDE system \eqref{alpha-sys} 
with the initial data 
\begin{equation}\label{id-sde}
	\akn(0,\cdot)=\alpha_{0,k},
\end{equation} 
where $\alpha_{0,k}$ are the coefficients 
in the expansion of the initial data \eqref{PP-ID'}:
$$
u_0(\mx)=\sum\limits_{k=1}^\infty 
\alpha_{0,k} e_k(\mx).
$$ 
To proceed, we denote 
\begin{align*}
	&\mF_{k}^{(n)}(\malpha)
	=\int_M  \mff\Bigl(\mx,\sum\limits_{j=1}^n 
	\ajn e_j(\mx)\Bigr)\cdot 
	\nabla e_k(\mx) \,dV(\mx), 
	\\ & 
	\mPhi_{k}^{(n)}(\malpha)=\int_M 
	\Phi\Bigl(\mx,\sum\limits_{j=1}^n 
	\ajn e_j(\mx)\Bigr) e_k(\mx) \,dV(\mx).
\end{align*}
A straightforward calculation, 
based on ($C_f$--3) and ($C_\Phi$--1), gives
$$
\abs{\mF_{k}^{(n)}(\malpha)}
+\abs{\mPhi_{k}^{(n)}(\malpha)} 
\leq C_n \bigl(1+\left|\malpha\right|\bigr), 
\quad  k=1,\ldots,n,
$$ 
Moreover, the functions $\mF_{k}^{(n)}$, 
$\mPhi_{k}^{(n)}$ are locally Lipschitz continuous. 
According to, e.g., \cite[Theorem 3.1.1]{Liu:2015vb}, 
the SDE system \eqref{alpha-sys}, \eqref{id-sde} admits 
a unique global solution $\malpha=\malpha (\omega, t)$, 
see also the global $n$-independent bound 
\eqref{uniform} below.

Thus, we obtain a sequence 
$\left\{u_n\right\}_{n\in\N}$ of approximate weak 
solutions to \eqref{PP-1'}, \eqref{PP-ID'} 
in the sense that, for any 
$\varphi_n \in V_n= \spann\{e_k\}_{k=1}^n$, 
\begin{align*}
	& d\int_M \left(u_n-\delta \Delta u_n\right)
	\varphi_n(\mx) \,dV(\mx) 
	+\int_M \Div\, {\mathfrak f}(\mx, u_n) 
	\varphi_n(\mx) \,dV(\mx) \, dt
	\\ & \qquad 
	=\eps \int_M \Delta u_n(t,\mx) 
	\varphi_n(\mx) \,dV(\mx) dt
	+\int_M \Phi(\mx, u_n)
	\varphi_n(\mx)\,dW_t,
	\\ & u_n|_{t=0}=u^n_0(\mx)
	:=\sum\limits_{k=1}^n
	\alpha_{0,j}e_k(\mx).
\end{align*}

\subsection{Convergence \& existence}
We will prove that $\left\{u_n\right\}$ 
is a Cauchy sequence in $L^2_{\prob}\left(\Omega; 
C([0,T];L^2(M))\right)$. We start by deriving a priori 
bounds in $L^2_{\prob}\left(\Omega;
L^\infty(0,T;H^1(M))\right)$, depending 
on $\delta$ and $\eps$ but not $n$, and also 
in the larger space $L^2_{\prob}\left(\Omega; 
L^\infty(0,T;L^2(M))\right)$, independent 
of $\delta$, $\eps$ and $n$.
 
By the SDEs \eqref{alpha-sys} 
and It\^o's formula,
\begin{align*}
	&\left(1-\delta+\delta \lambda^2_k\right) 
	d\left(\akn\right)^2
	= 2\eps (1-\lambda_k^2) \left(\akn\right)^2\, dt  
	\\ & \qquad 
	+ 2\int_M \mff(\mx,u_n) 
	\nabla \left(\akn e_k\right) \,dV(\mx) \, dt 
	\\ & \qquad \qquad 
	+\frac{1}{1-\delta+\delta \lambda_k^2}
	\left( \int_M \Phi(\mx,u_n) e_k \,dV\right)^2 \,dt 
	\\ & \qquad \qquad \qquad 
	+ 2\int_M \Phi(\mx,u_n)\left(\akn  e_k\right) 
	\,dV(\mx) \, dW_t.
\end{align*}  
Summing over $k=1,\dots, n$, noting that 
$$
\lara{\nabla e_k}{\nabla e_l}_{L^2} 
= -\lara{e_k}{\Delta e_l}_{L^2} 
= -\lara{e_k}{(I - \Lambda^2) e_l}_{L^2} 
= \delta_{kl}(\lambda_k^2-1)
$$
and 
$$\left\|\nabla u_n\right\|^2_{L^2(M)}
=\sum\limits_{k=1}^n
\left(\akn\right)^2\left(\lambda_k^2-1
\right),
$$ we get
\begin{align*}
	& \sum_{k=1}^n
	\left(1-\delta+\delta \lambda^2_k\right)
	\left(\akn(t)\right)^2
	+ 2\eps  \int_0^t \left\|\nabla u_n(t')
	\right\|^2_{L^2(M)}\, dt' 
	\\ & \qquad 
	=\sum_{k=1}^n
	\left(1-\delta+\delta \lambda^2_k\right)
	\left(\akn(0)\right)^2
	+2 \int_0^t \int_M \mff(\mx,u_n) 
	\nabla u_n \, dV(\mx)\, dt'
	\\ & \qquad \quad \quad 
	+\int_0^t \sum\limits_{k=1}^n 
	\frac{1}{1-\delta+\delta \lambda^2_k}
	\left(\int_M \Phi(\mx,u_n) e_k 
	\,dV(\mx)\right)^2 \,dt'
	\\ & \qquad \quad \quad\quad 
	+ 2\int_0^t \int_M u_n\Phi(\mx,u_n)
	\,dV(\mx) \, dW_{t'}.
\end{align*}
By the definitions of the $L^2$ 
and $H^1$ norms of $u_n$,
\begin{equation}\label{eq:energy-tmp1}
	\begin{split}
		&\left \| u_n(t)\right\|_{L^2(M)}^2
		+\delta\left \| u_n(t)\right\|_{H^1(M)}^2
		+2\eps  \int_0^t \left\|\nabla u_n(t')
		\right\|^2_{L^2(M)}\, dt' 
		\\ & \qquad 
		= \left \| u_n(0)\right\|_{L^2(M)}^2
		+\delta\left \| u_n(0)\right\|_{H^1(M)}^2
		+2 \int_0^t \int_M \mff(\mx,u_n) 
		\nabla u_n \, dV(\mx)\, dt'
		\\ & \qquad \quad \quad 
		+\int_0^t \sum\limits_{k=1}^n 
		\frac{1}{1-\delta+\delta \lambda^2_k}
		\left(\int_M \Phi(\mx,u_n) e_k 
		\,dV(\mx)\right)^2 \,dt'
		\\ & \qquad \quad \quad\quad 
		+2\int_0^t \int_M u_n\Phi(\mx,u_n)
		\,dV(\mx) \, dW_{t'}.
	\end{split}
\end{equation}
Denote the third and fourth terms on 
the right-hand side by $I_\mff$ and 
$I_\Phi$, respectively. By Young's product  
inequality and ($C_f$--3),
\begin{equation}\label{eq:If-est-alpha}
	\abs{I_f} \le 
	C_\kappa\int_0^t \left(1+\left\|u_n
	\right\|_{L^2(M)}^2\right)\, dt' 
	+\kappa\int_0^t \left\|\nabla u_n
	\right\|_{L^2(M)}^2\, dt', 
	\quad \kappa\in (0,1).
\end{equation}
By ($C_\Phi$--1) and since $0<\lambda_k 
\tok \infty$, 
$$
\abs{I_\Phi} \le 
C\int_0^t \left(1+\left\|u_n
\right\|_{L^2(M)}^2\right) \, dt'.
$$
Given $u_0\in H^1(M)$, inserting the 
last two bounds into \eqref{eq:energy-tmp1}, we obtain
\begin{equation}\label{eq:energy-tmp2}
	\begin{split}
		\left \| u_n(t)\right\|_{H^1(M)}^2
		&\le C+C\int_0^t \left\|u_n\right\|_{H^1(M)}^2 \, dt'
		\\ & \qquad 
		+2\int_0^t \int_M u_n\Phi(\mx,u_n)
		\,dV(\mx) \, dW_{t'},
	\end{split}
\end{equation}
for some $n$-independent constant $C=C(\delta,\eps)$. 
Taking the expectation $E(\cdot)$ and 
applying the Gronwall inequality, 
the final result becomes 
\begin{equation}\label{eq:energy-tmp3}
	\sup_{t\in [0,T]}
	E\left[\norm{u_n(t)}_{H^1(M)}^2\right] 
	\le C_T,
\end{equation}
where $C_T$ is a constant independent of $n$ (but 
dependent on $\delta$, $\eps$). This proves 
that $u_n$ is uniformly bounded in
$L^\infty\left(0,T;L^2_{\prob}(\Omega;H^1(M))
\right)$. Replacing $H^1$ by $L^2$, we can 
derive a bound that is also independent 
of $\delta$ and $\eps$. Indeed, absorbing the last 
term on the right-side of \eqref{eq:If-est-alpha} (taking 
$\kappa$ small enough) into the $\eps$-term on 
the left-hand side of \eqref{eq:energy-tmp1}, we obtain
\begin{equation}\label{eq:energy-tmp2-new}
	\begin{split}
		\norm{u_n(t)}_{L^2(M)}^2
		&\le C+C\int_0^t \norm{u_n}_{L^2(M)}^2 \, dt'
		\\ & \qquad +2\int_0^t \int_M u_n\Phi(\mx,u_n)
		\,dV(\mx) \, dW_{t'},
	\end{split}
\end{equation}
assuming that $u_0\in L^2(M)$ and $\delta 
\left\| u_0\right\|_{H^1(M)}^2 \lesssim 1$. 
As before, via Gronwall's inequality, this supplies the bound
\begin{equation*}%\label{eq:energy-tmp3-L2}
	 \sup_{t\in [0,T]} 
	 E\left[\norm{u_n(t)}_{L^2(M)}^2\right] 
	 \le C_T,
\end{equation*}
where $C_T$ is a constant independent of $n$, 
$\delta$ and $\eps$. 

It is possible to improve the 
estimate \eqref{eq:energy-tmp3}
by applying a martingale (BDG) inequality. 
Indeed, take $\sup_{t\in[0,\tau]}$ in 
\eqref{eq:energy-tmp2} and then the 
expectation $E(\cdot)$. 
Using \eqref{eq:energy-tmp3} 
in the resulting inequality, we arrive at  
\begin{equation}\label{eq:BDG-est1}
	E \left[\sup_{t\in [0,\tau]} 
	\norm{u_n(t)}_{H^1(M)}^2\right]
	\le C_T \bigl(1+I_m\bigr),
	\quad \tau\in [0,T], 
\end{equation}
where $C_T$ is $n$-independent and 
$$
I_m:=E \left|\, \sup_{t\in[0,\tau]}\int_0^t 
\int_M u_n\Phi(\mx,u_n)\,dV(\mx) \, dW_{t'}\right|.
$$
Making use of the BDG inequality, the Cauchy-Schwarz inequality, 
$\left|\Phi(\mx,u_n)\right|\lesssim 1
+\left|u_n\right|$, cf.~($C_\Phi$--1), and Young's 
product inequality, we obtain
\begin{align*}
	I_m & \le 
	C E\left[ \left(\int_0^{\tau} \left(\int_M u_n
	\Phi(\mx,u_n)\,dV(\mx)\right)^2 
	\, dt'\right)^{\frac12}\right]
	\\ & \le 
	\kappa E \left[\sup_{t\in [0,\tau]} 
	\norm{u_n(t)}_{L^2(M)}^2\right]
	+C_\kappa \int_0^{\tau} E\left\| u_n(t')
	\right\|_{L^2(M)}^2\, dt'+C_{\kappa,T,M}
	\\ & \le 
	\kappa E \left[\sup_{t\in [0,\tau]} 
	\norm{u_n(t)}_{H^1(M)}^2\right]
	+C_\kappa \int_0^{\tau} E \left[\sup_{t\in [0,t']}
	\left\| u_n(t')\right\|_{H^1(M)}^2\right]\, dt'
	+C_{\kappa,T,M},
\end{align*}
for any $\kappa\in (0,1)$, where the first term 
on the right-hand side can be absorbed into the 
term on the left-hand side of \eqref{eq:BDG-est1} 
(for small enough $\kappa$). An application of Gronwall's 
inequality then gives
\begin{equation}\label{eq:energy-tmp4}
	E\left[\sup_{t\in [0,T]}
	\norm{u_n(t)}_{H^1(M)}^2\right]\le C_T.
\end{equation}
This establishes the $L^2_{\prob}
\left(\Omega;L^\infty(0,T;H^1(M))\right)$ bound. We remark that 
the constant $C_T$ depends on $\delta$, $\eps$ (but not $n$). 
Replacing $H^1$ by $L^2$ in \eqref{eq:energy-tmp4}, 
the resulting constant becomes independent 
also of $\delta$, $\eps$. Indeed, applying   
the BDG inequality as before, this time starting 
from \eqref{eq:energy-tmp2-new}, we obtain 
\begin{equation}\label{eq:energy-tmp4-L2}
	E \left[\sup_{t\in [0,T]}
	\norm{u_n(t)}_{L^2(M)}^2\right] \le C_T,
\end{equation}
where $C_T$ is a constant independent of 
$n$, $\delta$ and $\eps$. Since 
$$
\norm{u_n}_{L^2(M)}^2
=\sum\limits_{k=1}^n \left|\akn\right|^2,
$$ 
we conclude that
\begin{equation}\label{uniform}
	E\left[\norm{\akn}_{L^\infty_t([0,T])}^2\right]
	\leq c, \quad k=1,\ldots,n,
\end{equation} 
for a constant $c>0$ that is 
independent of $n$ (and $\delta$, $\eps$). 
As a result, $\akn(t)$ is finite almost surely, 
for any $t\in [0,T]$, cf.~the SDEs \eqref{alpha-sys}.

\medskip

In what follows, we consider two approximations $u_n$ 
and $u_m$ with $m>n$. Denote by 
$\mff_n=\mff(\mx,u_n)$, $\mff_m=\mff(\mx,u_m)$ 
and $\Phi_n=\Phi(\mx,u_n)$, $\Phi_m=\Phi(\mx,u_m)$. 
Given \eqref{alpha-sys}, we can write 
SDEs for the differences $\akn-\akm$, $k\leq n<m$. 
Applying It\^o's formula gives
\begin{align*}
	&\left(1-\delta+\delta \lambda^2_k\right) 
	d\left(\akn-\akm\right)^2
	= 2\eps (1-\lambda_k^2) \left(\akn-\akm\right)^2\, dt  
	\\ & \qquad 
	+ 2\int_M \left(\mff_n-\mff_m\right) 
	\nabla \left( \left(\akn-\akm\right)e_k
	\right) \,dV(\mx) \, dt 
	\\ & \qquad \qquad 
	+\frac{1}{1-\delta+\delta \lambda_k^2}
	\left( \int_M \left(\Phi_n
	-\Phi_m\right) e_k \,dV\right)^2 \,dt 
	\\ & \qquad \qquad \qquad 
	+ 2\int_M \left(\Phi_n-\Phi_m\right)
	\left(\left(\akn-\akm\right)e_k\right)
	\,dV(\mx) \, dW_t.
\end{align*}  
Summing over $k=1,\dots, n$, denoting 
by 
$$
u_m(t,\cdot)\big|_{V_n}
=\sum\limits_{k=1}^n\akm e_k
$$ 
the projection of $u_m(t,\cdot)$ 
onto $V_n$ (recall $m>n$), we obtain
\begin{equation}\label{eq:energy-tmp5}
	\begin{split}
		& \left \| u_n(t,\cdot)-u_m(t,\cdot)
		\big|_{V_n}\right\|_{L^2(M)}^2
		+\delta\left \| u_n(t,\cdot)-u_m(t,\cdot)
		\big|_{V_n}\right\|_{H^1(M)}^2
		\\ & \quad 
		+2\eps  \int_0^t \left\|\nabla \left(
		u_n(t,\cdot)-u_m(t,\cdot)\big|_{V_n}
		\right)\right\|^2_{L^2(M)}\, dt' 
		\\ & \quad \qquad 
		= \left \| u_n(t,0)-u_m(0,\cdot)
		\big|_{V_n}\right\|_{L^2(M)}^2
		+\delta\left \| u_n(0,\cdot)-u_m(0,\cdot)
		\big|_{V_n}\right\|_{H^1(M)}^2
		\\ & \quad \qquad \qquad 
		+2 \int_0^t \int_M \left(\mff_n-\mff_m\right) 
		\nabla \left(u_n(t',\cdot)-u_m(t',\cdot)
		\big|_{V_n} \right) \, dV(\mx)\, dt'
		\\ & \quad \qquad \qquad \qquad 
		+\int_0^t \sum\limits_{k=1}^n 
		\frac{1}{1-\delta+\delta \lambda^2_k}
		\left(\int_M \left(\Phi_n-\Phi_m\right) e_k 
		\,dV(\mx)\right)^2 \,dt'
		\\ & \quad \qquad \qquad \qquad \qquad
		+2\int_0^t \int_M \left(u_n-u_m\right)
		\left(\Phi_n-\Phi_m\right)\,dV(\mx) \, dW_{t'}.
	\end{split}
\end{equation}
Denote by $D_f$ and $D_\Phi$ 
the third and fourth terms on the right-hand side 
of \eqref{eq:energy-tmp5}. By Young's product 
inequality and ($C_f$--2),
\begin{align*}
	\left| D_f\right| & \lesssim 
	C_\kappa\int_0^t \left\| u_n(t',\cdot)-u_m(t',\cdot)
	\right\|_{L^2(M)}^2\, dt' 
	\\ & \qquad 
	+\kappa\int_0^t \left \|\nabla \left(u_n(t',\cdot)
	-u_m(t',\cdot)\big|_{V_n}\right)
	\right\|_{L^2(M)}^2\, dt', 
	\quad \kappa\in (0,1),
\end{align*}
where the second term can be absorbed into 
the $\eps$-term on the left-hand side 
of \eqref{eq:energy-tmp5}. By the Cauchy-Schwarz 
inequality, ($C_\Phi$--2), and since 
$0<\lambda_k\to \infty$ as $k\to \infty$,
$$
\left| D_\Phi\right| \lesssim 
\int_0^t \left\| u_n(t',\cdot)-u_m(t',\cdot)
\right\|_{L^2(M)}^2\, dt'.
$$

Recall the $n$-independent $H^1$ estimate 
\eqref{eq:energy-tmp4}. One can use this to 
demonstrate that the remainder term 
$u_m(t,\cdot)-u_m(t,\cdot)\big|_{V_n}$ is 
uniformly (in $n$) small. Indeed,
\begin{align*}
	\left\|u_m(t,\cdot)-u_m(t,\cdot)
	\big|_{V_n}\right\|_{L^2(M)}^2 
	& = \left\|u_m(t,\cdot)\big|_{V_n^\bot}
	\right\|_{L^2(M)}^2
	\leq \frac{1}{\lambda_n^2}
	\left\|u_m(t,\cdot)\right\|_{H^1(M)}^2,
\end{align*}
where $\lambda_n\ton \infty$. 
Given $u_0\in H^2(M)$, cf.~\eqref{PP-ID}, 
we also have 
$$
\left\|u_m(0,\cdot)-u_m(0,\cdot)
\big|_{V_n}\right\|_{H^1(M)}^2
\lesssim \frac{1}{\lambda_n^2}
\left\|u_0\right\|_{H^2(M)}^2.
$$

Inserting the previous bounds into 
\eqref{eq:energy-tmp5} yields
\begin{equation}\label{eq:energy-tmp6}
	\begin{split}
		&\norm{u_n(t,\cdot)-u_m(t,\cdot)}_{L^2(M)}^2 
		\\ & \qquad 
		\le \norm{u_n(0,\cdot)-u_m(0,\cdot)}_{L^2(M)}^2
		+\delta\norm{u_n(0,\cdot)-u_m(0,\cdot)}_{H^1(M)}^2
		\\ & \qquad \quad
		+C\int_0^t \left\|u_n(t',\cdot)-u_m(t',\cdot)
		\right\|_{L^2(M)}^2\, dt'
		\\ &  \qquad \quad\quad
		+\frac{1}{\lambda_n^2}
		\left\|u_m(t,\cdot)\right\|_{H^1(M)}^2
		+\frac{1}{\lambda_n^2}
		\left\|u_0\right\|_{H^2(M)}^2
		\\ & \qquad\quad\quad\quad
		+2\int_0^t \int_M \left(u_n-u_m\right)
		\left(\Phi_n-\Phi_m\right)\,dV(\mx) \, dW_{t'},
	\end{split}
\end{equation}
for some $n$-independent constant $C=C(\delta,\eps)$.  
Taking the supremum over $t$ in \eqref{eq:energy-tmp6} 
and then the expectation, followed by an application of the BDG 
martingale inequality (as before), we arrive 
eventually at 
\begin{align*}
	& E\left[\sup_{t\in [0,\tau]}
	\norm{u_n(t,\cdot)-u_m(t,\cdot)}_{L^2(M)}^2\right]
	\\ & \qquad \le 
	\norm{u_n(0,\cdot)-u_m(0,\cdot)}_{L^2(M)}^2
	+\delta \norm{u_n(0,\cdot)-u_m(0,\cdot)}_{H^1(M)}^2
	\\ &  \qquad \quad
	+\frac{1}{\lambda_n^2}
	E\left[\sup_{t\in [0,T]}\left\|u_m(t,\cdot)
	\right\|_{H^1(M)}^2 \right]
	+\frac{1}{\lambda_n^2}\left\|u_0\right\|_{H^2(M)}^2
	\\ & \qquad\quad\quad
	+C\int_0^\tau E\left[\left\|u_n(t',\cdot)-u_m(t',\cdot)
	\right\|_{L^2(M)}^2\right]\, dt', 
	\quad \tau\in [0,T],
\end{align*}
for some $n$-independent constant $C$. 
Given $u_0\in H^2(M)$ and the 
estimate \eqref{eq:energy-tmp4}, an 
application of Gronwall's inequality finally gives
$$
\lim_{n,m\to \infty} 
E\left[\sup_{t\in [0,T]}
\norm{u_n(t,\cdot)-u_m(t,\cdot)}_{L^2(M)}^2\right]=0.
$$
In other words, $\left\{u_n\right\}_{n\in \N}$ 
is a Cauchy sequence in $L^2_{\prob}\left(\Omega;
C([0,T];L^2(M))\right)$, which is a Banach space. 
Hence, there is an adapted limit process $u$ such that
$u_n\to u$ in $L^2_{\prob}\left(\Omega;
C([0,T];L^2(M))\right)$. As a result, $u$ 
belongs to $C([0,T];L^2(M))$, almost surely.
Besides, cf.~\eqref{eq:energy-tmp4}, 
$u\in L^2_{\prob}\left(\Omega;
L^\infty(0,T;H^1(M))\right)$. In other words, 
$u$ satisfies \eqref{eq:pseudo-spaces}, 
and, thanks to \eqref{eq:energy-tmp4-L2}, 
also \eqref{eq:L2-uest-thm}. Given the strong 
convergence of $u_n$ to $u$, it is not 
difficult to show that $u$ is a weak solution 
to \eqref{PP-1}, \eqref{PP-ID}, in the 
sense of \eqref{mild-ppde}. This 
concludes the proof of the existence part 
of Theorem \ref{unique_sol_par}. Let us 
now turn to the uniqueness part of the theorem.  

\subsection{Uniqueness}
We will derive a $L^2_{\prob}\left(\Omega; 
C([0,T];L^2(M))\right)$
stability result for weak solutions, which at once implies 
the uniqueness part of Theorem \ref{unique_sol_par}. 
More precisely, we are going to prove that 
\begin{equation}\label{eq:uniq-tmp0}
	E\left[\sup_{t\in[0,T]}\norm{u(t)-v(t)}_{L^2(M)}^2\right]
	\lesssim \norm{u_0-v_0}_{L^2(M)}^2,
\end{equation}
where $u$, $v$ are two weak solutions 
of \eqref{PP-1'} with initial data 
$u_0$ and $v_0$, respectively. 
The difference satisfies (in the weak sense)
\begin{equation}\label{*-1}
	\begin{split}
		& d\left(u-v\right)
		+\mathrm{div}\, 
		\bigl({\mathfrak f}(\mx, u)
		-{\mathfrak f}(\mx, v)\bigr) 
		\, dt
		\\ & \qquad 
		=\eps \Delta \left(u-v\right) \, dt
		+ \delta d \Delta\left(u-v\right) 
		+ \bigl(\Phi(\mx,u)- \Phi(\mx,v)\bigr)
		\,dW_t.
	\end{split}
\end{equation} 
The two processes $u$, $v$ can be 
represented in terms of the $L^2(M)$-basis 
introduced earlier:
$$
u(\omega,t,\mx)=
\sum\limits_{k=1}^\infty 
\alpha_k(\omega,t)e_k(\mx), 
\quad v(\omega, t,\mx)
=\sum\limits_{k=1}^\infty 
\beta_k(\omega, t)e_k(\mx), 
$$ 
where $\left\{\alpha_k(\omega,t)
=\int_{M}u(\omega,t,\mx) e_k(\mx)\,dV(\mx)\right\}_k$, 
$\left\{\beta_k(\omega, t)
=\int_{M}v(\omega,t,\mx) e_k(\mx)\,dV(\mx)\right\}_k$. 
In passing, we note that this implies 
that $\alpha_k$ and $\beta_k$ are adapted.

By testing \eqref{*-1} with $e_k$, 
and then using the It\^{o} formula 
for every $k\in \N$, we obtain
\begin{align*}
	&\frac12\left(1-\delta+\delta \lambda^2_k\right) 
	d \bigl(\alpha_k(t)-\beta_k(t)\bigr)^2
	+\eps(\lambda^2_k-1) 
	\bigl(\alpha_k(t)-\beta_k(t)\bigr)^2 \, dt 
	\\& \quad = \int_M \bigl({\mathfrak f}(\mx,u)
	-{\mathfrak f}(\mx, v)\bigr) 
	\cdot  \nabla \bigl(\left(\alpha_k(t)
	-\beta_k(t) \right) e_k\bigr) \, dV(\mx) \, dt
	\\& \quad \qquad
	+\frac{1}{2\bigl(1-\delta+\delta \lambda^2_k\bigr)}
	\left(\, \int_M (\Phi(\mx, u)- \Phi(\mx, v))e_k 
	\,dV(\mx)\right)^2
	\\ & \quad \qquad\qquad
	+\int_M \bigl(\Phi(\mx, u)-\Phi(\mx, v)\bigr)
	\bigl(\left(\alpha_k(t)-\beta_k(t) \right) e_k 
	\,dV(\mx)\, dW_t
\end{align*} 
We divide this expression by 
$1-\delta+\delta \lambda^2_k$ and 
note that $\frac{\eps(\lambda^2_k-1)}
{1-\delta+\delta \lambda^2_k}
\ge-\frac{\eps}{\delta}$, 
in view of \eqref{assump-ed} and
since $0\leq \lambda_k \to \infty$ 
as $k\to \infty$. The result is  
\begin{align*}
	&\frac12 d \bigl(\alpha_k(t)-\beta_k(t)\bigr)^2
	-\frac{\eps}{\delta}
	\bigl(\alpha_k(t)-\beta_k(t)\bigr)^2 \, dt 
	\\& \quad 
	\le \frac{1}{1-\delta+\delta \lambda^2_k}
	\int_M \bigl({\mathfrak f}(\mx,u)
	-{\mathfrak f}(\mx, v)\bigr) 
	\cdot  \nabla \bigl(\left(\alpha_k(t)
	-\beta_k(t) \right) \, e_k\bigr) \, dV(\mx) \, dt
	\\& \quad \quad
	+\frac{1}{2\bigl(1-\delta+\delta \lambda^2_k\bigr)^2}
	\left(\, \int_M (\Phi(\mx, u)- \Phi(\mx, v))e_k 
	\,dV(\mx)\right)^2
	\\ & \quad \quad\quad
	+ \frac{1}{1-\delta+\delta \lambda^2_k}
	\int_M \bigl(\Phi(\mx, u)-\Phi(\mx, v)\bigr)
	\bigl(\left(\alpha_k(t)-\beta_k(t) \right) e_k 
	\,dV(\mx)\, dW_t.
\end{align*} 
Summing over $k\in \N$, taking into account that
$$
\sum_{k=0}^\infty  
\bigl(\alpha_k(t)-\beta_k(t)\bigr)^2
=\left\|u(t)-v(t)\right\|_{L^2(M)}^2
\le c(\omega),
$$
and integrating in time, we obtain
\begin{align*}
	& \norm{u(t)-v(t)}_{L^2(M)}^2
	\le \norm{u_0-v_0}_{L^2(M)}^2
	+\frac{2\eps}{\delta}\int_0^t 
	\norm{u(t')-v(t')}_{L^2(M)}^2\, dt' 
	\\& \quad 
	+2\int_0^t\int_M \bigl(\mff(\mx,u(t'))
	-{\mathfrak f}(\mx, v(t'))\bigr) 
	\\ & \quad \qquad\qquad\qquad 
	\times \sum_{k\in \N} 
	\frac{1}{1-\delta+\delta \lambda^2_k}
	\bigl(\alpha_k(t')
	-\beta_k(t')\bigr) \nabla e_k \, dV(\mx) \, dt'
	\\& \quad
	+\int_0^t\sum_{k\in \N}
	\frac{1}{\bigl(1-\delta+\delta \lambda^2_k\bigr)^2}
	\left(\, \int_M (\Phi(\mx,u(t'))-\Phi(\mx,v(t')))
	\, e_k \,dV(\mx)\right)^2\, dt'
	\\ & \quad
	+\int_0^t\sum_{k\in \N}
	\frac{2}{1-\delta+\delta \lambda^2_k}
	\int_M \bigl(\Phi(\mx,u(t'))-\Phi(\mx,v(t'))\bigr)
	\\ & \quad \qquad\qquad\qquad 
	\times \bigl(\alpha_k(t')-\beta_k(t') \bigr)\, e_k 
	\,dV(\mx)\, dW_{t'},
\end{align*} 
We denote the five terms on the 
right-hand side by $I_i(t)$, $i=1,\ldots,5$. 
By Young's product inequality and \eqref{assump-ed} again, 
\begin{align*}
	\left|I_2(t) \right\|
	& \le \int_0^t\left\|{\mathfrak f}(\cdot, u(t'))
	-{\mathfrak f}(\cdot, v(t'))\right\|_{L^2(M)}^2\, dt'
	\\ & \quad \qquad
	+\int_0^t\sum_{k\in \N} 
	\frac{1}{\left(1-\delta+\delta \lambda_k^2\right)^2} 
	\bigl(\alpha_k(t')-\beta_k(t')\bigr)^2 
	\left\|\nabla e_k\right\|_{L^2(M)}^2\, dt'
	\\ & \overset{\text{($C_f$--2)}}{\leq} 
	C_f \int_0^t\left\|u(t',\cdot)-v(t',\cdot)
	\right\|_{L^2(M)}^2\,dt'
	\\ & \quad\qquad 
	+\int_0^t\sum_{k\in \N} 
	\frac{\lambda_k^2-1}
	{\bigl(\delta\left(\lambda_k^2-1\bigr)+1\right)^2}
	\bigl(\alpha_k(t')-\beta_k(t')\bigr)^2\, dt' 
	\\ &
	\le \left(C_f + \frac{1}{\delta}\right)
	\int_0^t\left\|u(t')-v(t')\right\|_{L^2(M)}^2\, dt',
\end{align*}
and, by \eqref{assump-ed} 
and Bessel's inequality, 
\begin{align*}
	\left|I_3(t)\right\|
	& \le \frac{1}{\delta^2}
	\int_0^t\left\|\Phi(\mx,u(t'))
	-\Phi(\mx,v(t'))\right\|_{L^2(M)}^2\, dt'
	\\ & 
	\overset{\text{($C_\Phi$--2)}}{\le} 
	\frac{C_\Phi}{\delta^2} 
	\int_0^t\left\|u(t',\cdot)-v(t',\cdot)
	\right\|_{L^2(M)}^2\,dt'
\end{align*}
Consequently, taking the supremum over $t\in [0,\tau]$, 
$\tau\in [0,T]$, $T>0$, and then 
the expectation $E$, 
\begin{align}
	& E \left[\sup_{t\in [0,\tau]}
	\norm{u(t)-v(t)}_{L^2(M)}^2\right]
	\le \norm{u_0-v_0}_{L^2(M)}^2
	+C_\delta E\left[\, \int_0^\tau\left\|u(t',\cdot)-v(t',\cdot)
	\right\|_{L^2(M)}^2\,dt' \right]
	\notag \\ & \qquad 
	+ E\Biggl[\, \sup_{t\in [0,\tau]}\Biggl| \int_0^t\sum_{k\in \N}
	\frac{2}{1-\delta+\delta \lambda^2_k}
	\int_M \bigl(\Phi(\mx,u(t'))-\Phi(\mx,v(t'))\bigr)
	\notag
	\\ & \quad \qquad \qquad\qquad\qquad\qquad \times 
	\bigl( \alpha_k(t')-\beta_k(t') \bigr)\, e_k 
	\,dV(\mx)\, dW_{t'}\Biggr|\, \Biggr].
	\label{eq:uniq-tmp}
\end{align}

Let us estimate the last term of 
\eqref{eq:uniq-tmp}, which we 
denote by $M(t)$. By the BDG 
martingale inequality,
\begin{align*}
	M(t) &\lesssim E\Biggl[\,
	\int_0^\tau \Biggl(\, \sum_{k\in \N}
	\frac{2}{1-\delta+\delta \lambda^2_k}
	\int_M \bigl(\Phi(\mx,u(t'))-\Phi(\mx,v(t'))\bigr)
	\\ & \quad\qquad\qquad\qquad\qquad\qquad
	\times \bigl(\alpha_k(t')-\beta_k(t') \bigr)\, e_k 
	\,dV(\mx) \Biggr)^2
	\, dt' \Biggr]^{\frac12}.
\end{align*}
Given \eqref{assump-ed}, we can bound 
$\frac{2}{1-\delta+\delta \lambda^2_k}$ 
by $2/\delta^2$. An application of 
the Cauchy-Schwarz inequality then gives
$$
M(t)\lesssim_\delta
E\left [\, \int_0^\tau 
\left\|\Phi(\mx,u(t'))-\Phi(\mx,v(t'))\right\|_{L^2(M)}^2
\left\|u(t')-v(t')\right\|_{L^2(M)}^2
\, dt'\right]^{\frac12}. 
$$
and so, by ($C_\Phi$--2) and 
Young's product inequality, 
\begin{align*}
	M(t) & \lesssim_\delta 
	E \Biggl[\, \sup_{t\in [0,\tau]}
	\left\|u(t))-v(t)\right\|_{L^2(M)}
	\left(\int_0^\tau
	\left\|u(t')-v(t')\right\|_{L^2(M)}^2
	\, dt'\right)^{\frac12}\, \Biggr]
	\\ & \le
	\frac12 E \left[\sup_{t\in [0,\tau]}
	\left\|u(t))-v(t)\right\|^2_{L^2(M)}\right]
	+\frac12 \int_0^\tau E\left[
	\left\|u(t')-v(t')\right\|_{L^2(M)}^2\right]\, dt'.
\end{align*}
Inserting this bound into \eqref{eq:uniq-tmp} 
and then applying Gronwall's inequality, 
we finally arrive at the sought-after 
stability result \eqref{eq:uniq-tmp0}.

\section{Stochastic velocity averaging}
\label{sec:velocity-averaging}

Our aim is to prove the $L_{\loc}^1((0,T)\times M)$--convergence 
of the sequence $\left\{u_n\right\}$ of solutions to the 
pseudo-parabolic SPDE \eqref{PP-1}. The strategy is to 
rewrite \eqref{PP-1} in the kinetic formulation and 
then to use the velocity averaging result to be proved in this section. 
Roughly speaking, the kinetic formulation 
of \eqref{PP-1} has the form (see Section \ref{subsec:kinetic}) 
\begin{align*}
	d h_n+\Div^{g}_\mx \bigl(
	F(\mx,\lambda)h_n\bigr)\, dt
	=H_{n}\, dt+\Psi_{n} \, dW_t,
\end{align*} 
where $\Div^{g}$ is the divergence on 
the previously specified Riemannian manifold with 
metric $g$, whereas $H_n$ is a 
$\cD'((0,T)\times M\times \R)$-valued random variable. 
With the help of a smooth partition of unity subordinate to a finite atlas, 
we may consider localised versions of this equation “pulled back” to $\R^d$. 
``Gluing" together the solutions of the localised  equations will yield 
a global solution on $M$. We can produce a convergent 
subsequence of global solutions by establishing 
sub-sequential convergence of the solutions to the 
localised equations. In other words, since 
compactness is a local property, 
it is enough to consider the former kinetic 
SPDE on a single chart, i.e., the equation 
given in local coordinates. This means that the 
left-hand side of the above equation becomes (writing 
$\Div_\mx$ for the standard Euclidean divergence)
\begin{equation*}
	d h_n+\Div_\mx \bigl(F(\mx,\lambda)h_n\bigr)\, dt
	+\sum\limits_{j,k=1}^d\Gamma^j_{kj}F_k(\mx,\lambda)h_n\, dt,
\end{equation*}
with a similar modification of the right-hand side, noting 
that $H_n$ contains divergence and gradient operators, see Sections 
\ref{subsec:geometry} and \ref{subsec:kinetic}. 
Since $\sum\limits_{j,k=1}^d\Gamma^j_{kj}
F_k(\mx,\lambda)h_n$ is locally bounded, we can incorporate 
this term into the right-hand side of the equation
and then end up with a Euclidean SPDE of the form
\begin{equation*}
	d h_n+\Div_\mx\big(F(\mx,\lambda)h_n\bigr)\, dt
	=\bar{H}_{n}\, dt +\Phi_{n} \, dW_t, 
	\quad t\in (0,T), \,\, 
	\mx \in \R^d,\,\, \lambda \in \R^m.  
\end{equation*} 

\medskip

We shall now specify the kinetic SPDE and the 
velocity averaging problem more precisely. To this end, 
let $\Seq{h_n=h_n(\omega,t,\mx,\lambda)}_{n\in \N}$ be 
a sequence of $L^\infty(\R^d\times \R^m)$-valued 
stochastic processes, and suppose each $h_n$ solves 
a stochastic transport (continuity) 
equation of the form
\begin{equation}\label{eq-1}
	\begin{split}
		& d h_n 
		+\Div_\mx \bigl(F(\mx,\lambda)h_n\bigr)\, dt
		= \bigl( g_{n}+G_{n}\bigr)\, dt 
		+\Phi_{n} \, dW_{n}(t), 
		\\ & 
		h_n\big|_{t=0}=h_0(\mx,\lambda),
	\end{split}
\end{equation} 
where $(t,\mx,\lambda)\in (0,T)\times 
\R^d\times \R^m$ and $h_0\in L^\infty(\R^d\times \R^m)$ 
is a given initial function. Later we will
impose further conditions, see \eqref{eq:hk-L2t-L2weak} 
and the conditions (a)-(e).

In the SPDE \eqref{eq-1}, $F=F(\mx,\lambda)$ is 
a deterministic function (with a potentially 
discontinuous $\mx$-dependency) 
and $g_{n}=g_{n}(\omega,t,\mx,\lambda)$, $G_{n}
=G_{n}(\omega,t,\mx,\lambda)$ are 
distribution-valued random variables satisfying, 
for each $\rho \in C^{\infty}_c(\R^m)$,
\begin{equation}\label{notation-rho}
	\begin{split}
		g_{n,\rho} & :=\bigl\langle g_{n}, \rho\bigr \rangle
		= (-1)^{m_1} \int_{\R^m} \overline{g}_{n}
		\frac{\partial^{m_1}\rho}{\pa\lambda^{m_1}} \, d\lambda,
		\qquad m_1\in \N,
		\\ 
		G_{n,\rho} & := \bigl\langle G_{n},\rho \bigr \rangle
		= (-1)^{m_2} \int_{\R^m} \overline{G}_{n}
		\frac{\partial^{m_2}\rho}{\pa\lambda^{m_2}}\, d\lambda,
		\qquad m_2\in \N,
	\end{split}
\end{equation} 
for some distribution-valued random variables
$\overline{g}_{n}$, $\overline{G}_{n}$ such that
$g_{n,\rho}$ takes values in $L^1(0,T;W^{-1,r}_{\loc}(\R^d))$ 
(for $r>1$) and $G_{n,\rho}$ takes 
values in $\cM_{\loc}((0,T)\times \R^d)$.

\begin{remark}
Throughout the paper, the subscript ``\rm{loc}" 
refers to the unbounded part of a domain; 
for example, $G_{n,\rho} \in \cM_{\loc}((0,T)\times \R^d)$ 
means that $G_{n,\rho} \in \cM((0,T)\times K)$ 
for any $K\subset \subset \R^d$.
\end{remark}

\begin{remark}\label{rem:source-structure}
The specific structure \eqref{notation-rho} 
of the source terms $g_{n,\rho}$, $G_{n,\rho}$
is not important in this section (the relevant 
conditions (c)-(d) are gathered later on). However, this structure 
is supplied by the application in 
Section \ref{sec:singular-limit}, see the 
kinetic SPDE \eqref{eq:tilde-kinetic}.
\end{remark}

The noise amplitude $\Phi_{n}=\Phi_{n}(\omega,t,\mx,\lambda)$ 
is an adapted $L^2(\R^d\times \R^m)$ 
stochastic process, for which 
we frequently use the notation
\begin{equation}\label{eq:Phi-n-rho-average}
	\Phi_{n,\rho}:= \left\langle \Phi_{n}, \rho\right \rangle
	=\int_{\R^m} \Phi_n \rho\, d\lambda.
\end{equation}
As usual, ``adapted" means ``weakly adapted" in the sense
that, for every $\varphi \in L^2(\R^d\times \R^m)$, the 
real-valued process $(\omega,t)\mapsto 
\int_{\R^d}\int_{\R}\Phi_{n} \varphi(\mx,\lambda)
\, d\lambda\, d\mx $ is adapted with respect 
to a fixed filtered probability space 
$\bigl(\Omega,\F,\prob,\Seq{\F_t^n}\bigr)$. 
Finally, $W_n$ is a real-valued Wiener process defined 
on $\bigl(\Omega,\F,\prob,\Seq{\F_t^n}\bigr)$. 

The SPDE \eqref{eq-1} is interpreted 
in the $(t,\mx,\lambda)$-weak sense:
\begin{equation}\label{weak-sense-tx}
	\begin{split}
		&\int_0^T\int_{\R^{d+m}} h_n \pa_t \varphi
		\rho \, d\lambda \, d\mx \, dt
		+\int_{\R^{d+m}}  h_0(\mx,\lambda) 
		\varphi(0,\mx)\rho(\lambda)  \, d\lambda \, d\mx
		\\ & \qquad +\int_0^T\int_{\R^{d+m}} h_n F(\mx,\lambda)
		\cdot \nabla_{\mx}\varphi \rho \, d\lambda\, d\mx\,dt
		\\ & \qquad\qquad 
		= m_{n,\rho}(\varphi)
		+\int_0^T\int_{\R^{d+m}} \Phi_{n}
		\varphi\rho \,d\lambda \,d\mx \, dW_{n}(t), 
		\quad \text{almost surely},
	\end{split}	
\end{equation}
for all $\varphi\in C^\infty_c([0,T)\times \R^d)$ 
and $\rho \in C^{\infty}_c(\R^m)$, where 
\begin{equation}\label{eq:measure-m-def}
	m_{n,\rho}(\varphi)
	:=\bigl\langle g_{n,\rho}+ G_{n,\rho}, 
	\varphi\bigr\rangle
	=\int_{0}^{T} \int_{\R^d}
	\bigl( g_{n,\rho} +  G_{n,\rho} \bigr) 
	\varphi(t,\mx)\, d\mx \,dt.
\end{equation}
If $\varphi(t,\mx)=\phi(\mx)\gamma(t)$, we write 
$m_{n,\rho}(\phi\gamma)=:m_{n,\rho}(\phi)(\gamma)$. 
If $\gamma=\chi_{[t_1,t_2]}$, we also 
write $m_{n,\rho}(\phi)\bigl([t_1,t_2]\bigr)$.

For some upcoming computations involving the stochastic 
product rule, we need to turn the $(t,\mx,\lambda)$-weak 
formulation \eqref{weak-sense-tx} 
into a formulation that is weak in $(\mx,\lambda)$  
and pointwise in time $t$. The assumptions of 
the next lemma are related to (but weaker than) 
the properties derived for the kinetic formulation of the 
pseudo-parabolic SPDE \eqref{PP-1}, see 
the kinetic SPDE \eqref{eq:tilde-kinetic}, the 
estimates \eqref{eq:tilde-kinetic-est}, and the sign relation 
\eqref{eq:tilde-sign-functions}.

\begin{lemma}[weak formulation in $(\mx,\lambda)$, 
pointwise in $t$]\label{ito-eq}
Fix any integer $n$, and let $h_n \in L^\infty_{\omega,t,\mx,\lambda}:=
L^\infty(\Omega\times (0,T)\times \R^d \times \R^m)$ 
be a weak solution to \eqref{eq-1} in 
the sense of \eqref{weak-sense-tx}. For any 
$\rho \in C^\infty_c(\R^{d+m})$, suppose
\begin{itemize}
	\item $F\in L^1_{\loc}(\R^d\times \R^m;\R^d)$,
	
	\item $g_{n,\rho}$ is bounded in 
	$L^1_{\prob}\bigl(\Omega;L^1(0,T;W^{-1,r}_{\loc}(\R^d)\bigr)$, 
	for some $r>1$,

	\item $G_{n,\rho}$ is bounded in 
	$L^1_{\prob}\bigl(\Omega;\cM_{\loc}((0,T)\times \R^d)\bigr)$, and
 
	\item $\Phi_{n,\rho}$ is bounded in 
	$L^2_{\prob}\bigl(\Omega;L^2_{\loc}((0,T)\times \R^d)\bigr)$.
\end{itemize}
Then, for every $\varphi \in L^1(\R^{d+m})$, the function 
$t\mapsto \int_{\R^{d+m}}h_n(t) \varphi
\, d\lambda\, d\mx$ admits left and right 
limits at every instant of time (a.s.); more precisely, 
there exist functions $h_n^{\pm}\in L^\infty_{\omega,t,\mx,\lambda}$ 
such that, for any $\bar{t}\in (0,T)$,
\begin{equation}\label{l-r}
	\begin{split}
		&\lim_{t\downarrow \bar{t}} \int_{\R^{d+m}} h_n(t)  
		\varphi \, d\lambda\, d\mx
		=\int_{\R^{d+m}} h^+_n(\bar{t})  
		\varphi \, d\lambda\, d\mx, 
		\quad \text{a.s.},
		\\ &
		\lim_{t\uparrow \bar{t}} \int_{\R^{d+m}} h_n(t)  
		\varphi\, d\lambda\, d\mx
		=\int_{\R^{d+m}} h^-_n(t)
		\varphi\, d\lambda\, d\mx, 
		\quad \text{a.s.},
	\end{split}
\end{equation} 
and the set of times $t\in [0,T]$ for which  
\begin{equation}\label{l-r-1}
	\int_{\R^{d+m}} h^+_n(t)
	\varphi \, d\lambda\, d\mx
	\neq \int_{\R^{d+m}} h^-_n(t)
	\varphi \, d\lambda\, d\mx,
	\quad \text{a.s.},
\end{equation} 
is at most countable.

For almost all $0\leq t_1<t_2\leq T$ and all $(\phi,\rho)\in 
C^\infty_c(\R^d)\times C^\infty_c(\R^m)$,
\begin{equation}\label{weak-sense}
	\begin{split}
		& \int_{\R^{d+m}} h_n(t_2)\phi \rho \, d\lambda\, d\mx 
		-\int_{\R^{d+m}} h_n(t_1) \phi \rho \, d\lambda\, d\mx 
		\\ & \qquad 
		-\int_{t_1}^{t_2}\int_{\R^{d+m}}h_n F(\mx,\lambda)
		\cdot \nabla_{\mx}\phi\rho \,d\lambda\,d\mx\, dt
		\\ & \qquad\qquad 
		= m_{n,\rho}(\phi)\bigl([t_1,t_2]\bigr)
		+\int_{t_1}^{t_2}\int_{\R^{d}} 
		\Phi_{n,\rho}  \phi\,d\mx  \, dW_{n}(t),
		\quad \text{a.s.}
	\end{split}
\end{equation}
\end{lemma}

\begin{proof}
The proof borrows heavily from 
\cite[Proposition 8]{Debussche:2010fk}, see 
also \cite[Section 2.3]{Dotti:2018aa} 
and \cite[Lemma 1.3.1]{Dafermos:2016aa}. 

First, consider the the real-valued 
stochastic process
\begin{equation}\label{eq:Iphirho}
	I_{\phi,\rho}(t):=\int_{\R^{d+m}} h(t) 
	\phi(\mx) \rho(\lambda)\,d\lambda\, d\mx, 
	\quad \phi\in C^\infty_c(\R^{d}), 
	\quad \rho\in C^\infty_c(\R^{m}),
\end{equation}
where we have dropped the $n$-index 
(as it is kept fixed in the lemma). 
Moreover, set
$$
H_{\phi,\rho}(t):=I_{\phi,\rho}(t)
-\int_0^{t}\int_{\R^{d+m}} 
\Phi(t)\phi(\mx) \rho(\lambda)\,d\lambda\, d\mx \, dW(t)
=:I_{\phi,\rho}(t)-J_{\phi,\rho}(t).
$$

In what follows, it is enough to select $\phi$ and $\rho$ 
from countable collections $\Seq{\phi_\ell}$ 
and $\Seq{\rho_\ell}$ of test functions 
$\phi_\ell\in C^\infty_c(\R^{d})$ 
and $\rho_\ell\in C^\infty_c(\R^{m})$ 
that are dense in $L^1(\R^{d})$ and $L^1(\R^{m})$, 
respectively, recalling that 
the algebraic tensor product $L^1(\R^{d})\otimes 
L^1(\R^{m})$ is dense in $L^1(\R^{d}\times \R^{m})$.

The weak $(t,\mx,\lambda)$-formulation 
\eqref{weak-sense-tx} implies that
\begin{equation}\label{weak-sense-tmp1}
	\begin{split}
		& \int_0^T H_{\phi,\rho}(t) \gamma'(t)\, dt
		+\int_{\R^{d+m}}h_0(\mx,\lambda)\phi(\mx) 
		\rho(\lambda)\gamma(0)\,d\lambda\, d\mx
		\\ & \qquad 
		=\int_{0}^{T}\int_{\R^{d+m}} 
		h(t)F(\mx,\lambda)
		\cdot \nabla_{\mx}\phi(\mx) \, \rho(\lambda) 
		\gamma(t)\, d\lambda \, d\mx \, dt
		+m_\rho(\phi)(\gamma),
	\end{split}	
\end{equation}
for any $\gamma\in C^1_c([0,T))$. In view of our assumptions 
on $F$, $g_\rho$, $G_\rho$ and since 
$h\in L^\infty_{\omega,t,\mx,\lambda}$, the equation
\eqref{weak-sense-tmp1} implies that $H_{\phi,\rho}(t)$ 
is a.s.~of bounded total variation on $[0,T]$. 
As a result, $H_{\phi,\rho}(t)$ admits left 
and right limits at all times $\bar{t}\in [0,T]$. 
Since $J_{\phi,\rho}$ is continuous, 
it follows that $I_{\phi,\rho}$ admits 
left and right limits at all times $\bar{t}\in [0,T]$, 
that is, the limits
$$
I_{\phi,\rho}(\bar{t}+)=
\lim_{t\downarrow \bar{t}}I_{\phi,\rho}(t),
\quad
I_{\phi,\rho}(\bar{t}-)=
\lim_{t\uparrow \bar{t}}I_{\phi,\rho}(t)
$$ 
exist at all times $\bar{t}$, for all 
$\omega\in \Omega_1$, where $\Omega_1$ 
is of full $\prob$-measure. 

Fix $\omega\in \Omega_1$. We have 
$$
I_{\phi,\rho}(\omega,\bar{t}+)=\lim_{\kappa\downarrow 0}
\frac{1}{\kappa}\int_{\bar{t}}^{\bar{t}+\kappa} 
I_{\phi,\rho}(\omega,t)\,dt
= \lim_{\kappa\downarrow 0}\int_{\R^{d+m}} 
\bar{A}_\kappa(\omega,\cdot)
\phi\,d\lambda\, d\mx,
$$
where $\bar{A}_\kappa(\omega,\cdot):=\frac{1}{\kappa}
\int_{\bar{t}}^{\bar{t}+\kappa} h(\omega,t) \,dt$. Noting 
that $\bar{A}_\kappa(\omega,\cdot):\R^{d+m}\to \R$ 
is bounded uniformly in $\kappa$ 
(and also in $\omega$, $\bar{t}$), that is, 
$\norm{\bar{A}_\kappa(\omega,\cdot)}_{L^\infty(\R^{d+m})}
\le C$, there exist a subsequence 
$\left\{\bar{A}_{\kappa_j}(\omega,\cdot)\right\}_j
\subset \left\{\bar{A}_\kappa(\omega,\cdot)\right\}_\kappa$ and 
a limit $h^+(\bar{t})=h^+(\bar{t})(\omega,\cdot)
\in L^\infty(\R^{d+m})$ such that 
$\bar{A}_{\kappa_j}(\omega,\cdot)\to h^+(\bar{t})(\omega,\cdot)$ 
weak-$\star$ in $L^\infty(\R^{d+m})$. 

The limit $h^+(\bar{t})(\omega,\cdot)$ 
is unique. Indeed, if there is another 
subsequence $\left\{\bar{A}_{\kappa_k}(\omega,\cdot)\right\}_k
\subset \left\{\bar{A}_\kappa(\omega,\cdot)\right\}_\kappa$ 
and another limit $\tilde{h}^+(\omega,\cdot)\in L^\infty(\R^{d+m})$ 
such that $\bar{A}_{\kappa_k}(\omega,\cdot)\to \tilde{h}^+$ 
weak-$\star$ in $L^\infty(\R^{d+m})$, then necessarily 
$\int_{\R^{d+m}} h^+(\bar{t})(\omega,\cdot)\phi \rho
\,d\lambda\, d\mx=I_{\phi,\rho}(\bar{t}+)
=\int_{\R^{d+m}}\tilde{h}^+(\omega,\cdot) 
\phi\rho \,d\lambda\, d\mx$ for all $\phi\in \{\phi_\ell\}$, $\rho 
\in \Seq{\rho_\ell} $, which is enough 
to conclude that $h^+(\bar{t})(\omega,\cdot)
=\tilde{h}^+(\omega,\cdot)$ a.e.~in $\R^{d+m}$. 
Consequently, the entire sequence 
$\left\{\bar{A}_\kappa(\omega,\cdot)\right\}_\kappa$ converges 
weakly-$\star$ in $L^\infty(\R^{d+m})$ to 
$h^+(\bar{t})(\omega,\cdot)$, 
which implies that the subsequence $\left\{\kappa_j\right\}$ 
does not depend on (the fixed) $\omega\in \Omega_1$. 
In particular, this ensures that whenever $h^+(\bar{t})$ 
is viewed as a function of $(\omega,\mx,\lambda)$, $h^+(\bar{t})$ 
is measurable; indeed, $E\int_{\R^{d+m}} 
\bar{A}_\kappa\phi \rho \,d\lambda\, d\mx
\to E \int_{\R^{d+m}}h^+(\bar{t})
\phi \rho\,d\lambda\, d\mx$. 

Summarising, for any $\bar{t}\in [0,T)$, there 
exists $h^+(\bar{t})\in L^\infty(\Omega\times \R^{d+m})$ 
such that
$$
\lim_{t\downarrow \bar{t}}\int_{\R^{d+m}} 
h(t) \phi \rho\,d\lambda\, d\mx
=\int_{\R^{d+m}} h^+(\bar{t}) 
\phi \rho \,d\lambda\, d\mx, 
\quad \text{$\forall \phi\in \Seq{\phi_\ell}$, 
$\forall \rho \in \Seq{\rho_\ell}$},
$$ 
for all $\omega\in \Omega_1$, where $\Omega_1$ is 
of full $\prob$-measure. By a density argument, this equation 
continues to hold with $\phi\rho$ replaced by any $\varphi=\varphi(\mx,\lambda) 
\in L^1(\R^{d+m})$. Similarly, for any $\bar{t}\in [0,T]$, there 
exists $h^-(\bar{t})\in L^\infty(\Omega\times \R^{d+m})$ 
such that ($\forall \omega\in \Omega_1$)
$$
\lim_{t\uparrow \bar{t}}
\int_{\R^{d+m}} h(t) \phi\rho\,d\lambda\, d\mx
= \int_{\R^{d+m}} h^-(\bar{t})
\phi \rho\,d\lambda\, d\mx, 
\quad \text{$\forall \phi\in \Seq{\phi_\ell}$, 
$\forall \rho \in \Seq{\rho_\ell}$},
$$ 
and this equation continues to hold with $\phi\rho$ 
replaced by any $\varphi=\varphi(\mx,\lambda) 
\in L^1(\R^{d+m})$.

Given the existence of these weak (left and right) traces, 
using \eqref{weak-sense-tmp1} with $\gamma$ 
of the form ($\kappa>0$)
$$
\gamma(t)=\gamma_\kappa(t)=
\begin{cases}
	\frac{t-\bar{t}+\kappa}{\kappa}, 
	& \bar{t}-\kappa \leq t \leq \bar{t},
	\\ 
	\frac{\bar{t}-t+\kappa}{\kappa}, 
	& \bar{t}\leq t \leq \bar{t}+\kappa,
	\\ 0, & \text{otherwise}
\end{cases}
$$
and then sending $\kappa\to 0$, we obtain
$$
\int_{\R^{d+m}} h^+(\bar{t}) \phi \rho \,d\lambda\, d\mx
-\int_{\R^{d+m}} h^- (\bar{t})\phi\, \rho\,d\lambda\, d\mx
=m_\rho(\phi)\!\left(\bigl\{\bar{t} \bigr\}\right), 
\quad \text{a.s.}
$$
The set of $\bar{t}$ for which 
$m_\rho(\phi)\!\left(\bigl\{\bar{t} \bigr\}\right)\neq 0$ 
a.s.~is at most countable, recalling that 
$\bigl\langle G_\rho, \phi \bigr\rangle$ 
is a finite measure on $[0,T]$ (so the set of atoms is countable), 
while $\bigl\langle g_\rho, \phi \bigr\rangle$ is integrable on $[0,T]$ 
(thus no atoms). This concludes the proof of \eqref{l-r} and \eqref{l-r-1}.

Finally, let us prove the relation \eqref{weak-sense}, which is a 
consequence of the obtained regularity. 
Using \eqref{weak-sense-tmp1} with $\gamma$ 
of the form ($\kappa>0$)
$$
\gamma(t)=\gamma_\kappa(t)=
\begin{cases}
	\frac{t-t_1+\kappa}{\kappa}, 
	& t_1-\kappa \leq t \leq t_1,
	\\ 
	1, & t_1<t<t_2,
	\\
	\frac{t_2-t+\kappa}{\kappa}, 
	& t_2 \leq t \leq t_2+\kappa,
	\\ 
	0, & \text{otherwise}
\end{cases}
$$
and then sending $\kappa\to 0$, we obtain
\begin{equation}\label{weak-sense-tmp2}
	\begin{split}
		& H_{\phi,\rho}(t_2+)-H_{\phi,\rho}(t_1-)
		=\int_{t_1}^{t_2}\int_{\R^{d+m}} 
		h(t)F(\mx,\lambda)
		\cdot \nabla_{\mx}\phi\rho 
		\, d\lambda \, d\mx \, dt
		+m(\phi)\bigl([t_1,t_2]\bigr).
	\end{split}
\end{equation}
Combining \eqref{weak-sense-tmp2} with \eqref{l-r-1} and  
the continuity of $J_{\phi,\rho}(t)$, we 
arrive at \eqref{weak-sense}.
\end{proof}

\begin{remark}
According to Lemma \ref{ito-eq}, a weak solution $h$ of the 
kinetic SPDE \eqref{eq-1}---in the sense of \eqref{eq:weak-sol-tmp1}---admits 
weak left and right limits at each instant of time $t$. 
Consequently, upon replacing $h$ by its right-continuous 
version (we do not change the notation), we may assume that 
the It\^{o} equation \eqref{weak-sense} holds 
with $t_0=0$ and $t_2=t$ for all $t\in (0,T)$.  
Recall that there are general theorems \cite{Revuz:1999wi} 
ensuring that many real-valued stochastic processes $X(t)$ (discontinuous 
semimartingales) have a right-continuous version that 
necessarily has left-limits everywhere. 
Right-continuous processes with left-limits everywhere are referred to 
as \textit{c{\`a}dl{\`a}g}. In view of Lemma \ref{ito-eq}, whenever 
convenient, we may assume that $X(t)=\action{h(t)}{\phi\rho}$, 
cf.~\eqref{eq:Iphirho}, is c{\`a}dl{\`a}g. This property and the 
resulting It\^{o} equation \eqref{weak-sense} (with $t_0=0$ 
and $t_2=t$)  will allow us to perform 
some temporary computations involving the real-valued 
It\^{o} chain rule, or more precisely the product rule. 

The (discontinuous) real-valued stochastic processes 
$X(t)=\action{h(t)}{\phi\rho}$ is of the form 
$X(t)=A(t)+M(t)$, cf.~\eqref{weak-sense}, where $A(t)$ is a 
finite variation process and $M(t)$ is a continuous martingale. 
Moreover, we have that $A(0)=\action{h_0}{\varphi}
-m(\phi)(\set{0})$ and $M(0)=0$. The fact that 
the initial condition $X(0)=\action{h_0}{\phi\rho}$ 
fails to be satisfied (in general) is of no 
consequence for the question of compactness analysed here. 
In the proof below, we need to determine the equation 
satisfied by the product of two such processes $X_1,X_2$. 
Noting that
$$
X_1(t)X_2(t)=A_1(t)A_2(t)+A_1(t)M_2(t)+A_2(t)M_1(t)
+M_1(t)M_2(t),
$$ 
we can calculate the first three terms using 
standard calculus, while the fourth term can be computed 
using the It\^{o} product formula for continuous martingales. 
Alternatively, one may use the It\^{o} product formula 
for discontinuous semimartingales to compute 
$d\bigl(X_1(t)X_2(t)\bigr)$.
\end{remark}

\medskip

In what follows, we will impose 
the following conditions:  

\medskip

\begin{itemize} 
\item[(a)] $\norm{h_n}_{L^\infty(\Omega\times
\R_+\times\R^d\times\R^m)}\lesssim 1$, and  
\begin{equation}\label{eq:hk-L2t-L2weak}
	h_n \ton 0 \quad 
	\text{in $L^2_t\bigl(L^2_{\loc,w}(\R^d\times \R^m)\bigr)$, 
	almost surely},
\end{equation}
where the ``strong-in-$t$, weak in $\mx,\lambda$" space 
$L^2_t\bigl(L^2_{\loc,w}(\R^d\times \R^m)\bigr)$ 
is defined as in \eqref{eq:L2tL2x-weak}.

\item[(b)] $F\in L^q_{\loc}(\R^d\times \R^m;\R^d)$, 
for some $q>2$ satifying $\frac{1}{2}+\frac{1}{q}<1$.

\item[(c)] for each $\rho=\rho(\lambda)\in C^\infty_c(\R^m)$,
\begin{equation}\label{gk-a}
	E\left[\norm{g_{n,\rho}}_{L^2(0,T;W^{-1,r}_{\loc}(\R^d))}
	\right] \ton 0, 
\end{equation} 
for some $r>1$.

\item[(d)] for each $\rho=\rho(\lambda) \in C^\infty_c(\R^m)$,
\begin{equation}\label{gk-b-1-new}
	E\left[\norm{G_{n,\rho}}_{\cM_{\loc}((0,T)\times \R^d)}\right]
	\lesssim_{T,\rho} 1;
\end{equation}

\item[(e)] Each $\Phi_{n}$ is adapted, and 
for every $\rho\in C_c(\R^m)$,
$$
E\left[\norm{\Phi_{n,\rho}}_{L^2_{\loc}((0,T)\times \R^d)}^2\right]
\lesssim_{T,\rho} 1.
$$
\end{itemize}

\begin{remark}
The assumptions (a)--(e) of are modelled 
on the properties derived for 
the kinetic SPDE \eqref{eq:tilde-kinetic}, see the 
convergences \eqref{eq:tilde-conv}, the 
estimates \eqref{eq:tilde-kinetic-est}, 
and the sign property \eqref{eq:tilde-sign-functions}.
\end{remark}

Fix some $s\in (0,1)$ and $K\subset\subset \R^d$. 
By a simple adaptation of the arguments in 
\cite[Theorem 6, p.~7]{Evans:1990wt}, the space 
$\cM((0,T)\times K)$ is compactly embedded in 
$W^{-s,r}((0,T)\times K)$ for any 
$r\in \bigl[1,\frac{d+1}{d+1-s}\bigr)$. 
Consequently, the assumption \eqref{gk-b-1-new} implies that
\begin{equation}\label{gk-b-1}
	E\left[\norm{G_{n,\rho}}_{W^{-s,r}_{\loc}((0,T)\times \R^d)}
	\right]\lesssim_{\rho, T} 1,
\end{equation} 
for some $s<1$, $r>1$, and a constant 
$C$ independent of $n$.

We will now examine the compactness  properties of 
the sequence $\Seq{h_n}$ of $(t,\mx,\lambda)$-weak solutions 
of the kinetic SPDEs \eqref{eq-1}.  By Lemma \ref{ito-eq}, we may 
assume that $h_n$ satisfies the  It\^{o} equation \eqref{weak-sense} 
with $t_1=0$  and $t_2=t$ for any $t\in (0,T)$, which is the 
form of the equation to be used below. The main 
tool of investigation will be the ``stochastic" variant 
of the $H$-measure  given by Theorem \ref{thm:H-measure}.  
To set the scene for using the $H$-measure, we need 
to reformulate \eqref{eq-1} a bit. 

In what follows, to avoid the proliferation of symbols, we will 
assume that $h_n$ is compactly supported 
in $\mx$, as required by Theorem \ref{thm:H-measure}. We 
can make such an assumption without loss of generality. 
Indeed, since the kinetic SPDE \eqref{eq-1} is linear, 
we can multiply the equation by a cutoff function 
$\chi\in C_c^\infty(\R_\mx^d)$ and  end up with an equation 
of the same type, but with solutions $\chi h_n$ that 
are compactly supported with respect to $\mx$. 

First, we fix localisation functions $\varphi_1,\varphi_2 
\in C^\infty_c(\R^d)$ and $\rho_1,\rho_2\in C^\infty_c(\R^m)$. 
From \eqref{eq-1}, via Lemma \ref{ito-eq}, it follows 
that (replacing $\lambda$ by $\mp$)
\begin{equation}\label{p}
	\begin{split}
		& d\int_{\R^m}  
		\bigl(h^\mp_n \varphi_1\bigr) \, \rho_1(\mp) \,d\mp
		+\Div_\mx \int_{\R^m}
		\bigl(F(\mx,\mp) h^{\mp}_n \varphi_1\bigr)\,
		\rho_1(\mp) \,d\mp\, dt
		\\& \quad -\nabla_\mx \varphi_1\cdot  
		\int_{\R^m}\bigl(F(\mx,\mp) h^\mp_n \bigr)\, 
		\rho_1(\mp)\, d\mp\, dt
		%\\ & \qquad\qquad 
		= \bigl(g_{n,\rho_1}
		+G_{n,\rho_1}\bigr) \varphi_1  \, dt 
		+\bigl(\Phi_{n,\rho_1} \varphi_1
		\bigr)\, dW_n(t), 
	\end{split}
\end{equation} 
where $h^\mp_n:=h_n(\omega,t,\mx,\mp)$.

From \eqref{p}, replacing $\varphi_1, \rho_1, \mp$ by 
$\varphi_2, \rho_2,\mq$, respectively, we obtain
\begin{equation}\label{q}
	\begin{split}
		& d \int_{\R^m} \bigl(\varphi_2 h^\mq_n\bigr)
		\, \rho_2(\mq) \, d\mq
		+\Div_\mx \int_{\R^m} \bigl(F(\mx,\mq) 
		h^\mq_n\varphi_2 \bigr)\, 
		\rho_2(\mq)\, d\mq\, dt
		\\& \quad 
		-\nabla_\mx \varphi_2\cdot
		\int_{\R^m}\bigl(F(\mx,\mq) h^\mq_n 
		\bigr)\, \rho_2(\mq) \, d\mq \, dt
		= \bigl(g_{n,\rho_2}
		+G_{n,\rho_2}\bigr) 
		\varphi_2 \, dt
		+\bigl(\Phi_{n,\rho_2}\varphi_2\bigr)\, dW_n(t).
	\end{split}
\end{equation} 

In the sequel, we also use the notation 
$F^{\mp}:=F(\mx,\mp)$ and $F^{\mq}:=F(\mx,\mq)$.
By applying the spatial Fourier transform 
$\F_{\mx}$ in \eqref{p} and \eqref{q}, we 
obtain the following SDEs
\begin{equation*}
	\begin{split}
		& d\int_{\R^m}\F_{\mx} 
		\bigl(h^\mp_n \varphi_1 \bigr)
		\,\rho_1(\mp) \, d\mp
		-2\pi i \mxi' \cdot \int_{\R^m}
		\F_{\mx} \bigl(F^{\mp} h^{\mp}_n 
		\varphi_1\bigr)\,\rho_1(\mp)\,d\mp \, dt 
		\\ & \quad
		- \int_{\R^m} \F_{\mx}
		\left(\nabla_\mx \varphi_1\cdot 
		\bigl(F^\mp h^\mp_{n} \bigr)\right)\,
		\rho_1(\mp)\, d\mp\, dt
		= \F_{\mx}\left(\bigl(g_{n,\rho_1}
		+G_{n,\rho_1}\bigr)\varphi_1\right)\, dt 
		+\F_{\mx}\bigl(\Phi_{n,\rho_1} \varphi_1\bigr) 
		\, dW_n(t),
	\end{split}
\end{equation*} 
and
\begin{equation*}
	\begin{split}
		&d\int_{\R^m}
		\F_{\mx}\bigl( h^\mq_n \varphi_2\bigr)\,
		\rho_2(\mq)\,  d\mq
		-2\pi i \mxi' \cdot \int_{\R^m}
		\F_{\mx}\bigl(F^{\mq} h^{\mq}_n\varphi_2 \bigr)
		\, \rho_2(\mq)\,d\mq\, dt
		\\ & \quad 
		- \int_{\R^m}\F_{\mx}\left(\nabla_\mx \varphi_2
		\cdot  \bigl(F^\mq h^\mq_n\bigr)\right)
		\,\rho_2(\mq) \, d\mq\, dt
		= \F_{\mx}\left(\bigl(g_{n,\rho_2}
		+G_{n,\rho_2}\bigr) \varphi_2\right)\, dt 
		+\F_{\mx}\bigl(\Phi_{n,\rho_2} \varphi_2\bigr) 
		\, dW_n(t).
	\end{split}
\end{equation*} 
By the It\^o product formula, 
\begin{equation*}
	\begin{split}
		&d \int_{\R^{2m}} 
		\F_{\mx}\bigl(h^\mp_n \varphi_1 \bigr) 
		\, \overline{\F_{\mx}\bigl(h^\mq_n \varphi_2 \bigr)}
		\, \rho_1(\mp) \rho_2(\mq)\,\,d\mp \,d\mq
		\\ & \qquad
		= \int_{\R^{2m}} 2\pi i\mxi' \cdot 
		\F_{\mx}\bigl(F^{\mp} h^{\mp}_n\varphi_1\bigr) 
		\, \overline{\F_{\mx}\bigl(h^\mq_n\varphi_2\bigr)} 
		\, \rho_1(\mp) \rho_2(\mq) \,d\mp \,d\mq \, dt 
		\\ & \qquad\qquad  
		+ \int_{\R^{2m}}\overline{2\pi i \mxi'
		\cdot \F_{\mx}\bigl(F^{\mq} h^{\mq}_n\varphi_2\bigr)}
		\, \F_{\mx}\bigl(h^\mp_n\varphi_1\bigr)
		\, \rho_1(\mp) \rho_2(\mq) \,d\mp\, d\mq \, dt  
		\\ & \qquad\qquad
		-\int_{\R^{2m}} \F_{\mx}\left(\nabla_\mx \varphi_1
		\cdot  \bigl(F^\mp h^\mp_n \bigr)\right) 
		\, \overline{\F_{\mx}\bigl(h^\mq_n \varphi_2 \bigr)}
		\, \rho_1(\mp)\rho_2(\mq) \, d\mp\, d\mq\, dt
		\\ & \qquad\qquad 
		-\int_{\R^{2m}}\overline{\F_{\mx}
		\left(\nabla_\mx \varphi_2\cdot 
		\bigl(F^\mq h^\mq_n  \bigr)\right)} 
		\, \F_{\mx}\bigl(h^\mp_n \varphi_1 \bigr)
		\, \rho_1(\mp)\rho_2(\mq)\, d\mp \,d\mq\, dt
		\\ & \qquad\qquad 
		+\F_{\mx}\bigl(\Phi_{n,\rho_1} 
		\varphi_1\bigr)\,\overline{\F_{\mx}
		\bigl(\Phi_{n,\rho_2} \varphi_2\bigr)} \, dt 
		\\ & \qquad\qquad 
		+ \int_{\R^m}\F_{\mx}
		\left(\bigl(g_{n,\rho_1}+G_{n,\rho_1}\bigr)
		\varphi_1\right)
		\,\overline{\F_{\mx}\bigl(h^\mq_n \varphi_2 \bigr)} 
		\, \rho_2(\mq)\, d\mq\, dt
		\\ & \qquad\qquad 
		+ \int_{\R^m}\overline{\F_{\mx}
		\left(\bigl(g_{n,\rho_2}+G_{n,\rho_2}\bigr)
		\varphi_2\right)}
		\,\F_{\mx}\bigl(h^\mp_n \varphi_1 \bigr) 
		\, \rho_1(\mp)  \, d\mp\, dt
		\\& \qquad \qquad 
		+\int_{\R^{m}} 
		\F_{\mx}\bigl(h^\mp_n\varphi_1\bigr) 
		\, \overline{\F_{\mx}
		\bigl(\Phi_{n,\rho_2} \varphi_2\bigr)}
		\, \rho_1(\mp)\, d\mp\, dW_n(t)
		\\ & \qquad \qquad 
		+\int_{\R^{m}}
		\overline{\F_{\mx}\bigl(h^\mq_n\varphi_2\bigr)}
		\, \F_{\mx}\bigl(\varphi_1\Phi_{n,\rho_1}
		\bigr) \, \rho_2(\mq) \,d\mq \, dW_n(t).
	\end{split}
\end{equation*}
We will manipulate this equation by performing 
the following series of operations:

\smallskip

--- for fixed $\xi_0\in \R$ and $\mxi'\in \R^d$, 
multiply by $\chi_M(\mxi)
\Psi(\mxi)$ where, as before, $\mxi=(\xi_0,\mxi')$, and 
$0\leq\chi_M\leq 1$ is a smooth function, 
$\supp \chi_M \subset B(0,M)^c$, i.e.,
$\chi_M$ is supported out of the ball 
$B(0,M)\subset \R^{d+1}$, and $\chi_M\equiv 1$ 
in $B(0,M+1)^c$, for $M>1$ Moreover,
\begin{equation}\label{eq:Psi-symbol}
	\Psi(\mxi)=\frac{\phi(\xi_0) \psi\bigl(\mxi
	/\abs{\mxi}\bigr)}{\abs{\mxi}}, 
	\qquad \phi\in L^2(\R)\cap L^1(\R), 
	\quad \psi\in C^d(\bS^d);
\end{equation}
	
--- integrate over $\mxi'\in \R^d$; 

--- multiply the resulting equation 
by $e^{-2\pi i t \xi_0} \varphi(t)$, for some 
$\varphi\in C^\infty((0,T))$ with 
$\supp(\varphi)\subset (0,T)$ (i.e., we 
apply the Fourier transform in 
time), then integrate over 
$\xi_0 \in \R$, and then apply $E[\cdot]$. 

\smallskip

To simplify the presentation in what follows, we write 
multiple integrals as a single integral without 
a domain of integration, and use 
the short-hand differential notation
$$
dV_{\omega,t,\mxi}
:=dt \, d\mxi\, dP(\omega).
$$ 
For example,
$$
\int \cdots \,d\mp\, d\mq 
\, dV_{\omega,t,\mxi} 
\quad \text{means}\quad
E\int\limits_{\R^{d+1+2m}}\int\limits_0^T \cdots 
\,dt \,d\mp\, d\mq \, d\mxi.
$$
After performing the steps listed above and utilising the 
simplifying notation, we arrive at 
\begin{equation}\label{weak-sense-1}
	\begin{split}
		& \underbrace{\int 2\pi ie^{-2\pi i t \xi_0} 
		\varphi(t) \, \chi_M(\mxi)\, \Psi(\mxi)
		\,\xi_0\, \F_{\mx}\bigl(h^\mp_n \varphi_1\bigr) 
		\, \overline{\F_{\mx}\bigl(h^\mq_n \varphi_2 \bigr)} 
		\, \rho_1(\mp) \rho_2(\mq) 
		\,d\mp\, d\mq \, dV_{\omega,t,\mxi}}_{=:\cI_{\xi_0}} 
		\\ & \qquad 
		-\underbrace{\int e^{-2\pi i t \xi_0} \varphi'(t) \,
		\chi_M(\xi)\, \Psi(\mxi)
		\, \F_{\mx}\bigl(h^\mp_n\varphi_1\bigr)
		\,\overline{\F_{\mx}\bigl(h^\mq_n\varphi_2\bigr)} 
		\, \rho_1(\mp) \rho_2(\mq)
		\,d\mp\, d\mq \, dV_{\omega,t,\mxi}}_{=:I_0}
		\\ & \quad 
		= \underbrace{\int  e^{-2\pi i t \xi_0} \varphi(t)
		\, 2\pi i \, \chi_M(\mxi)\Psi(\mxi)
		\, \mxi'\cdot 
		\F_{\mx}\bigl(F^{\mp} h^{\mp}_n\varphi_1\bigr)
		\, \overline{\F_{\mx}\bigl(h^\mq_n\varphi_2\bigr)}
		\, \rho_1(\mp)\rho_2(\mq)
		\,d\mp\, d\mq \, dV_{\omega,t,\mxi}}_{=:\cI_{\mxi',1}}
		\\ & \qquad 
		+\underbrace{\int e^{-2\pi i t \xi_0} \varphi(t) 
		\,  \overline{2\pi i \, 
		\chi_M(\mxi)\Psi(\mxi)\, \mxi'
		\cdot \F_{\mx}\bigl(F^{\mq} h^{\mq}_n\varphi_2\bigr)}
		\, \F_{\mx}\bigl(h^\mp_n\varphi_1\bigr)
		\, \rho_1(\mp)\rho_2(\mq)
		\,d\mp\, d\mq \, dV_{\omega,t,\mxi}}_{=:\cI_{\mxi',2}}
		\\ & \qquad 
		-\underbrace{\int e^{-2\pi i t \xi_0} \varphi(t) \, 
		\chi_M(\mxi)\Psi(\mxi)
		\, \F_{\mx}\left(\nabla_\mx \varphi_1
		\cdot \bigl(F^\mp h^\mp_n\bigr)\right) 
		\, \overline{\F_{\mx}\bigl(h^\mq_n\varphi_2\bigr)}
		\, \rho_1(\mp)\rho_2(\mq)
		\,d\mp\, d\mq \, dV_{\omega,t,\mxi}}_{=:I_1}
		\\ & \qquad
		-\underbrace{\int e^{-2\pi i t \xi_0} \varphi(t)
		\, \chi_M(\mxi)\Psi(\mxi) 
		\, \overline{\F_{\mx}\left(\nabla_\mx \varphi_2
		\cdot  \bigl(F^\mq h^\mq_n \bigr)\right)} 
		\, \F_x\bigl(h^\mp_n \varphi_1\bigr)
		\, \rho_1(\mp)\rho_2(\mq)
		\,d\mp\, d\mq \, dV_{\omega,t,\mxi}}_{=:I_2}
		\\ & \qquad 
		+ \underbrace{\int e^{-2\pi i t \xi_0} \varphi(t) 
		\chi_M(\mxi)\Psi(\mxi) 
		\, \F_{\mx}\bigl(\Phi_{n,\rho_1}\varphi_1\bigr)
		\,\overline{\F_{\mx}\bigl(\Phi_{n,\rho_2}
		\varphi_2\bigr)} \, dV_{\omega,t,\mxi}}_{=:I_3} 
		\\ & \qquad
		+\underbrace{\int e^{-2\pi i t \xi_0} \varphi(t) 
		\,\chi_M(\mxi)\Psi(\mxi)
		\, \F_{\mx}\left(\bigl(g_{n,\rho_1}
		+G_{n,\rho_1}\bigr) \varphi_1\right)
		\,\overline{\F_{\mx}\bigl(h^\mq_n\varphi_2\bigr)} 
		\, \rho_2(\mq) \, d\mq 
		\, dV_{\omega,t,\mxi}}_{=:I_{4,1}+I_{4,2}}
		\\ & \qquad
		+ \underbrace{\int e^{-2\pi i t \xi_0} \varphi(t) 
		\, \chi_M(\mxi)\Psi(\mxi)
		\,\overline{\F_{\mx}\left(\bigl(g_{n,\rho_2}
		+G_{n,\rho_2}\bigr) \varphi_2\right)}
		\,\widehat{h^\mp_n \varphi_1}\, \rho_1(\mp) 
		\,d\mp \, dV_{\omega,t,\mxi}}_{=:I_{5,1}+I_{5,2}};
	\end{split}
\end{equation}
or, slightly reorganised,
\begin{equation}\label{weak-sense-2}
	\begin{split}
		& \cI_{\mxi_0}-\cI_{\mxi,1}-\cI_{\mxi,2} 
		= I_0-I_1-I_2+I_3+I_{4,1}+I_{4,2}
		+I_{5,1}+I_{5,2}.
	\end{split}
\end{equation} 

This equation will eventually allow us to use the 
$H$-measures. First, however, let us further 
simplify by estimating the terms 
on the right-hand side of \eqref{weak-sense-2}. 
We claim that 
\begin{equation}\label{<M}
	I_0,I_1,I_2,I_3 \leq 
	\frac{C}{M},
\end{equation} 
for some constants $C, \alpha >0$ independent of $n$.
We will only establish the bound for $I_3$; the 
other three terms in \eqref{<M} 
can be handled in the same way. 
Note that 
$$
\int_{\R}\abs{\chi_M(\mxi)
\Psi(\mxi)}\,d\xi_0
\lesssim \frac{\norm{\psi}_{C(\bS^d)}
\norm{\phi}_{L^1(\R)}}{M}. 
$$
Using this, the Cauchy-Schwarz inequality and 
Plancherel's theorem, we deduce that
\begin{equation*}%\label{I1}
	\begin{split}
		\abs{I_3} & \leq 
		\frac{\norm{\psi}_{C(\bS^d)}
		\norm{\phi}_{L^1(\R)}}{M} 
		\, E\Biggl[\int_0^T \abs{\varphi(t)}
		\, \norm{\F_\mx\bigl(\Phi_{n,\rho_1}
		\varphi_1 \bigr)(t)}_{L^2(\R^{d})} 
		\, \norm{\F_\mx\bigl(\Phi_{n,\rho_2} 
		\varphi_2\bigr)}_{L^2(\R^{d})} \,dt \Biggr] 
		\\ & \leq 
		\frac{\norm{\psi}_{C(\bS^d)}
		\norm{\phi}_{L^1(\R)}
		\norm{\varphi}_{L^\infty(0,T)}}{M}
		E \Biggl[ 
		\int_0^T\norm{\bigl(\Phi_{n,\rho_1}
		\varphi_1\bigr)(t)}_{L^2(\R^d)} 
		\, \norm{\bigl(\Phi_{n,\rho_2} 
		\varphi_2\bigr)(t)}_{L^2(\R^d)}\, dt\Biggr]
		\\ & \lesssim_{\psi,\phi,\varphi} 
		\frac{1}{M} \left(E\left[ \norm{\Phi_{n,\rho_1} 
		\varphi_1}_{L^2((0,T)\times \R^d)}^2\right]
		+E\left[\norm{\Phi_{n,\rho_2}
		\varphi_2}_{L^2((0,T)\times \R^d)}^2\right]\right)
		\\ &
		\lesssim_{\varphi_1,\varphi_2,\rho_1,\rho_2,T}
		\frac{1}{M}, \quad \text{by assumption (e)}. 
	\end{split}
\end{equation*} 

Next, to estimate $I_{4,1}$, let us write
\begin{align*}
	J & := e^{-2\pi i t \xi_0} \varphi(t)
	\,\chi_M(\mxi)\Psi(\mxi)
	\, \F_{\mx}\bigl(g_{n,\rho_1}\varphi_1\bigr)
	\,\overline{\F_{\mx}\bigl(h^\mq_n\varphi_2\bigr)}  
	\\ & = 
	\F_t\circ \F_t^{-1}(\phi)(\xi_0)\, e^{-2\pi i t \xi_0} 
	 \varphi(t) 
	\\ & \qquad \times 
	\F_{\mx}\circ \F_\mx^{-1}\Bigl(
	\chi_M(\mxi)\psi(\cdot)
	\, \overline{\F_{\mx}\bigl(h^\mq_n\varphi_2\bigr)}\Bigr)
	\, \F_{\mx}\circ \F_{\mx}^{-1}
	\left(\frac{1}{\abs{\mxi}}
	\F_{\mx}\bigl(g_{n,\rho_1}\varphi_1\bigr)\right)
	\\ & =  
	\F_t\circ \F_t^{-1}(\phi)(\xi_0)\, e^{-2\pi i t \xi_0} \varphi(t)\,
	\overline{\F_{\mx}\Bigl(\cA_{\chi_M(\cdot,\cdot)\psi(\cdot)}
	\bigl(h^{\mq}_n\varphi_2\bigr)
	\Bigr)}(t,\mxi')\, \F_{\mx}\Bigl(\cT_{-1}
	\bigl(g_{n,\rho_1}\varphi_1\bigr)
	\Bigr)(t,\mxi'),
\end{align*} 
where $\psi(\cdot)=\psi\bigl(\mxi
/\abs{\mxi}\bigr)$. We refer to 
Definition \ref{multiplier} and Theorem \ref{riesz} 
for information about the multiplier operator 
$\cA_{\chi_M(\cdot)\psi(\cdot)}$ 
and the Riesz potential $\cT_{-1}$, noting also 
that the multiplier $\chi_M(\cdot)\psi(\cdot)$ 
is a bounded continuous function. Integrating 
this relation, first with respect to $t$ and then $\xi_0$, 
making use of the Plancherel theorem 
($\int f\, \overline{g}
=\int \widehat{f}\,\overline{\widehat{g}}$) 
in the $\xi_0$-integral, we obtain
$$
\iint J \,dt\,d\tau 
=\int \F_t^{-1}(\phi)(t)\, \varphi(t)\, 
\F_\mx\left(\cA_{\chi_M(\cdot,\cdot)\psi(\cdot)}
\bigl(h^\mq_n\varphi_2\bigr)\right)\!(t,\mxi')
\, \F_\mx\left(\cT_{-1}
\bigl(g_{n,\rho_1}\varphi_1\bigr)\right)\!(t,\mxi')\,dt.
$$
Integrating with respect to $\mxi'$ and again 
using the Plancherel theorem, we arrive at
$$
\iint J \,dt\,d\mxi
=\iint \F_t^{-1}(\phi)(t)\, \varphi(t)
\,\overline{\cA_{\chi_M(\cdot)\psi(\cdot)}
\, \bigl(h^\mq_n\varphi_2 \varphi_1\bigr)}(t,\mx)
\, \cT_{-1} \bigl(g_{n,\rho_1}\bigr)
(t,\mx)\,dt\,d\mx.
$$ 
Inserting this into $I_{4,1}$, 
using H\"older's inequality with 
$\frac{1}{r'}+\frac{1}{r}=1$ and 
$r$ defined in \eqref{gk-a}, and 
the Marcinkiewicz multiplier 
theorem (cf.~Theorem \ref{m1}),
\begin{equation}\label{gk}
	\begin{split}
		\abs{I_{4,1}}& =
		\abs{E\left[\int_0^T \int_{\R^d}\F_t^{-1}(\phi)(t)
		\, \varphi(t)\, 
		\overline{\cA_{\chi_M(\cdot,\cdot)\psi(\cdot)}
		\left(\int_{\R^m} \bigl(h^\mq_n \varphi_2\bigr)
		\,\rho_2(\mq)\, d\mq \right)}
		\cT_{-1}\bigl(g_{n,\rho_1}
		\, \varphi_1\bigr)\,d\mx \,dt\right]}  
		\\ & \leq 
		\norm{\F_t^{-1}(\phi)\varphi}_{L^\infty(0,T)}
		E\Biggl[\norm{\cA_{\chi_M(\cdot,\cdot)\psi(\cdot)}
		\left(\int_{\R^m} \bigl(h^\mq_n \varphi_2\bigr)  
		\, \rho_2(\mq) \, d\mq 
		\right)}_{L^{r'}((0,T)\times \R^d)} 
		\\ & \qquad \qquad \qquad 
		\qquad \qquad \qquad \qquad \times
		\norm{\cT_{-1}\bigl(g_{n,\rho_1} 
		\varphi_1 \bigr)}_{L^r((0,T)\times \R^{d})} \Biggr] 
		\\ & \lesssim_{\varphi,\phi} 
		E\left[\norm{\int_{\R^m} \bigl(h^\mq_n \varphi_2\bigr)  
		\, \rho_2(\mq) \, d\mq}_{L^{r'}((0,T)\times \R^d)} 
		\norm{g_{n,\rho_1}
		\varphi_1}_{W^{-1,r}((0,T)\times \R^d)}\right] 
		\\ &  \lesssim T \norm{\varphi_1}_{L^\infty(\R^d)}
		\norm{\rho_2}_{L^1(\R^m)}\norm{\varphi_2}_{L^{r'}(\R^d)} 
		E\left[\norm{g_{n,\rho_1}}_{W^{-1,r}((0,T)\times K)}\right]
		\\ & \lesssim_{T,\varphi_1,\varphi_2,\rho_2}
		E\left[
		\norm{g_{n,\rho_1}}_{W^{-1,r}((0,T)\times K)}
		\right]
		=: \mathbf{o}_{n\to\infty}(1)
		\overset{(c), \eqref{gk-a}}
		{\longrightarrow} 0 \quad \text{as $n\to \infty$}, 
	\end{split}
\end{equation} 
where we have also used 
that $\norm{h_n}_{L^\infty_{\omega,t,\mx,\mxi}}
\lesssim 1$, cf.~assumption (a). We proceed similarly 
to estimate $I_{4,2}$, this time splitting the multiplier 
$\chi_M\psi(\cdot)/\abs{\mxi}$ into the product of two parts:
$$
\frac{\chi_M(\mxi)\psi(\mxi)}{\abs{\mxi}}
= \frac{\chi_M(\mxi)\psi(\mxi)}{\abs{\mxi}^{1-s}}
\frac{1}{\abs{\mxi}^s},
$$ where $s$ is defined in \eqref{gk-b-1}. 
Repeating the calculations leading to \eqref{gk}, 
we deduce that
\begin{equation*}%\label{gk-part2}
	\begin{split}
		\abs{I_{4,2}} &\lesssim_{\varphi,\phi}
		E\Biggl[\norm{\cA_{\frac{\chi_M(\mxi)
		\psi(\cdot)}{\abs{\mxi}^{1-s}}}
		\left(\int_{\R^m} \bigl(h^\mq_n \varphi_2\bigr)
		\, \rho_2(\mq) \, d\mq 
		\right)}_{L^{r'}((0,T)\times\R^d)} 
		\\ & \qquad \qquad \qquad 
		\qquad \qquad \qquad \qquad \times
		\norm{\cA_{\frac{1}{\abs{\mxi}^s}}
		\bigl(G_{n,\rho_1}\varphi_1
		\bigr)}_{L^r((0,T)\times \R^{d})} \Biggr]
		\\ & \leq E\left[\frac{1}{M^{1-s}}
		\norm{\int_{\R^m} \bigl(h^\mq_n \varphi_2\bigr)  
		\, \rho_2(\mq) \, d\mq}_{L^{r'}((0,T)\times \R^d)}
		\norm{G_{n,\rho_1}
		\varphi_1}_{W^{-s,r}((0,T)\times \R^{d})}\right].
		\\ & \lesssim_{T,\varphi_1,\varphi_2,\rho_2} 
		\frac{1}{M^{1-s}}
		E\left[\norm{G_{n,\rho_1}}_{W^{-s,r}((0,T)\times K)}\right]
		\overset{(d), \eqref{gk-b-1}}{\lesssim}
		\frac{1}{M^{1-s}}.
	\end{split}
\end{equation*}

Similarly, we prove that
\begin{equation*}%\label{gk-part3-4}
	\abs{I_{5,1}} \lesssim 
	\mathbf{o}_{n\to\infty}(1),
	\qquad
	\abs{I_{5,2}} \lesssim 
	\frac{1}{M^{1-s}}.
\end{equation*}

We summarise our findings in a lemma.

\begin{lemma}\label{lem:H-measure-tmp}
With $\cI_{\xi_0}=\cI_{\xi_0}\bigl(\chi_M;n\bigr)$, 
$\cI_{\mxi',1}=\cI_{\mxi',1}\bigl(\chi_M;n\bigr)$, 
and $\cI_{\mxi',2}=\cI_{\mxi',2}\bigl(\chi_M;n\bigr)$ 
defined in \eqref{weak-sense-1}, i.e., 
\begin{align*}
	& \cI_{\xi_0}\bigl(\chi_M;n\bigr) 
	=\int 2\pi ie^{-2\pi i t \xi_0} 
	\varphi(t) \, \chi_M(\mxi)\, \Psi(\mxi)
	\,\xi_0\, \F_{\mx}\bigl(h^\mp_n \varphi_1\bigr) 
	\, \overline{\F_{\mx}\bigl(h^\mq_n \varphi_2 \bigr)} 
	\, \rho_1(\mp) \rho_2(\mq) 
	\,d\mp\, d\mq \, dV_{\omega,t,\mxi},
	\\ & 
	\cI_{\mxi',1}\bigl(\chi_M;n\bigr) =
	\int  e^{-2\pi i t \xi_0} \varphi(t)
	\, 2\pi i \, \chi_M(\mxi)\Psi(\mxi)
	\, \mxi'\cdot 
	\F_{\mx}\bigl(F^{\mp} h^{\mp}_n\varphi_1\bigr)
	\, \overline{\F_{\mx}\bigl(h^\mq_n\varphi_2\bigr)}
	\, \rho_1(\mp)\rho_2(\mq)
	\,d\mp\, d\mq \, dV_{\omega,t,\mxi},
	\\ & 
	\cI_{\mxi',2}\bigl(\chi_M;n\bigr) =
	\int e^{-2\pi i t \xi_0} \varphi(t) 
	\,  \overline{2\pi i \, 
	\chi_M(\mxi)\Psi(\mxi)\, \mxi' 
	\cdot \F_{\mx}\bigl(F^{\mq} h^{\mq}_n\varphi_2\bigr)}
	\, \F_{\mx}\bigl(h^\mp_n\varphi_1\bigr)
	\, \rho_1(\mp)\rho_2(\mq)
	\,d\mp\, d\mq \, dV_{\omega,t,\mxi},
\end{align*}
we have
$$
\abs{\cI_{\xi_0}\bigl(\chi_M;n\bigr)
-\cI_{\mxi,1}\bigl(\chi_M;n\bigr)
-\cI_{\mxi,2}\bigl(\chi_M;n\bigr)}\lesssim 
\mathbf{o}_{n\to\infty}(1)+\frac{1}{M^{1-s}},
$$
where $\mathbf{o}_{n\to\infty}(1)\to 0$ 
as $n\to \infty$, uniformly in $M\to \infty$, 
and $s\in (0,1)$ is fixed in \eqref{gk-b-1}. 
\end{lemma}

Given Lemma \ref{lem:H-measure-tmp}, we 
are going to derive strong compactness 
of the velocity averages $\inn{h_n,\rho}$ 
using the $H$-measure defined 
in Theorem \ref{thm:H-measure} and the non-degeneracy 
condition \eqref{non-deg}. The following theorem 
is the main result of this section. 

\begin{theorem}[stochastic velocity averaging on $\R^d$]\label{thm-1} 
Let $\Seq{h_n}$ be a sequence of weak solutions to 
the kinetic SPDE \eqref{eq-1} in the sense of \eqref{weak-sense}. 
Suppose the assumptions $(a)$---$(e)$ 
and the non-degeneracy condition \eqref{non-deg} hold. 
There exists a (not relabeled) 
subsequence $\Seq{h_n}$ such that
\begin{equation}\label{res-1-new}
	E\left[\norm{h_n}_{L^2_{\loc}((0,T)\times \R^d;
	L^2_{\loc,w}(\R^m))}^2\right]\ton 0,
\end{equation}
i.e.,
$$
E\left[ \int_0^T\int_K\abs{\int_L h_n\rho \, d\lambda}^2
\, d\mx \, dt \right]\ton 0, 
\qquad \forall \rho\in L^2(L), 
\quad \forall L\subset\subset \R^m, 
\quad \forall K \subset\subset \R^d.
$$
\end{theorem}

\begin{proof}
By assumption (a), there exists a non-relabeled 
subsequence $\Seq{h_n}$ such that
$$
h_{n}(\omega,t,\cdot,\cdot) 
\overset{n \uparrow 0}{\longrightarrow} 0 
\quad\text{weakly in $L^2_{\loc}(\R^{d+m})$, 
for a.e.~$(\omega,t)\in \Omega\times (0,T)$}.
$$
By properties of the Fourier transform, this 
implies that the following convergences hold a.s.:
$$
\int_{\R^m} \F_{\mx}\bigl(h^\mp_n\varphi_1\bigr)(t,\mxi') 
\, \rho_1(\mp)\, d\mp \ton 0,
\quad 
\int_{\R^m}\F_{\mx}\bigl(h^\mq_n\varphi_2\bigr)
(t,\mxi')\, \rho_2(\mq)\, d\mq \ton 0,
$$ 
for a.e.~$\mxi'\in \R^d$, $t\in (0,T)$. 
By (a), $\norm{h_k}_{L^\infty}\lesssim 1$ and so
$\abs{\F_{\mx}\bigl(h^\mp_n\varphi_1\bigr)(t,\mxi)}
\leq \norm{\bigl(h^\mp_n\varphi_1\bigr)
(t,\cdot)}_{L^1(\R^d)}\lesssim 1$.  
Thus, applying the dominated convergence 
theorem in $(t,\mxi)$, 
\begin{align*}
	\cI_{\xi_0}\bigl(1-\chi_M;n\bigr)\ton 0, \quad
	\cI_{\mxi',1}\bigl(1-\chi_M;n\bigr)\ton,\quad
	\cI_{\mxi',2}\bigl(1-\chi_M;n\bigr)\ton 0,
\end{align*}
for each fixed $M>1$. Replacing $\chi_M$ by the 
constant $1$, we will simplify the 
notation as follows: 
$$
\cI_{\xi_0}(n):= \cI_{\xi_0}\bigl(1;n\bigr),
\quad
\cI_{\mxi',1}(n):=\cI_{\mxi',1}\bigl(1;n\bigr),
\quad
\cI_{\mxi',2}(n):=\cI_{\mxi',2}\bigl(1;n\bigr).
$$
It thus follows that, for any fixed $M>1$, 
$$
\lim_{n\to\infty}\cI_{\xi_0}(n)=
\lim_{n\to\infty}\cI_{\xi_0}\bigl(\chi_M;n\bigr)
+\lim_{n\to\infty}\cI_{\xi_0}\bigl(1-\chi_M;n\bigr)
=\lim_{n\to\infty}\cI_{\xi_0}\bigl(\chi_M;n\bigr)
$$
Similarly,
\begin{align*}
	\lim_{n\to\infty}\cI_{\mxi',1}(n)
	=\lim_{n\to\infty}\cI_{\mxi',1}\bigl(\chi_M;n\bigr),
	\qquad
	\lim_{n\to\infty}\cI_{\mxi',2}(n)
	=\lim_{n\to\infty}\cI_{\mxi',2}\bigl(\chi_M;n\bigr).
\end{align*}
Thanks to Lemma \ref{lem:H-measure-tmp}, this 
implies that  
$$
\lim_{n\to \infty}
\abs{\cI_{\xi_0}(n)-\cI_{\mxi',1}(n)
-\cI_{\mxi',2}(n)}\lesssim \frac{1}{M^{1-s}}.
$$
Since the right-hand side can be made arbitrary small 
by taking $M$ sufficiently large, we conclude that 
\begin{equation}
\label{localization-prep}
\lim_{n\to \infty}\abs{\cI_{\xi_0}(n)
-\cI_{\mxi',1}(n)-\cI_{\mxi',2}(n)}=0.
\end{equation}

Recall that $\varphi$ is a $C^\infty_c(0,T)$ function 
and $\phi\in L^1(\R)\cap L^2(\R)$, see \eqref{eq:Psi-symbol}. 
Clearly, the sequence $\Seq{\sqrt{\varphi}\, h_n}$ 
satisfies the conditions of Corollary \ref{cor:extension-1}, and 
so there exists a subsequence---still denoted by 
$\Seq{\sqrt{\varphi} \, h_n}$---and a 
corresponding $H$-measure $\mu_\phi$ such that
\begin{align*}
	& \lim_{n\to \infty}\cI_{\xi_0}(n)
	\\ & \quad 
	=\lim_{n\to \infty} \int 2\pi ie^{-2\pi i t \xi_0}
	\varphi(t) \, \chi_M(\mxi)\, \Psi(\mxi)
	\,\xi_0\, \F_{\mx}\bigl(h^\mp_n \varphi_1\bigr) 
	\, \overline{\F_{\mx}\bigl(h^\mq_n \varphi_2 \bigr)} 
	\, \rho_1(\mp) \rho_2(\mq) 
	\,d\mp\, d\mq \, dV_{\omega,t,\mxi}
	\\ & \quad
	= \lim_{n\to \infty}
	\, \int\, \phi(\xi_0)\psi\left(\frac{\mxi}{\abs{\mxi}}\right)
	\,\frac{\xi_0}{\abs{\mxi}}
	\, \F_t\left(\F_{\mx}\bigl(h^\mp_n\sqrt{\varphi}\varphi_1\bigr) 
	\, \overline{\F_{\mx}\bigl(h^\mq_n\sqrt{\varphi}\varphi_2 \bigr)}
	\right)\, \rho_1(\mp) \rho_2(\mq)
	\,d\mp\, d\mq \, dV_{\omega,\xi_0,\mxi}
	\\ & \quad 
	=\int_{\R^{2m}}\Bigl\langle \mu_\phi(\mp,\mq,\cdot,\cdot),
	\, \varphi_1 \varphi_2 \rho_1(\mp)\rho_2(\mq)
	\otimes \xi_0\psi \Bigr\rangle \,d\mp\, d\mq,
\end{align*}
\begin{align*}
	& \lim_{n\to \infty}\cI_{\mxi',1}(n)
	\\ & \quad 
	=\lim_{n\to \infty}\int 2\pi i e^{-2\pi i t \xi_0} \varphi(t)
	\, \chi_M(\mxi)\Psi(\mxi)\, \mxi'\cdot 
	\F_{\mx}\bigl(F^{\mp} h^{\mp}_n\varphi_1\bigr)
	\, \overline{\F_{\mx}\bigl(h^\mq_n\varphi_2\bigr)}
	\, \rho_1(\mp)\rho_2(\mq)\,d\mp\, d\mq \, dV_{\omega,t,\mxi}
	\\ & \quad
	=\lim_{n\to \infty}\int \phi(\xi_0)
	\psi\left(\frac{\mxi}{\abs{\mxi}}\right)
	\, \F_t\left(\mxi'\cdot 
	\F_{\mx}\bigl(F^{\mp} h^{\mp}_n \sqrt{\varphi} 
	\varphi_1\bigr)
	\, \overline{\F_{\mx}\bigl(h^\mq_n \sqrt{\varphi}
	\varphi_2\bigr)}\right)\, \rho_1(\mp)\rho_2(\mq)
	\,d\mp\, d\mq \, dV_{\omega,\xi_0,\mxi}
	\\ & \quad 
	=\int_{\R^{2m}}\Bigl\langle \mu_\phi(\mp,\mq,\cdot,\cdot), 
	\, \varphi_1\varphi_2\rho_1(\mp)\rho_2(\mq)
	\otimes \mxi' \cdot F^\mp \,\psi \Bigr\rangle
	\,d\mp\, d\mq,
\end{align*}
and
\begin{align*}
	& \lim_{n\to \infty}\cI_{\mxi',2}(n)
	\\ & \quad 
	=\lim_{n\to \infty} \int e^{-2\pi i t \xi_0} \varphi(t)
	\,  \overline{2\pi i \, \chi_M(\mxi)\Psi(\mxi)\, \mxi'
	\cdot \F_{\mx}\bigl(F^{\mq} h^{\mq}_n \varphi_2\bigr)}
	\, \F_{\mx}\bigl(h^\mp_n \varphi_1\bigr)
	\, \rho_1(\mp)\rho_2(\mq)
	\,d\mp\, d\mq \, dV_{\omega,t,\mxi}
	\\ & \quad
	=\lim_{n\to \infty}\int\phi(\xi_0) 
	\, \F_t \left(\, \overline{2\pi i \, 
	\psi\left(\frac{\mxi}{\abs{\mxi}}\right)\, \mxi'
	\cdot \F_{\mx}\bigl(F^{\mq} h^{\mq}_n 
	\sqrt{\varphi} \varphi_2\bigr)}
	\, \F_{\mx}\bigl(h^\mp_n \sqrt{\varphi} 
	\varphi_1\bigr) \right)
	\, \rho_1(\mp)\rho_2(\mq)
	\,d\mp\, d\mq \, dV_{\omega,t,\mxi}
	\\ & \quad 
	=\int_{\R^{2m}}
	\Bigl\langle \overline{\mu_\phi(\mp,\mq,\cdot,\cdot)}, 
	\, \varphi_1 {\varphi_2} \rho_1(\mp)\rho_2(\mq)
	\otimes \mxi' \cdot F^\mq \,\psi  \Bigr\rangle
	\,d\mp\, d\mq.
\end{align*}
Inserting this into \eqref{localization-prep}, 
keeping in mind that $\mu_\phi$ can be represented 
by a measure (see Corollary \ref{cor:extension-1}), we 
obtain the identity
\begin{equation}\label{localization1}
	\varphi_1 \varphi_2 \otimes \rho_1(\mp)\rho_2(\mq) 
	\otimes \psi\left((\xi_0+F^\mp \cdot \mxi')\, d\mu_\phi
	+F^\mp \cdot \mxi' \, d\bar{\mu}_\phi \right)=0,
\end{equation}
where $\bar{\mu}_\phi$ denotes the conjugate of the complex 
measure $\mu_\phi$. If we write $\mu_\phi=\mu_1+i \mu_2$ and 
use the arbitrariness of $\varphi$, $\varphi_1$, $\varphi_2$, 
$\rho_1$, $\rho_2$, $\psi$, we 
conclude from \eqref{localization1} that
\begin{equation}\label{localization2}
	\begin{split}
		& \left(\xi_0+\left(F^\mp+F^\mq\right)\cdot \mxi'\right)\, d\mu_1=0,
		\\ & 
		\left(\xi_0+\left(F^\mp-F^\mq\right)\cdot \mxi'\right)\, d\mu_2=0.
	\end{split}
\end{equation}

Finally, we will prove that \eqref{localization2} implies
$\mu_1\equiv \mu_2\equiv 0$, i.e., $ \mu \equiv 0$. By 
Corollary \ref{cor:extension-1}, $\mu_\varphi = f\, d\nu$, 
cf.~\eqref{representation}, where $f$ is complex. Accordingly, 
let us write $\mu_i=f_i(\mp,\mq,\mx,\mxi)\, d\nu(\mx,\mxi)$, 
where $\nu \in L_{\operatorname{loc}}^1(\R^d;\cM(\bS^d))$ 
and $f_i\in L^2(\R^{2m};L^1_{\nu}(\R^d\times \bS^d))$, $i=1,2$. 
We repeat the procedure from \cite[page 253]{LM2}. 
To this end, set
$$
A(\mp,\mq,\mx,\mxi)=\xi_0
+\left(F^\mp+F^\mq\right)\cdot \mxi',
$$ 
and test \eqref{localization2} against
$$
\rho(\mp)\rho(\mq)\varphi(\mx)
\frac{A(\mp,\mq,\mx,\mxi)}{\abs{A(\mp,\mq,\mx,\mxi)}^2+\delta},
$$ 
for some fixed smooth functions $\rho$ and $\varphi$. 
Taking into account Corollary \ref{cor:extension-1}, 
\begin{equation}\label{loc-3}
	\int_{\R^d \times \bS^d} \int_{\R^{2m}}
	\rho(\mp)\rho(\mq) \varphi(\mx)
	\frac{\abs{A(\mp,\mq,\mx,\mxi)}^2}{\abs{A(\mp,\mq,\mx,\mxi)}^2+\delta} 
	\, f_1(\mp,\mq,\mx,\mxi)\, d\mp \, d\mq\, d\nu(\mx,\mxi)=0,
\end{equation}
where the measure $\nu$ is regular 
with respect to $\mx\in \R^d$. From the non-degeneracy 
condition \eqref{non-deg}, 
$$
\mp\mapsto \abs{A(\mp,\mq,\mx,\mxi)}^2>0
\quad \text{a.e.},
$$ 
for a.e.~$\mq\in \R^m$ and 
$\nu$-a.e.~$(\mx,\mxi) \in \R^d\times \bS^d$. Therefore, 
for a.e.~$\mq\in \R^m$ and 
$\nu$-a.e.~$(\mx,\mxi) \in \R^d\times \bS^d$,
$$
\int_{\R^{m}} \rho(\mp)\rho(\mq) \varphi(\mx)
\frac{\abs{A(\mp,\mq,\mx,\mxi)}^2}{\abs{A(\mp,\mq,\mx,\mxi)}^2+\delta}
\, f_1(\mp,\mq,\mx,\mxi)\, d\mp 
\to \int_{\R^{2m}} \rho(\mp)\rho(\mq) 
\varphi(\mx) \, f_1(\mp,\mq,\mx,\mxi)\, d\mp=0,
$$
as $\delta\to 0$. By the dominated convergence theorem, sending 
$\delta \to 0$ in \eqref{loc-3} yields
\begin{equation*}
	\int_{\R^d \times \bS^d} \int_{\R^{2m}} 
	\rho(\mp)\rho(\mq) \varphi(\mx)
	\, f_1(\mp,\mq,\mx,\mxi)\, d\mp 
	\, d\mq \, d\nu(\mx,\mxi)=0,
\end{equation*}
which implies that $\mu_1=0$. Similarly, putting 
$A(\mp,\mq,\mx,\mxi)=\xi_0-\left(F^\mp +F^\mq\right)
\cdot \mxi$ and repeating the previous 
procedure, we arrive at the conclusion $\mu_2=0$.

The vanishing of the $H$-measure $\mu_\varphi$ 
immediately implies the strong convergence
$$
\int_{\R^m} \sqrt{\varphi} h_n^\mp \rho(\mp) \,d\mp
\ton 0 \quad 
\text{in $L^2\bigl(\Omega;L^2_{\loc}((0,T)\times \R^d)\bigr)$},
$$ 
along a subsequence, a fact that 
follows by specifying in \eqref{rev11-1} the functions 
$\psi\equiv 1$, $\phi=\overline{\F_t^{-1}(\varphi)}$,
$\phi_1=\varphi_1$, $\phi_2=\varphi_2$, and $\varphi_1
=\varphi_2=\chi_K(\mx)\rho(\lambda)$ (after an 
approximation argument), $K\subset\subset \R^d$, and $\rho(\lambda)
\in L^2_{\loc}(\R^m)$, see Remark \ref{rem-all}. 
By the arbitrariness of $\varphi$, $K$ and $\rho$, the 
convergence statement \eqref{res-1-new} holds.
\end{proof} 

\begin{remark}
It is enough that conditions (c) 
and (d) from Theorem \ref{thm-1} hold 
for $\rho$ from a countable 
dense subset $Q$ of $C^\infty_c(\R^m)$.
\end{remark}

We will now remove the zero convergence assumption 
from the last theorem. We start with the following crucial 
lemma, which shows that we can pass 
to the limit in the SPDE \eqref{eq-1} for $h_n$ 
under rather weak convergence assumptions.

\begin{lemma}[limit kinetic SPDE]
\label{lem:weak-limit-kinetic}
Suppose the following conditions hold:

\noindent (i) for each $n\in \N$, $h_n$ satisfies 
the $(t,\mx,\lambda)$-weak formulation \eqref{weak-sense-tx}, 
$\norm{h_n}_{L^\infty_{\omega,t,\mx,\lambda}}\lesssim 1$, and 
$$
h_n\ton h \quad 
\text{in $L^2_t\bigl(L^2_{\loc,w}(\R^d\times \R^m)\bigr)$, a.s.};
$$

\noindent (ii) for each $\rho=\rho(\lambda)\in C^\infty_c(\R^m)$, 
there exists a random variable $g_{\rho}$
such that
\begin{equation*}
	 g_{n,\rho}  \ton g_\rho \quad 
	 \text{in $L^2(0,T;W^{-1,r}_{\loc}(\R^d))$, a.s.}, 
\end{equation*} 
for some $r>1$, and 
$$
E\left[\norm{g_{n,\rho}}_{L^2(0,T;W^{-1,r}(K))}^2\right]
\lesssim_{T,K} 1, \quad K\subset \subset \R^d;
$$

\noindent (iii) for each $\rho=\rho(\lambda) \in C^\infty_c(\R^m)$ 
there exists a random variable $G_{\rho}$
such that
\begin{equation*}
	G_{n,\rho} \overset{n \uparrow \infty}{\dscon}
	G_\rho \quad \text{in $\cM_{\loc}((0,T)\times \R^d)$, a.s.},
\end{equation*}
and 
$$
E\left[\norm{G_{n,\rho}}_{\cM((0,T)\times K)}^2\right]
\lesssim_{T,K} 1, \quad K\subset \subset \R^d;
$$ 
 
\noindent (iv) Each $\Phi_{n,\rho}$ is adapted and there exists 
$\Phi_\rho$ such that
$$
\Phi_{n,\rho} \overset{n \uparrow \infty}{\dscon}
\Phi_\rho \quad 
\text{in $L^2_t\bigl(L^2_{\loc,w}(\R^d)\bigr)$, a.s.},
$$
and
$$
E\left[\norm{\Phi_{n,\rho}}_{L^2((0,T)\times K)}^2\right]
\lesssim_{T,K} 1, 
\quad K\subset \subset \R^d.
$$ 
Finally, $W_n$ converges to a Wiener process $W$ in $C([0,T])$, a.s., 
where $W$ and $\Phi_\rho$ are defined on a filtered space 
$\bigl(\Omega,\F,\prob,\Seq{\F_t}\bigr)$.

\medskip

Then, almost surely, for all $(\varphi,\rho)\in 
C^\infty_c([0,T)\times \R^d)\times C^{\infty}_c(\R^m)$,
\begin{equation}\label{weak-sense-tx-limit}
	\begin{split}
		&\int_0^T\int_{\R^{d+m}} h \pa_t \varphi
		\rho \, d\lambda \, d\mx \, dt
		+\int_{\R^{d+m}}  h_0(\mx,\lambda) 
		\varphi(0,\mx)\rho(\lambda)  \, d\lambda \, d\mx
		\\ & \qquad +\int_0^T\int_{\R^{d+m}} h F(\mx,\lambda)
		\cdot \nabla_{\mx}\varphi \rho \, d\lambda\, d\mx\,dt
		= m_{\rho}(\varphi)
		+\int_0^T\int_{\R^{d+m}} \Phi_{\rho}
		\varphi\,d\lambda \,d\mx \, dW(t),
	\end{split}	
\end{equation} 
where $m_\rho$ is defined similarly 
to \eqref{eq:measure-m-def}, that is, 
$m_{\rho}(\phi)\bigl([t_1,t_2]\bigr)
=m_{\rho}\bigl(\phi\chi_{[t_1,t_2]}\bigr)$ and
\begin{equation}\label{eq:measure-m-def-limit}
	m_{\rho}(\varphi)
	:=\bigl\langle g_{\rho}+G_{\rho},\varphi\bigr\rangle
	=\int_{0}^{T} \int_{\R^d}
	\bigl(g_{\rho}+G_{\rho} \bigr) 
	\varphi(t,\mx)\, d\mx \,dt,
\end{equation}
for any $\varphi=\varphi(t,\mx)$ for which the 
right-hand side of \eqref{eq:measure-m-def-limit} makes sense.
\end{lemma}

\begin{remark}
The assumptions (i)--(iv) are tailored 
to the kinetic SPDE \eqref{eq:tilde-kinetic} linked to the 
pseudo-parabolic SPDE \eqref{PP-1}, see 
\eqref{eq:tilde-conv}, \eqref{eq:tilde-kinetic-est}, 
and \eqref{eq:tilde-sign-functions}. However, to deduce the limit equation 
\eqref{weak-sense-tx-limit}, the first assumption 
can be relaxed. Indeed, (i) may be replaced 
by weak $L^2$ convergence in all the
variables $t,\mx,\lambda$ (strong $L^2$ convergence 
in $t$ is not needed).
\end{remark}

\begin{remark}
Recall that the source terms $g_{n,\rho}$, $G_{n,\rho}$ 
have the structure \eqref{notation-rho}, which 
is motivated by the kinetic SPDE \eqref{eq:tilde-kinetic}. 
The convergences supplied by the forthcoming Lemma \ref{lem:as-repr} 
imply that the limits $g_{\rho}$, $G_{\rho}$ 
exhibit the same structure, i.e.,
\begin{align*}
	g_{\rho}=(-1)^{m_1} \int_{\R^m} \overline{g}
	\frac{\partial^{m_1}\rho}{\pa\lambda^{m_1}} \, d\lambda,
	\quad m_1\in \N,
	\qquad 
	G_{\rho}=(-1)^{m_2} \int_{\R^m} \overline{G}
	\frac{\partial^{m_2}\rho}{\pa\lambda^{m_2}}\, d\lambda,
	\quad m_2\in \N,
\end{align*}
for some distribution-valued random 
variables $\overline{g}$, $\overline{G}$. Nevertheless, 
this structure is not needed in this section, see also 
the related Remark \ref{rem:source-structure}.
\end{remark}

\begin{proof}
Fix test functions $(\varphi,\rho)\in C^\infty_c([0,T)\times \R^d) 
\times C^{\infty}_c(\R^m)$, and let us express 
the $(t,\mx,\lambda)$-weak formulation 
\eqref{weak-sense-tx} symbolically as 
$$
\cI_n(\omega)=0, 
\quad \text{for $\prob$-a.e.~$\omega\in \Omega$}.
$$
Denote by $\cI=\cI(\omega)$ the (limit) random variable
\begin{equation}\label{weak-sense-tx-limit-Idef}
	\begin{split}
		\cI(\omega) & =\int_0^T\int_{\R^{d+m}} h \pa_t \varphi
		\rho \, d\lambda \, d\mx \, dt
		+\int_{\R^{d+m}}  h_0(\mx,\lambda) 
		\varphi(0,\mx)\rho(\lambda)  \, d\lambda \, d\mx
		\\ & \qquad 
		+\int_0^T\int_{\R^{d+m}} h F(\mx,\lambda)
		\cdot \nabla_{\mx}\varphi \rho \, d\lambda\, d\mx\,dt
		\\ & \qquad\qquad 
		-m_{\rho}(\varphi)
		-\int_0^T\int_{\R^{d}} \Phi_{\rho}
		\varphi \,d\mx \, dW(t), 
		\quad \text{$\omega\in \Omega$},
	\end{split}	
\end{equation}
where $m_{\rho}(\varphi)$ is defined by 
\eqref{eq:measure-m-def-limit}. The random variable 
$\cI$ depends on the test functions $\varphi,\rho$. 
Whenever we need to make that dependency explicit, 
we write $\cI_{\varphi,\rho}=\cI_{\varphi,\rho}(\omega)$.

Given Lemma \ref{ito-eq}, it is enough 
to prove that $\cI=0$ a.s. To this end, we will 
verify that
\begin{equation}\label{eq:lim-IminusIn}
	\lim_{n\to \infty} \int_{\Omega} \chi_A(\omega) 
	\bigl(\cI(\omega)-\cI_n(\omega)\bigr)\, d\prob(\omega)=0, 
\end{equation}
for any measurable set $A\in \F$. Without 
loss of generality, we assume that $\varphi\rho$ 
is supported in $(0,T)\times K\times L$, for 
some $K\subset\subset \R^d$, $L\subset\subset \R^m$.  

By assumption (i) 
$$
\int_{(0,T)\times \R^{d+m}}  h_n \partial_t \varphi \rho 
\, d\lambda\, d\mx\, dt
\ton
\int_{(0,T)\times \R^{d+m}}  h \partial_t \varphi \rho 
\, d\lambda\, d\mx\,dt, 
\quad \text{a.s.},
$$
and, by Vitali's convergence theorem, there 
is convergence also in $L^1(\Omega)$. 
In particular,
\begin{align*}
	&\lim_{n\to \infty}
	\int_{\Omega} \chi_A(\omega) 
	\int_{(0,T)\times \R^{d+m}} h_n \partial_t\varphi \rho 
	\, d\lambda \, d\mx \, dt\, d\prob(\omega)
	\\ & \qquad =\int_{\Omega} \chi_A(\omega) 
	\int_{(0,T)\times \R^{d+m}} h\partial_t\varphi \rho 
	\, d\lambda \, d\mx \, dt\, d\prob(\omega).
\end{align*}
Similarly, using assumption (b), which 
implies that $F\cdot \nabla_{\mx}\varphi \rho
\in L^2((0,T)\times K\times L)$,
\begin{align*}
	& \lim_{n\to \infty}
	\int_{\Omega} \chi_A(\omega) 
	\int_{(0,T)\times \R^{d+m}} 
	h_n F(\mx,\lambda)\cdot \nabla_{\mx}\varphi \rho 
	\, d\lambda \, d\mx \, dt\, d\prob(\omega)
	\\ & \qquad 
	=\int_{\Omega} \chi_A(\omega) 
	\int_{(0,T)\times \R^{d+m}} 
	h F(\mx,\lambda)\cdot \nabla_{\mx}\varphi \rho 
	\, d\lambda \, d\mx \, dt\, d\prob(\omega).
\end{align*}

Next, by assumption (iii),
$$
\action{G_{n,\rho}}{\varphi}
\ton \action{G_{\rho}}{\varphi} 
\,\, \text{a.s.},
\quad 
E\left[\abs{\action{G_{n,\rho}}{\varphi}}^2\right]
\lesssim_{T,K,\varphi} 1.
$$
Thus, by the Vitali convergence theorem,
$\action{G_{n,\rho}}{\varphi}\ton 
\action{G_{\rho}}{\varphi}$ in $L^1(\Omega)$. 
Similarly, we can prove that 
$\action{g_{n,\rho}}{\varphi}
\ton \action{g_{\rho}}{\varphi}$ in $L^1(\Omega)$. 
Combining last two convergences, 
$$
m_{n,\rho}(\varphi)\ton
m_{\rho}(\varphi)\quad 
\text{in $L^1(\Omega)$},
$$
which, in particular, implies $\int_{\Omega} 
\chi_A m_{n,\rho}(\varphi)\, d\prob\ton 
\int_{\Omega} \chi_A m_{\rho}(\varphi)\, d\prob$.

To conclude that \eqref{eq:lim-IminusIn} holds, it remains 
to prove convergence of the random variable $M_n(T)$, where
$$
M_n(t):=\int_0^t\int_{\R^d} 
\Phi_{n,\rho}\varphi \,d\mx \, dW_n(t).
$$
This follows from assumption (i)
and \cite[Lemma 2.1]{Debussche:2011aa}.

\medskip

Summarising, we have proved that $\cI_{\varphi,\rho}=0$, 
a.s., for arbitrary but fixed test functions 
$\varphi\in C^\infty_c([0,T)\times \R^d)$ 
and $\rho \in C^\infty_c(\R^m)$, where 
$\cI=\cI_{\varphi,\rho}$ is defined by \eqref{weak-sense-tx-limit-Idef}. 
Using the separability of $C^\infty_c$, the 
exceptional set may be chosen independent of 
the particular test functions $\varphi$, $\rho$.
\end{proof}

One important implication of Lemma \ref{lem:weak-limit-kinetic} is 
that we can remove the zero convergence assumption (a) of Theorem \ref{thm-1}, 
cf.~\eqref{eq:hk-L2t-L2weak}.

\begin{corollary}[stochastic velocity averaging on $\R^d$, 
non-zero limit]\label{cor:velocity-average}
Under the assumptions (i)--(iv) of 
Lemma \ref{lem:weak-limit-kinetic}, (b) of Theorem \ref{thm-1}, 
and the non-degeneracy condition \eqref{non-deg}, we have
\begin{equation}\label{res-1-nonzero}
 	E \left[\norm{h_n-h}_{L^2_{\loc}((0,T)\times \R^d;
 	L^2_{\loc,w}(\R^m))}^2\right]\ton 0,
\end{equation}
where $h$ denotes the limit from (i), that is,
$$
E\left[\int_0^T\int_K\abs{\int_L \bigl(h_n-h\bigr)\rho 
\, d\lambda}^2\, d\mx \, dt \right]\ton 0, 
\qquad \forall \rho\in L^2(L), 
\quad \forall L\subset\subset \R,
\quad \forall K\subset\subset \R^d.
$$
\end{corollary}

\begin{proof}
Recall that $h_n$ is a weak solution of \eqref{eq-1} in 
the sense of \eqref{weak-sense-tx}. By Lemma \ref{lem:weak-limit-kinetic}, 
the limit $h$ satisfies \eqref{weak-sense-tx-limit}. Consequently, 
the difference $\bar h_n:=h_n-h$ is a $(t,\mx)$-weak solution of 
\begin{equation}\label{eq:bar-h-eqn}
	d \action{\bar h_n}{\rho} 
	+\Div_\mx \bigl(\action{ F(\mx,\cdot)\bar h_n}{\rho}\bigr)\, dt
	= \bigl(\bar{g}_{n,\rho}+\bar G_{n,\rho}\bigr)\, dt
	+M_{n,\rho},
\end{equation}
where 
$\bar g_{n,\rho}:=g_{n,\rho}-g_{\rho}$, 
$\bar G_{n,\rho}:=G_{n,\rho}-G_{\rho}$, and 
$$
M_{n,\rho}(t)=\Phi_{n,\rho} \, dW_{n}(t)-\Phi_\rho \, dW(t).
$$
The SPDE \eqref{eq:bar-h-eqn} is essentially of 
the form \eqref{eq-1}, where the conditions {\em (i)}, 
{\em (ii)}, {\em (iii)} of Lemma \ref{lem:weak-limit-kinetic} imply 
conditions (a), (c), (d) of Theorem \ref{thm-1} for 
$\bar h_n$, $\bar{g}_{n,\rho}$, and $\bar G_{n,\rho}$, 
respectively. The difference in form lies in the appearance of two martingale 
terms in $M_n(t)$, whereas \eqref{eq-1} 
contains one such term. However, this does 
not affect the validity of Theorem \ref{thm-1}, which only
relies on the martingale property of the stochastic integral 
and condition (e) ($\Phi_{n,\rho}$, $\Phi_{\rho}$ 
satisfy this condition by {\em (iv)}). 
As a result, Theorem \ref{thm-1} yields the 
claim \eqref{res-1-nonzero}.
\end{proof}

It is elementary to amend the proof of
Corollary \ref{cor:velocity-average} so that it 
applies to a compact Riemannian manifold.

\begin{corollary}[stochastic velocity averaging on $M$]
\label{cor:velocity-average-M}
Suppose the kinetic SPDE \eqref{eq-1} is defined on a 
compact Riemannian manifold $M$, and that the assumptions 
of Corollary \ref{cor:velocity-average} hold in each chart from a 
finite collection of charts covering $M$. Then there exists 
a subsequence $\Seq{h_n}$ (not relabelled) such that  
\begin{equation*}%\label{res-1-nonzero}
	E\left[\norm{h_n-h}_{L^2((0,T)\times M;
	L^2_{\loc,w}(\R^m))}^2\right]\ton 0,
\end{equation*}
where $h$ denotes the limit from the assumption (i) of 
Lemma \ref{lem:weak-limit-kinetic}, that is,
$$
E\left[\int_0^T\int_M\abs{\int_L \bigl(h_n-h\bigr)\rho 
\, d\lambda}^2\, dV(\mx) \, dt\right]\ton 0, 
\qquad \forall \rho\in L^2(L), 
\quad \forall L\subset\subset \R.
$$
\end{corollary} 

\begin{proof}
We fix a finite atlas $\Seq{(M_i,\alpha_i)}_{i=1}^I$ 
covering the manifold $M$. From the discussion at the 
beginning of this section, the sequences 
$\Seq{h_n\circ \alpha_i}_{n\in \N}$, $i=1,\ldots,I$, satisfy 
kinetic SPDEs of the type \eqref{eq-1}. 
By Corollary \ref{cor:velocity-average} and a diagonal argument, 
there is a subsequence $\Seq{h_n}_{n\in\N}$ (not relabelled) 
such that
\begin{equation}\label{local-cv}
	\lim_{n\to \infty}
	E\left[\norm{\int_{L} \bigl(h_n\circ\alpha_i
	-h\circ\alpha_i\bigr)\rho\,d\lambda}^2_{L^2((0,T)
	\times \alpha_i^{-1}(M_i))}\right]=0,
	\quad  \forall \rho\in L^2_{\loc}(\R^m),
\end{equation}  
for all $i=1,\ldots,I$ and $L\subset\subset \R^m$. 
Using a partition of unity $\Seq{\zeta_i}$ 
subordinate to the atlas $\Seq{(M_i,\alpha_i)}$, 
\begin{align*}
	&\lim_{n\to \infty}
	E\left[\int_0^T\int_{M}\abs{\int_{L} \bigl(h_n-h\bigr) 
	\rho\,d\lambda}^2 \,dV(\mx)\, dt\right]
	\\ & \,\, 
	= \lim_{n\to \infty}\sum_i 
	E \left[\int_0^T\int_{M_i} \zeta_i(\mx) 
	\abs{\int_{L} \bigl(h_n(t,\mx,\lambda)-h(t,\mx,\lambda)\bigr) 
	\rho(\lambda)\, d\lambda}^2\, dV(\mx)\, dt\right]
	\\ & \,\,
	=\lim_{n\to \infty}\sum_i
	E \left[\int_0^T\int_{\alpha_i^{-1}(M_i)} 
	\zeta_i(\alpha_i(\mx)) 
	\abs{\int_L \bigl(h_n\circ \alpha_i(\mx,\lambda)
	-h\circ \alpha_i(\mx,\lambda)\bigr)
	\rho(\lambda)\, d\lambda}^2 
	\, J_i(\mx)\, d\mx\right]
	\overset{\eqref{local-cv}}{=}0,
\end{align*} 
where $\Seq{J_i}$ are the corresponding Jacobians.
\end{proof}

\section{Vanishing diffusion-capillarity limit}
\label{sec:singular-limit}

This section investigates the vanishing 
diffusion-capillarity limit of \eqref{PP-1}, 
intending to prove Theorem \ref{main-thm}. According to 
Theorem \ref{unique_sol_par}, the pseudo-parabolic 
SPDE \eqref{PP-1}, \eqref{PP-ID} possesses 
a unique solution $u_k \in L_{\prob}^2(\Omega;L^2((0,T);$ $H^1(M)))$. 
We will deduce some additional a priori 
estimates (dependent on $k$) related 
to higher regularity and derive a kinetic formulation 
of the pseudo-parabolic SPDE. Finally, using 
Skorokhod-Jakubowski representations of random 
variables in quasi-Polish spaces and 
Theorem \ref{thm-1} (stochastic velocity averaging), 
we prove strong convergence towards a 
weak solution of a stochastic conservation 
law with discontinuous flux.

\subsection{Additional a priori estimates}

We need an auxiliary result. 

\begin{lemma}\label{lem:gc_and_integral}
Suppose $\mff\in C^1(M\times \R)$ satisfies the 
geometry compatibility condition \eqref{gc}
and let $u\in H^1(M)$. Then 
$\int_M \larab{\mff(\mx,u(\mx))}{\nabla u(\mx)}
\, dV(\mx) = 0$.
\end{lemma}

\begin{proof}
Let $u_m\in \cinfty(M)$ be such that 
$u_m\to u$ in $H^1(M)$ as $m\to \infty$. Then 
$$
\int_M \larab{\mff(\mx,u_m(\mx))}{\nabla u_m(\mx)}
\, dV(\mx) \to \int_M 
\larab{\mff(\mx,u(\mx))}{\nabla u(\mx)}
\, dV(\mx),
$$
so it suffices to prove the claim for smooth $u$. 
Assuming $u\in C^\infty(M)$, let 
$X(\mx):=\int_0^{u(\mx)} f(\mx,\lambda)\,d\lambda$. 
Then due to \eqref{gc}, in local 
coordinates we have
\begin{align*}
	\ddiv X &= \frac{\pa X^l}{\pa x^l} 
	+\Gamma^j_{lj} X^l = \mff^l(x,u)\pa_lu 
	+\int_0^u \frac{\pa}{\pa x^l}\mff^l(\mx,\lambda)
	+\Gamma^j_{lj}(\mx) \mff^l(x,\lambda)\,d\lambda
	\\& = \mff^l(x,u)\pa_lu
	+\int_0^u \ddiv \mff(\mx,\lambda)\,d\lambda 
	= \mff^l(x,u)\pa_lu 
	= \mff^l(x,u) g_{lm}(g^{mr}\pa_r u)
	= \lara{\mff}{\nabla u}.
\end{align*}
Hence, by Stokes' theorem, 
$\int_M \lara{\mff}{\nabla u}\, dV 
= \int_M \ddiv X\, dV = 0$.
\end{proof}

We proceed with the a priori estimates that 
will put us in a position to apply the 
functional analytic tools developed before.

\begin{lemma}\label{osnov}
In addition to the conditions required 
by Theorem \ref{unique_sol_par}, suppose that 
the flux $\mff$ satisfies ($C_f$--4). 
Moreover, assume that $\delta_k\leq \eps_k \le 1/2$. 
For each fixed $k\in \N$, the solution $u_k$ 
of the pseudo-parabolic SPDE 
\eqref{PP-1}, \eqref{PP-ID} satisfies, 
for any $t\in [0,T]$,
\begin{equation}\label{prva}
	\begin{split}
		&\frac{1}{2}(1-\delta_k)
		E\left[\norm{u_k(t)}^2_{L^2(M)}\right]
		+\frac{1}{2}\delta_k E\left[
		\norm{u_k(t)}^2_{H^1(M)}\right]
		\\ & \qquad \qquad \qquad 
		+\eps_k E\left[\int_0^t 
		\norm{\nabla u_k(t')}_{L^2(M)}^2
		\, dt'\right]\le C_0,
	\end{split}
\end{equation}

\begin{equation}\label{cetvrta}
	\begin{split}
		&\frac{1}{2}\delta_k (1-\delta_k)
		E\left[\norm{u_k(t,\cdot)}^2_{H^1(M)}\right]
		+\frac{1}{2}\delta_k^2 E\left[
		\norm{u_k(t,\cdot)}^2_{H^2(M)}\right] 
		\\ & \qquad \qquad \quad
		+\frac{\eps_k\delta_k}{2}  
		E\left[\int_0^t 
		\norm{u_k(t',\cdot)}_{H^2(M)}^2\, dt'\right] 
		\leq \tilde{C}_0,
	\end{split}
\end{equation}
where
\begin{align*}
	C_0 & = C_0\left(t,\norm{u_0}_{L^2_x}^2,
	\delta_k \norm{u_0}_{H^1}^2\right)
	\\ & =e^{C_\Phi t}\left(C_\Phi t+
	\frac{1}{2}(1-\delta_k)\norm{u_0}^2_{L^2(M)} 
	+\frac{1}{2}\delta_k \norm{u_0}^2_{H^1(M)}\right),
	\\ \tilde{C}_0 
	& = \tilde{C}_0\left(t,\norm{u_0}_{L^2_x}^2,
	\delta_k \norm{u_0}_{H^1_{\mx}}^2,
	\delta_k^2 \norm{u_0}_{H^2_{\mx}}^2, 
	\frac{\delta_k}{\eps_k^2}\norm{\mff_k'}_{\infty}^2\right)
	\\ & =\frac{1}{2}\delta_k (1-\delta_k)
	\norm{u_0}_{H^1(M)}^2
	+ \frac{1}{2}\delta_k^2 
	\norm{u_0}^2_{H^2(M)}
	\\ & \qquad 
	+C_0\left(t,\norm{u_0}_{L^2_x}^2,
	\delta_k \norm{u_0}_{H^1_x}^2\right)
	\left(\frac12 +C_\Phi t+\delta_k 
	+\frac{\delta_k}{2\eps_k^2}
	\norm{\mff_k'}_{\infty}^2\right)
	\\ & \qquad\qquad
	+ C_\Phi t+t\frac{\delta_k}{2\eps_k} 
	\norm{\sup\limits_{\lambda\in \R}
	\norm{\mff_k(\cdot,\lambda)}_g}_{L^2(M)}^2,
\end{align*}
and $C_\Phi$ is the constant appearing 
in ($C_\Phi$--1). 

In addition, we have the following higher 
moment bounds: for any finite $p>2$:
\begin{equation}\label{eq:moment-diss}
	E \left[\sup_{t\in [0,T]}
	\norm{u_n(t)}_{L^2(M)}^p \right]
	\leq \overline{C}_0,
	\qquad 
	\eps_k^{p/2} E\left[\abs{\int_0^T\int_M 
	\abs{\nabla u_k}^2 \, d \mx \, dt'}^{p/2}\right]
	\leq \overline{C}_0,
\end{equation}
where $\overline{C}_0 =\overline{C}_0
\left(T,\norm{u_0}_{L^2_x}^2,\delta_k \norm{u_0}_{H^1}^2\right)$ 
is independent of $k$.
\end{lemma}

\begin{proof}
We divide the proof into two steps.

\medskip

\noindent \underline{1. Estimates \eqref{prva} 
and \eqref{eq:moment-diss}.} 
First, note that we can represent the solution 
via the $L^2$-orthonormal basis $\Seq{e_j}_{j\in \N}$ 
used for the Galerkin approximations:
\begin{equation*}%\label{sol-form-1}
	u_k(\omega,t,\mx)
	=\sum\limits_{j=1}^\infty 
	\alpha_j(\omega,t)e_j(\mx),
\end{equation*} 
where each $\alpha_j$ is a continuous adapted 
process on $[0,T]$ (suppressing the dependence 
on $k$). For any $s\in \N_0$, we have
\begin{align}\label{Hs}
	& E\left[\norm{u_k}^2_{L^2(0,T;H^s(M))}\right]
	=E\left[\int_0^T
	\sum\limits_{j=1}^\infty \lambda_j^{2s} 
	\abs{\alpha_j(t)}^2\, dt \right].
\end{align} 
By the orthonormality of $\Seq{e_j}_{j\in \N}$, 
as in \eqref{alpha-sys}, the processes $\alpha_j$ 
satisfy the SDEs
\begin{equation}\label{basic}
	\begin{split}
		& \left(1-\delta_k+\delta_k \lambda^2_j\right) 
		d\alpha_j 
		+\eps_k\left(\lambda^2_{j}-1\right)\alpha_j \, dt
		\\ & \qquad 
		=\int_M  \mff_k(\mx,u_k)\cdot 
		\nabla e_j \,dV(\mx) \, dt
		+\int_M \Phi(\mx,u_k) e_j \,dV(\mx)\, dW_t,
		\quad j\in \N.
	\end{split}
\end{equation} 
By It\^o's formula,  
\begin{equation}\label{1-1}
	\begin{split}
		& \frac{1}{2}\left(1-\delta_k
		+\delta_k \lambda^2_j\right)\, 
		d\abs{\alpha_j}^2
		+\eps_k(\lambda^2_j-1)\abs{\alpha_j}^2\, dt
		\\ & \qquad 
		= \int_M  \mff_k(\mx,u_k)\cdot 
		\nabla \left(\alpha_j e_j(\mx)\right)\,dV(\mx)\, dt
		\\ & \qquad \qquad
		+\frac{1}{2\left(1-\delta_k
		+\delta_k\lambda^2_j\right)}
		\left(\int_M \Phi(\mx,u_k) e_j(\mx) 
		\,dV(\mx)\right)^2\, dt
		\\ & \qquad \qquad \qquad
		+\int_M \Phi(\mx,u_k) \alpha_j
		e_j(\mx) \,dV(\mx)\, dW_t,
	\end{split}
\end{equation}
where we note that $\frac{1}{2\left(1-\delta_k
+\delta_k\lambda^2_j\right)}\le 1$, as $\delta_k\le \frac12$ 
and $0<\lambda_j\to \infty$ as $j\to \infty$. 
We sum \eqref{1-1} with respect to $j\in \N$, 
integrate over $(0,t)$, and use \eqref{Hs} with 
$s=0$ and $s=1$, to obtain
\begin{equation}\label{(a)}
	\begin{split}
		& \frac{1}{2}(1-\delta_k)
		\norm{u_k(t)}^2_{L^2(M)}
		+\frac{1}{2}\delta_k\norm{u_k(t)}^2_{H^1(M)} 
		+\eps_k \int_0^t\norm{\nabla u_k(t')}_{L^2(M)}^2\, dt'
		\\ & \qquad
		\le \frac{1}{2}(1-\delta_k)\norm{u_0}^2_{L^2(M)} 
		+\frac{1}{2}\delta_k \norm{u_0}^2_{H^1(M)}
		\\ & \qquad \qquad
		+\int_0^t \int_M  \bigl\langle \mff_k(\mx,u_k(t')),
		\nabla u_k(t')\bigr\rangle \,dV(\mx) \,dt' 
		\\ & \qquad \qquad\qquad
		+\int_0^t\norm{\Phi(\cdot,u_k(t'))}_{L^2(M)}^2\, dt' 
		+\int_0^t\int_M  u_k(t')\Phi(\mx,u_k(t'))
		\,dV(\mx)\, dW_{t'}.
	\end{split}
\end{equation} 
Using Lemma \ref{lem:gc_and_integral}, the third term on 
the right-hand side is zero. We can bound the fourth term 
by $2C_\Phi\left(t+\int_0^t\norm{u_k(t')}_{L^2(M)}^2\, dt\right)$,  
where $C_\Phi$ is the constant appearing 
in the assumption ($C_\Phi$--1). Applying the 
expectation operator in \eqref{(a)} and 
then Gronwall's inequality, we obtain
\begin{equation*}%\label{1-2}
	\begin{split}
		& \frac{1}{2}(1-\delta_k)
		E\left[\norm{u_k(t)}_{L^2(M)}^2\right]
		+\frac{1}{2}\delta_k 
		E\left[\norm{u_k(t)}^2_{H^1(M)}\right]
		+\eps_k E\left[\int_0^t 
		\norm{\nabla u_k(t,\cdot,\cdot)}_{L^2(M)}^2
		\,dt'\right]
		\\ & \qquad 
		\le e^{C_\Phi t}\left(C_\Phi t+
		\frac{1}{2}(1-\delta_k)
		\norm{u_0}_{L^2(M)}^2 
		+\frac{1}{2}\delta_k
		\norm{u_0}_{H^1(M)}^2\right)
		=\Bbb{C}_0,
	\end{split}
\end{equation*}
where $\Bbb{C}_0$ depends on 
$\norm{u_0}_{L^2(M)}^2$, 
$\delta_k \norm{u_0}_{H^1(M)}^2$, 
and (monotonously increasing on) $t$. 
This proves the first estimate \eqref{prva}.

Finally, let us prove \eqref{eq:moment-diss}. 
From \eqref{(a)},
\begin{equation}\label{eq:moment-diss-tmp1}
	\begin{split}
		& \norm{u_k(t)}^2_{L^2(M)}
		+\eps_k \int_0^t\int_M\abs{\nabla u_k}^2
		\, \,dV(\mx) \, dt'
		\\ & \qquad 
		\le C_0 +\int_0^t\norm{\Phi(\cdot,u_k(t'))}_{L^2(M)}^2\, dt' 
		+\int_0^t\int_M  u_k(t')\Phi(\mx,u_k(t'))
		\,dV(\mx)\, dW_{t'},  
	\end{split}
\end{equation}
where $C_0=\frac{1}{2}(1-\delta_k)\norm{u_0}^2_{L^2(M)}
+\frac{1}{2}\delta_k \norm{u_0}^2_{H^1(M)}$. 
Raise the inequality \eqref{eq:moment-diss-tmp1} 
to power $p/2$, $p>2$, take 
the supremum in time, and then the expectation. 
Proceeding as in the proof of \eqref{eq:energy-tmp4-L2}, 
applying the BDG inequality and eventually 
Gronwall's inequality, we arrive first at
\begin{equation}\label{eq:moment-diss-tmp2}
	E \left[\sup_{t\in [0,\tau]}
	\norm{u_n(t)}_{L^2(M)}^p\right] \le C,
	\quad \tau\in [0,T],
\end{equation}
where the constant $C$ is independent of $k$ but 
dependent on $C_0$. This proves the first 
part of \eqref{eq:moment-diss}. 
Second, from the same inequality, using \eqref{eq:moment-diss-tmp2}, 
we obtain the second part of \eqref{eq:moment-diss}.

\medskip

\noindent \underline{2. Estimate \eqref{cetvrta}.} In what 
follows, we work with finite sums in order 
to avoid some regularity issues that 
would otherwise appear, setting
$$
u^{(n)}_k(\omega,t,\mx)
:=\sum\limits_{j=1}^n \alpha_{j}(\omega,t)e_j(\mx),
\quad n\in \N,
$$
where each $\alpha_j$ is a continuous adapted 
process on $[0,T]$ (suppressing the dependency 
on $k$). Multiplying \eqref{basic} by $\lambda^2_j$ and 
then applying the It\^o formula, we obtain:
\begin{equation}\label{2-1}
	\begin{split}
		& \frac{1}{2}\left((1-\delta_k)
		\lambda^2_j+\delta_k \lambda^4_j\right)
		\, d\abs{\alpha_j}^2 
		+\eps_k\left(\lambda^4_j-\lambda^2_j\right)
		\abs{\alpha_j}^2 \, dt 
		\\ & \qquad 
		= \int_M \mff_k(\mx,u_k)\cdot 
		 \nabla \left(\alpha_j\lambda^2_j e_j(\mx)\right) 
		\,dV(\mx) \, dt 
		\\ & \qquad\qquad
		+\frac{\lambda^2_j}{2(1-\delta_k+\delta_k \lambda^2_j)}
		\left(\int_M \Phi(\mx,u_k) e_j(\mx) 
		\,dV(\mx)\right)^2 \, dt 
		\\ & \qquad \qquad\qquad 
		+ \int_M \Phi(\mx,u_k) \lambda^2_j 
		\alpha_j e_j(\mx) \,dV(\mx)\, dW_t.
	\end{split}
\end{equation}
where we note that $\frac{\delta_k\lambda^2_j}{2(1-\delta_k
+\delta_k \lambda^2_j)}\le \frac12$.
Next, we multiply \eqref{2-1} by $\delta_k$, sum over 
$j=1,\ldots,n$, and integrate over $(0,t)$:
\begin{equation}\label{(b)}
	\begin{split}
		& \frac{1}{2}\delta_k (1-\delta_k)
		\norm{u^{(n)}_k(t)}_{H^1(M)}^2
		+\frac{1}{2}\delta_k^2
		\norm{u^{(n)}_k(t)}^2_{H^2(M)}
		+\eps_k\delta_k\int_0^t
		\norm{u^{(n)}_k(t')}_{H^2(M)}^2 \, dt'
		\\ & \qquad \leq 
		\frac{1}{2}\delta_k(1-\delta_k)\norm{u_0}^2_{H^1(M)}
		+ \frac{1}{2}\delta_k^2\norm{u_0}^2_{H^2(M)}
		+ \eps_k\delta_k \int_0^t 
		\norm{u^{(n)}_k(t')}_{H^1(M)}^2 \, dt'
		\\ & \qquad \qquad  
		+\delta_k \int_0^t \int_M  
		\mff_k(\mx,u_k(t'))\cdot \nabla \left(u^{(n)}_k
		-\Delta u^{(n)}_k\right)\,dV(\mx)\, dt'
		\\ & \qquad\qquad\qquad
		+\frac12\int_0^t\norm{\Phi(\cdot ,u_k(t'))}_{L^2(M)}^2\, dt'
		\\ & \qquad \qquad \qquad \qquad
		+\delta_k \int_0^t \int_M \Phi(\mx,u_k(t')) 
		\sum\limits_{j=1}^n \lambda_j^2
		\alpha_j e_j(\mx) \,dV(\mx) \, dW_{t'}.
	\end{split}
\end{equation}
Taking the expectation in \eqref{(b)}, we obtain 
\begin{equation}\label{2-2}
	\begin{split}
		& \frac{1}{2} \delta_k (1-\delta_k) 
		E\left[\norm{u^{(n)}_k(t)}_{H^1(M)}^2\right]
		+\frac12\delta_k^2
		E\left[\norm{u^{(n)}_k(t)}_{H^2(M)}^2\right]
		\\& \qquad\qquad
		+\eps_k\delta_k 
		E\left[ \int_0^t 
		\left( \norm{u^{(n)}_k(t')}_{H^2(M)}^2
		-\norm{u^{(n)}_k(t')}_{H^1(M)}^2 \right)
		\,dt'\right] 
		\\ & \qquad 
		\leq \frac{1}{2}\delta_k (1-\delta_k)
		\norm{u_0}_{H^1(M)}^2
		+ \frac{1}{2}\delta_k^2 
		\norm{u_0}^2_{H^2(M)}+I_1+I_2+I_3,
	\end{split}
\end{equation}
where 
\begin{align*}
	I_1 & :=\delta_k E\left[\int_0^t 
	\int_M  \mff_k(\mx,u_k)\cdot 
	\nabla u^{(n)}_k(t')\,dV(\mx)\, dt'\right],
	\\ 
	I_2 & := \delta_k E \left[\int_0^t 
	\int_M  \Div \mff_k(\mx,u_k(t')) 
	\Delta u^{(n)}_k(t') \,dV(\mx)\, dt'\right],
	\\ 
	I_3 &:= \frac12 E\left[\int_0^t 
	\norm{\Phi(\cdot,u_k(t'))}_{L^2(M)}^2\,dt'\right].
\end{align*}
Let us estimate the terms $I_1,I_2,I_3$ separately. 

Using Young's product inequality, ($C_f$--2) 
and \eqref{prva}, we obtain
\begin{equation}\label{31}
	\begin{split}
		I_1 & \leq \frac{\delta_k}{2\eps_k} 
		E\left[
		\norm{\mff_k(\cdot,u_k)}^2_{L^2((0,t)\times M)}
		\right]
		+\frac{\eps_k}{2} E\left[ 
		\norm{\nabla u_k}^2_{L^2((0,t)\times M)}
		\right]
		\\ & \leq t\frac{\delta_k}{2\eps_k} 
		\norm{\sup\limits_{\lambda\in \R}
		\norm{\mff_k(\cdot,\lambda)}_g}_{L^2(M)}^2
		+\frac12 \Bbb{C}_0\left(t,\norm{u_0}_{L^2_{\omega,x}}^2,
		\delta_k \norm{u_0}_{L^2_\omega H^1_x}^2\right).
	\end{split}
\end{equation}

By ($C_f$--2), cf.~\eqref{gc}, we have
$$
\Div \mff_k(\mx,u_k) 
=\left(\Div \mff_k\right)(\mx,\lambda)
\big|_{\lambda=u_k(\mx)}
+\larab{\mff_k'(\mx,u_k)}{\nabla u_k}
=\larab{\mff_k'(\mx,u_k)}{\nabla u_k}.
$$
Therefore, using ($C_f$--2), Young's 
product inequality, and again \eqref{prva}, 
\begin{equation}\label{32}
	\begin{split}
		I_2 & \leq \frac{\delta_k}{2\eps_k}
		\norm{\mff_k'}_{\infty}^2 
		E\left[\norm{\nabla u_k}_{L^2((0,t)\times M)}^2\right]
		+\frac{\eps_k \delta_k}{2}
		E\left[\norm{\Delta u^{(n)}_k}^2_{L^2((0,t)\times M)}\right]
		\\ & \leq \frac{\delta_k}{2\eps_k^2}
		\norm{\mff_k'}_{\infty}^2 
		\Bbb{C}_0\left(t,\norm{u_0}_{L^2_{\omega,x}}^2,
		\delta_k \norm{u_0}_{L^2_\omega H^1_x}^2\right)
		+\frac{\eps_k \delta_k}{2} 
		E\left[\int_0^t \norm{u^{(n)}_k(t')}_{H^2(M)}^2\,dt'\right],
	\end{split}
\end{equation}
where the last term can be absorbed into the left-hand side 
of \eqref{2-2}. 

By ($C_\Phi$--1) and \eqref{prva},
\begin{equation}\label{33}
 	I_3 \leq C_\Phi t\left(1+
 	\Bbb{C}_0\left(t,\norm{u_0}_{L^2_{\omega,x}}^2,
 	\delta_k \norm{u_0}_{L^2_\omega H^1_x}^2\right)\right).
\end{equation} 

Finally, regarding one of the terms on the left-hand 
side of \eqref{2-2},
\begin{equation}\label{34}
	\eps_k\delta_k 
	E\left[ \int_0^t \norm{u^{(n)}_k(t')}_{H^1(M)}^2
	\, dt' \right] \overset{\eqref{prva}}{\leq} 
	\delta_k \Bbb{C}_0\left(t,\norm{u_0}_{L^2_{\omega,x}}^2,
	\delta_k \norm{u_0}_{L^2_\omega H^1_x}^2\right).
\end{equation}

Inserting \eqref{31}, \eqref{32}, \eqref{33}, \eqref{34} 
into \eqref{2-2}, and then letting $n\to \infty$, 
we arrive at \eqref{cetvrta}.
\end{proof}

\subsection{Kinetic formulation of the pseudo-parabolic SPDE}
\label{subsec:kinetic}

The kinetic formulation of \eqref{PP-1}, 
which is the content of the next lemma, is 
needed to apply the stochastic 
velocity averaging result (Theorem \ref{thm-1}) 
to deduce the strong convergence of $u_k$.

\begin{lemma}[kinetic SPDE]\label{kinetic-L}
In addition to the conditions required 
by Theorem \ref{unique_sol_par}, suppose that 
the flux $\mff$ satisfies ($C_f$--4). 
Regarding the diffusion-capillarity parameters 
$\eps_k,\delta_k$ in \eqref{PP-1}, suppose that 
the relation \eqref{neps} holds. 
For each fixed $k\in \N$, let $u_k$ be the unique solution of 
the pseudo-parabolic SPDE \eqref{PP-1}, \eqref{PP-ID}, 
cf.~Theorem \ref{unique_sol_par}. Then the function 
\begin{equation}\label{sol-form}
	h_k=h_k(\omega,t,\mx,\lambda)
	:=\sign\bigl(u_k(\omega,t,\mx)-\lambda\bigr)
\end{equation}
satisfies the kinetic SPDE
\begin{equation}\label{kinetic}
	\begin{split}
 		& dh_k+\Div_{\mx} 
 		\bigl(\mff_k'(\mx,\lambda)h_k\bigr)\,dt 
 		\\ & \quad 
 		= \Bigl[\Div_{\mx} G_k^{(0,1)}
 		-\pa_{\lambda} G_k^{(0,2)}
 		+G_k^{(1)}+\pa_\lambda G_k^{(2)}\Bigr]\,dt
 		+\Bigl[G_k^{(0,3)}-G_k^{(3)} \Bigr]\, dW_t
	\end{split}
\end{equation} 
in $\cD_{t,\mx,\lambda}'$, $\prob$--a.s., i.e.~in the weak 
sense of \eqref{weak-sense-tx}. The 
distribution-valued random variables
$$
G_k^{(0,1)},\, G_k^{(0,2)}, \, G_k^{(0,3)},\, 
G_k^{(1)}, \, G_k^{(2)}, \, G_k^{(3)}:
\Omega\to \cD'_{t,\mx,\lambda}
$$ 
have the following properties:
\begin{itemize}
	\item[(i)] There exist $\bar{G}_k^{(0,1)}$ 
	and $G_k^{(0,1,\lambda)}$ such that 
	$G_k^{(0,1)}=\bar{G}_k^{(0,1)}
	+\pa_{\lambda} G_k^{(0,1,\lambda)}$, where
	$$
	E \left[\norm{\bar{G}_k^{(0,1)}}_{L^2((0,T)
	\times M\times \R)}^2\right], \,
	E \left[
	\norm{G_k^{(0,1,\lambda)}}_{L^2((0,T)
	\times M\times \R)}^2\right] 
	\lesssim \eps_k\tok 0.
	$$
	Furthermore, 
	$$
	E\left[\norm{G_k^{(0,2)}}_{\cM((0,T)\times M\times \R)}^2
	\right]\lesssim 1.
	$$
	For any $\rho\in C_c(\R)$, 
	$\inn{G_k^{(0,3)},\rho}$ is an adapted 
	$L^2(M)$-valued process. There exist 
	$\bar{G}_k^{(0,3)}$ and $G_k^{(0,3,\lambda)}$ 
	such that $G_k^{(0,3)}=\bar{G}_k^{(0,3)}
	+\pa_{\lambda} G_k^{(0,3,\lambda)}$, where
	$$
	E \left[\norm{\bar{G}_k^{(0,3)}}_{L^2((0,T)
	\times M\times \R)}^2\right], \,
	E \left[
	\norm{G_k^{(0,3,\lambda)}}_{L^2((0,T)
	\times M\times \R)}^2\right] 
	\lesssim 1.
	$$

	\item[(ii)] There exists $\bar{G}^{(1)}_k\in 
	\Lip(\R;L^1([0,T]\times M))$ such that 
	${G}^{(1)}_k=\pa^2_{\lambda\lambda} \bar{G}^{(1)}_k$ 
	and, for every $\lambda \in \R$,
	$$
	E\left[\norm{\bar{G}_k^{(1)}}_{L^1_{\operatorname{loc}}
	((0,T)\times M\times \R)}^2\right]\lesssim 1.
	$$
	
	\item[(iii)] We have
	$$
	E\left[\norm{G_k^{(2)}}_{\cM((0,T)\times M\times \R)}^2
	\right]\lesssim 1.
	$$
	
	\item[(iv)] For $\rho\in C_c(\R)$, 
	$\inn{G_k^{(3)},\rho}$ is an adapted 
	$L^2(M)$-valued process. There exist 
	$\bar{G}_k^{(3)}$ and $G_k^{(3,\lambda)}$ 
	such that $G_k^{(3)}=\bar{G}_k^{(3)}
	+\pa_{\lambda} G_k^{(3,\lambda)}$, where
	$$
	E \left[\norm{\bar{G}_k^{(3)}}_{L^2((0,T)
	\times M\times \R)}^2\right], \,
	E \left[
	\norm{G_k^{(3,\lambda)}}_{L^2((0,T)
	\times M\times \R)}^2\right] 
	\lesssim 1.
	$$
\end{itemize}
\end{lemma}

\begin{proof}
We need to derive an entropy 
admissibility condition from \eqref{PP-1}. 
Here we cannot simply use 
the It\^o formula since the 
time-derivative appears not only 
on the function $u_k$, but also on the 
Laplacian of $u_k$. Therefore, we need to split 
the equation into separate parts satisfied by $u_k$ 
and $\Delta u_k$.  We first calculate, for 
any $S\in C^{2}(\R)$ with $S'\in C_0^1(\R)$,
\begin{align}\label{f2}
	\begin{split}
		&\Div \left(\int_{-\infty}^{u_k}
		\mff_k'(\mx,\lambda)
		S'(\lambda)\, d\lambda \right)
		\\ & \quad 
		= S'(u_k)(\mff_k')^i(\mx,u_k)
		\pa_i u_k+\int_{-\infty}^{u_k}S'(\lambda) 
		\pa_i (\mff_k')^i(\mx,\lambda)\,d\lambda
		+\int_{-\infty}^{u_k}S'(\lambda) 
		\Gamma^j_{lj} (\mff_k')^l(\mx,\lambda)\,d\lambda
		\\ & \quad
		= S'(u_k)\Div \bigl(\mff_k({\mx},u_k)\bigr)
		-S'(u_k)\Div \bigl(\mff_k(\mx,\lambda))
		\Big|_{\lambda=u_k} 
		\\ &\qquad\qquad
		+\int_{-\infty}^{u_k}
		S'(\lambda) \Div \bigl(\mff_k'(\mx,\lambda)
		\bigr)\,d\lambda
		=S'(u_k)\Div \bigl(\mff_k({\mx}, u_k)\bigr),
	\end{split}
\end{align}
keeping in mind the geometry compatibility 
condition \eqref{gc}. 

Note that
\begin{equation}\label{f3}
	S'(u_k) \Delta u_k
	=\Delta S(u_k)-S''(u_k)\abs{\nabla u_k}^2
	=\Div \bigl( S'(u_k) \nabla u_k \bigr)
	-S''(u_k)\abs{\nabla u_k}^2.
\end{equation}

Next, assume that we have
\begin{equation}\label{genform}
	du_k=\mu\, dt +\sigma\, dW_t,
\end{equation} 
where $\mu,\sigma \in L^1_{\prob}(\Omega;H^1((0,T)\times M))$ 
will be determined below (here we suppress the 
$k$-dependence in the notation). 
We have (in the $\mx$-weak sense, at least)
\begin{equation}\label{delta}
	d\Delta u_k=\Delta \mu\,dt 
	+\Delta \sigma \, dW_t. 
\end{equation}
Combining \eqref{genform} and \eqref{delta} 
with \eqref{PP-1}, we see that 
\begin{align*}
	&\bigl(\mu-\delta_k \Delta \mu\bigr)\, dt
	+\bigl(\sigma-\delta_k\Delta \sigma\bigr)\, dW_t
	=du_k-\delta_k d\Delta u_k
	\\ & \qquad 
	= -\Div \mff_k(\mx, u_k)\, dt
	+\eps_k \Delta u_k\, dt+\Phi(\mx,u_k)\, dW_t.  
\end{align*}
Consequently, we require
\begin{align}
\label{Poiss1}
	\mu-\delta_k \Delta \mu 
	& =-\Div \mff_k(\mx, u_k)
	+\eps_k \Delta u_k,
	\\ \label{Poiss2}
	\sigma-\delta_k \Delta \sigma
	& =\Phi(\mx,u_k).
\end{align} 
Next, we write
$$
\mu=\sum\limits_{j=1}^\infty 
\alpha_j(\omega,t)e_j(\mx), \quad 
\sigma=\sum\limits_{j=1}^\infty 
\beta_j(\omega, t)e_j(\mx),
$$ 
where $\alpha_j$, $\beta_j$ are adapted continuous 
processes, i.e., we represent $\mu$, $\sigma$ 
via the $L^2(M)$ orthonormal basis $\Seq{e_k}_{k\in \N}$ 
and compute the coefficients (as we have done several 
times before). Given \eqref{Poiss1},
\begin{align*}
	\sum\limits_{j=1}^\infty \alpha_j
	\bigl(1+\delta_k(\lambda_j^2-1)\bigr) 
	e_j(\mx) =-\Div\mff_k(\mx, u_k)
	+\eps_k \Delta u_k.
\end{align*}
Multiplying by $e_i$, integrating over $M$, 
and using the orthonormality of $\Seq{e_j}_{j\in \N}$, 
we obtain
\begin{equation}\label{drift}
	\begin{split}
		& \alpha_i
		=\frac{1}{1+\delta_k(\lambda_i^2-1)}
		\int_M \bigl(-\Div \mff_k(\mx, u_k)
		+\eps_k \Delta u_k\bigr) e_i(\mx)\,dV(\mx) 
		\implies
		\\ & 
		\mu=\sum_{j=1}^\infty 
		\frac{e_j}{1+\delta_k(\lambda_j^2-1)}
		\int_M \bigl(-\Div \mff_k(\mx, u_k)
		+\eps_k \Delta u_k\bigr)e_j(\mx) \,dV(\mx).
	\end{split}
\end{equation}
Similarly, in view of \eqref{Poiss2}, we obtain for 
the noise coefficient $\sigma$ that
\begin{equation}\label{diff}
	\begin{split}
		\beta_i & =\frac{1}{1+\delta_k(\lambda_i^2-1)} 
		\int_M \Phi(\mx,u_k) e_i(\mx) \,dV(\mx) 
		\implies
		\\
		\sigma & =\sum_{j=1}^\infty
		\frac{e_j}{1+\delta_k(\lambda_j^2-1)}
		\int_M  \Phi(\mx,u_k) e_j(\mx) \,dV(\mx).
	\end{split}
\end{equation}
Observe that the latter infinite sum is 
convergent in $L^2(M)$ for $d\prob \times dt$--a.e.~$(\omega,t)$, by 
the orthogonality of $\left\{e_k\right\}$ and since the 
squares of the coefficients in the sum satisfy
\begin{align*}
	& \sum_{j=1}^\infty
	\frac{1}{(1+\delta_k(\lambda_j^2-1))^2}
	\Big(\int_M  \Phi(\mx,u_k) e_j(\mx) \,dV(\mx)\Big)^2 
	\\ & \qquad 
	\leq \sum_{j=1}^\infty
	\left(\int_M  \Phi(\mx,u_k) e_j(\mx) \,dV(\mx)\right)^2
	=\norm{\Phi(\cdot,u_k)}_{L^2(M)}^2<\infty. 
\end{align*}
The same argument for finiteness applies to similar 
sums appearing elsewhere.

Given \eqref{genform}, \eqref{delta} 
and \eqref{drift}, \eqref{diff}, we apply 
It\^o's formula to compute $S(u_k)$:
\begin{align*}
	dS(u_k) & =\left(S'(u_k)\mu
	+\frac{1}{2}S''(u_k)\sigma^2 \right)\,dt
	+S'(u_k)\sigma \, dW_t
	\\ & =S'(u_k)\sum_{j=1}^\infty
	\frac{e_j(\mx)}{1+\delta_k(\lambda_j^2-1)}
	\int_M \bigl(-\Div \mff_k(\my, u_k)
	+\eps_k \Delta u_k (\my)\bigr)e_j(\my)\,dV(\my)\,dt
	\\ &\qquad
	+\frac{S''(u_k)}{2}\left(\sum_{j=1}^\infty
	\frac{e_j(\mx)}{1+\delta_k(\lambda_j^2-1)} 
	\int_M \Phi(\my,u_k) e_j(\my) \,dV(\my)\right)^2\, dt
	\\ &\qquad \qquad 
	+ S'(u_k)\sum_{j=1}^\infty 
	\frac{e_j(\mx)}{1+\delta_k(\lambda_j^2-1)}
	\int_M \Phi(\my,u_k) e_j(\my) \,dV(\my)\, dW_t.
\end{align*} 

Taking into account the identity 
$$
\frac{1}{1+\delta_k(\lambda^2_j-1)}
=1-\frac{\delta_k(\lambda^2_j
-1)}{1+\delta_k(\lambda^2_j-1)}, 
$$
we obtain
\begin{align}
	dS(u_k) & = 
	S'(u_k)\sum_{j=1}^\infty  e_j(\mx) 
	\int_M \left(-\Div \mff_k(\my, u_k)
	+\eps \Delta u_k\right)e_j(\my) \,dV(\my)\, dt
	\notag 
	\\ & \qquad
	+ S'(u_k)\sum_{j=1}^\infty e_j(\mx) 
	\int_M \Phi(\my,u_k) e_j(\my) \,dV(\my) \, dW_t
	\notag
	\\ & \qquad 
	-S'(u_k)\sum_{j=1}^\infty e_j(\mx) 
	\frac{\delta_k(\lambda_j^2-1)}{1+\delta_k(\lambda_j^2-1)}
	\notag \\ & \qquad \qquad \qquad \times
	\int_M \bigl(-\Div \mff_k(\my, u_k)
	+\eps_k \Delta u_k\bigr)e_j(\my) \,dV(\my)\,dt
	\notag \\
	&\qquad
	+\frac{S''(u_k)}{2} \left(\sum_{j=1}^\infty
	e_j(\mx)\frac{1}{1+\delta_k(\lambda_j^2-1)}
	\int_M \Phi(\my,u_k) e_j(\my) \,dV(\my)\right)^2 \, dt
	\notag \\ & \qquad 
	-S'(u_k)\sum\limits_{j=1}^\infty  
	e_j(\mx)\frac{\delta_k(\lambda_j^2-1)}
	{1+\delta_k(\lambda_j^2-1)}
	\int_M \Phi(\my,u_k) e_j(\my) \,dV(\my) \, dW_t
	\notag
	\\ & =: I_k^{(1)}\,dt +I_k^{(2)}\,dW_t 
	+ I_k^{(3)}\,dt
	+ I_k^{(4)}\,dt+I_k^{(5)}\, dW_t.
	\label{ec-1}
\end{align}
Recalling that $h_k=\sign (u_k-\lambda)$,
$$
S(u_k) = \innb{\delta(\cdot-u_k),S(\cdot)}
= \frac{1}{2}\innb{\sign(u_k-\cdot),S'(\cdot)}
=\frac{1}{2}\innb{h_k,S'},
$$
so that $dS(u_k)=\frac{1}{2}d\innb{h_k, S'}$. 
Given that $\Seq{e_j}_{j\in \N}$ is 
an orthonormal basis of $L^2(M)$,
\begin{align*}
	I_k^{(1)} & =S'(u_k) \bigl(-\Div\mff_k(\mx,u_k)
	+\eps_k \Delta u_k\bigr)
	\\ & \overset{\eqref{f2},\eqref{f3}}{=}
	-\Div \int_{-\infty}^{u_k} 
	\mff_k'(\mx,\lambda)S'(\lambda)\, d\lambda
	+\Div\bigl(S'(u_k)\,\eps_k \nabla u_k\bigr)
	-S''(u_k)\, \eps_k \abs{\nabla u_k}^2. 
\end{align*}
By \eqref{gc},
\begin{align*}
	\Div \int_{-\infty}^{u_k}
	\mff_k'(\mx,\lambda)S'(\lambda)\, d\lambda 
	&=\Div \int_{\R}
	\frac{1}{2} \bigl(\sign(u_k-\lambda)+1\bigr)
	\mff_k'(\mx,\lambda)S'(\lambda)\, d\lambda
	=\frac{1}{2}\Div \innb{\mff_k'h_k,S'}.
\end{align*}
Furthermore,
\begin{align*}
	\Div\bigl(S'(u_k)\,\eps_k \nabla u_k\bigr)
	-S''(u_k)\, \eps_k \abs{\nabla u_k}^2
	= \frac12 \inn{\Div G_k^{(0,1)},S'}
	-\frac12 \inn{\pa_{\lambda} G_k^{(0,2)},S'},
\end{align*}
where 
$$
G_k^{(0,1)}=2\delta(\lambda-u_k)\,
\eps_k\nabla u_k, \qquad
G_k^{(0,2)}=2\delta(\lambda-u_k)\,
\eps_k\abs{\nabla u_k}^2, 
$$
so that 
\begin{equation}\label{eq:Ik1}
	I_k^{(1)}
	=-\frac{1}{2}\Div \innb{\mff_k'h_k,S'}
	+\frac12 \inn{\Div G_k^{(0,1)},S'}
	-\frac12 \inn{\pa_{\lambda} G_k^{(0,2)},S'}.
\end{equation}

Again by the orthonormality 
of the basis $\Seq{e_j}_{j\in \N}$,  
we have that $I_k^{(2)}=S'(u_k)\Phi(\mx,u_k)$, and
\begin{align*}
	S'(u_k)\Phi(\mx,u_k)
	& =\inn{\delta(\cdot-u_k) \Phi(\mx,\cdot),S'(\cdot)}
	= \frac12\inn{G_k^{(0,3)},S'},
\end{align*} 
where 
$$
G_k^{(0,3)}= 2\delta(\lambda-u_k)
\Phi(\mx,\lambda). 
$$
Therefore, 
\begin{equation}\label{eq:Ik2}
	I_k^{(2)}=\frac12 \inn{G_k^{(0,3)},S'}.
\end{equation}

Writing $S'(u_k)=\innb{\delta(\cdot-u_k),S'(\cdot)}$, 
we obtain
\begin{equation}\label{eq:Ik3-5}
	I_k^{(3)}=\frac12\inn{G_k^{(1)},S'},
	\quad
	I_k^{(4)}=\frac12 
	\inn{\pa_\lambda G_k^{(2)},S'},
	\quad
	I_k^{(5)}=-\frac12 \innb{G_k^{(3)},S'},
\end{equation}
where
\begin{align*}
	& G_k^{(1)} = -2\delta(\lambda-u_k)
	\sum_{j=1}^\infty  
	\frac{e_j(\mx)\delta_k(\lambda_j^2-1)}
	{1+\delta_k(\lambda_j^2-1)}
	\int_M \bigl(-\Div \mff_k(\my,u_k)
	+\eps_k \Delta u_k\bigr)e_j(\my) \,dV(\my),
	\\ & 
	G_k^{(2)} = -\delta(\lambda-u_k) 
	\left(\,\sum_{j=1}^\infty 
	\frac{e_j(\mx)}{1+\delta_k(\lambda_j^2-1)}
	\int_M \Phi(\my,u_k) e_j(\my) \,dV(\my) \right)^2,
	\\ & 
	G_k^{(3)}=2\delta(\lambda-u_k)\sum_{j=1}^\infty 
	\frac{e_j(\mx)\delta_k(\lambda_j^2-1)}
	{1+\delta_k(\lambda_j^2-1)}
	\int_M \Phi(\my,u_k)e_j(\my) \,dV(\my).
\end{align*}
At this point, we are abusing the notation in two instances. 
First, the Dirac delta measure is denoted by $\delta(\cdot)$, 
although this partially conflicts with the capillarity parameter $\delta_k$. 
Second, the velocity variable $\lambda$ of the kinetic function 
$h_k=h_k(\omega,t,\mx,\lambda)$ partially conflicts 
with the $j$th eigenvalue $\lambda_j$ linked 
to the orthonormal basis $\left\{e_k\right\}_{k\in \N}$ introduced 
in Section \ref{sec:prelim}. 

Inserting the identities \eqref{eq:Ik1}, 
\eqref{eq:Ik2} and \eqref{eq:Ik3-5} 
into \eqref{ec-1} gives 
\begin{equation}\label{eq:kinetic-tmp}
	\begin{split}	
		& \innb{d h_k,S'}+\innb{\Div\bigl(\mff_k' h_k\bigr),S'}\,dt
		\\ & \quad 
		=\inn{\Div G_k^{(0,1)}
		-\pa_{\lambda} G_k^{(0,2)}
		+ G_k^{(1)}
		+\pa_\lambda G_k^{(2)},S'}\,dt
		+\inn{G_k^{(0,3)}-G_k^{(3)},S'}\, dW_t,
	\end{split}
\end{equation}
which holds in $\cD_{t,\mx}'$, almost surely. 
Using the arbitrariness of $S'$, we 
arrive at the kinetic SPDE 
\eqref{kinetic}. Strictly speaking, 
we have derived the weak form 
of \eqref{eq:kinetic-tmp} for test functions 
$\varphi\in \cD_{t,\mx,\lambda}$ 
that are tensor products $\phi\otimes \rho$ of 
$\phi\in \cD_{t,\mx}$ and $\rho=S'\in \cD_\lambda$. 
Using the density of such functions 
in $\cD_{t,\mx,\lambda}$ and the Schwarz kernel 
theorem we can extend \eqref {eq:kinetic-tmp} 
to $\cD'_{t,\mx,\lambda}$, so that \eqref{kinetic} 
holds. In the deterministic case, some 
related kinetic decompositions can be found  
in \cite{Hwang:2004}.

In view of \eqref{prva}, the following 
$k$-uniform bounds on $G_k^{(0,1)}$, $G_k^{(0,2)}$ 
and $G_k^{(0,3)}$ are easily verified:
\begin{equation}\label{eq:G0-bound-tmp}
	\begin{split}
		& E\left[\abs{\inn{G_k^{(0,1)},
		\phi\otimes \rho}}^2\right]
		\lesssim \eps_k\norm{\phi}_{L^2((0,T)\times M)}^2
		\norm{\rho}_{L^\infty(\R)}^2,
 		\\ & 
 		E\left[\abs{\inn{G_k^{(0,2)},
 		\phi\otimes \rho}}\right]
		\lesssim	
		\norm{\phi}_{L^\infty((0,T)\times M)}
		\norm{\rho}_{L^\infty(\R)},
		\\ & 
 		E\left[\abs{\inn{G_k^{(0,3)},
 		\phi\otimes \rho}}^2\right]
		\lesssim	
		\norm{\phi}_{L^1(0,T;L^2(M))}^2
		\norm{\rho}_{L^\infty(\R)}^2.
	\end{split}
\end{equation}
In view of the second part of \eqref{eq:G0-bound-tmp}, 
$$
E\left[\norm{G_k^{(0,2)}}_{\cM_{t,\mx,\lambda}}
\right]\lesssim 1.
$$
Since $H^1(\R)\subset L^\infty(\R)$, 
the first part of \eqref{eq:G0-bound-tmp} implies
$$
E\left[\abs{\inn{G_k^{(0,1)},
\phi\otimes \rho}}^2\right]
\lesssim \eps_k\norm{\phi}_{L^2((0,T)\times M)}^2
\norm{\rho}_{H^1(\R)}^2,
$$
and there exist $\bar{G}_k^{(0,1)}$ and 
$G_k^{(0,1,\lambda)}$ such that 
$G_k^{(0,1)}=\bar{G}_k^{(0,1)}
+\pa_{\lambda} G_k^{(0,1,\lambda)}$, where
$$
E \left[\norm{\bar{G}_k^{(0,1)}}_{L^2_{t,\mx,\lambda}}^2\right], 
\, E \left[
\norm{G_k^{(0,1,\lambda)}}_{L^2_{t,\mx,\lambda}}^2\right] 
\lesssim \eps_k\tok 0.
$$
Similarly, using also that 
$L^2_{t,\mx}\subset L^1_tL^2_{\mx}$, 
$$
E\left[\abs{\inn{G_k^{(0,3)},
\phi\otimes \rho}}^2\right]
\lesssim	
\norm{\phi}_{L^2((0,T)\times M)}^2
\norm{\rho}_{H^1(\R)}^2,
$$
and there exist $\bar{G}_k^{(0,3)}$ and 
$G_k^{(0,3,\lambda)}$ such that 
$G_k^{(0,3)}=\bar{G}_k^{(0,3)}
+\pa_{\lambda} G_k^{(0,3,\lambda)}$, where
$$
E \left[\norm{\bar{G}_k^{(0,3)}}_{L^2_{t,\mx,\lambda}}^2
\right], \,
E \left[
\norm{G_k^{(0,3,\lambda)}}_{L^2_{t,\mx,\lambda}}^2
\right] 
\lesssim 1.
$$
Summarising, we have established 
part (i) of the lemma.

It remains to show that the distribution-valued 
random variables $G_k^{(1)}$, $G_k^{(2)}$, 
and $G_k^{(3)}$ satisfy, respectively, 
parts (ii), (iii) and (iv) of the lemma. 
To this end, we shall use 
Lemma \ref{osnov} and the fact that
\begin{equation}\label{ineq}
	\frac{\sqrt{\delta_k\abs{\lambda_j^2-1}}}
	{1+\delta_k(\lambda_j^2-1)} \leq 1,
\end{equation}
recalling that the diffusion-capillarity parameters 
$\eps_k,\delta_k$ satisfy \eqref{neps}.

\medskip

\noindent \underline{Estimate of $G_k^{(1)}$:} We have 
$G_k^{(1)}=\pa^2_{\lambda\lambda} \bar{G}_k^{(1)}$, where 
$$
\bar{G}_k^{(1)}=-\abs{\lambda-u_k}\, \sum_{j=1}^\infty  
\frac{e_j(\mx)\delta_k(\lambda_j^2-1)}
{1+\delta_k(\lambda_j^2-1)}
\int_M \bigl(-\Div \mff_k(\my,u_k)
+\eps_k \Delta u_k\bigr)e_j(\my) \,dV(\my).
$$
For $\prob$-a.e.~$\omega\in\Omega$, using \eqref{ineq} 
and the Cauchy-Schwarz inequality (for sums), we have
\begin{align}
	\int_0^T\int_M \abs{\bar{G}_k^{(1)}}
	\, dV(\mx)\, dt  
	& =\Biggl| \int_0^T\sum_{j=1}^\infty
	\int_M \abs{\lambda-u_k}\, e_j(\mx) 
	\frac{\delta_k(\lambda_j^2-1)}
	{1+\delta_k(\lambda_j^2-1)} \,dV(\mx)
	\notag \\ & \qquad\qquad \qquad\times
	\int_M \bigl(-\Div \mff_k(\my, u_k)
	+\eps_k \Delta u_k\bigr)e_j(\my)\,dV(\my)
	\, dt \,d\lambda \Biggr|
	\notag \\ &  
	\leq  \int_0^T 
	I_u(\omega,t) I_\mff(\omega,t,\lambda) \,dt 
	+ \eps_k \delta_k^{\frac12} 
	\int_0^T I_u(\omega,t) I_\Delta(\omega,t)\,dt,
	\label{eq:G1-tmp}
\end{align} where
\begin{align*}
	I_u & :=\left(\sum_{j=1}^\infty \delta_k (\lambda_j^2-1)
	\abs{\int_M  |\lambda -u_k|\,
	e_j(\mx)\,dV(\mx)}^2\right)^{\frac{1}{2}}
	\\ & 
	=\delta_k \norm{\nabla_{\mx}\abs{\lambda-u_k}}_{L^2(M)} 
	\leq \delta_k \, \| \nabla u_k(t,\cdot)\|_{L^2(M)},
	\\ 
	I_\mff & :=\left(\sum_{j=1}^\infty
	\abs{\int_M \Div \mff_k(\my,u_k) e_j(\mx)
	\,dV(\my)}^2\right)^{\frac{1}{2}}
	\\ & 
	=\norm{\Div \mff_k(\cdot,u_k)}_{L^2(M)}
	\overset{\eqref{gc}}{\leq} 
	\norm{\mff_k'}_{L^\infty} 
	\norm{\nabla u_k(\omega,t,\cdot)}_{L^2(M)},
	\\ I_\Delta & := \left(\sum\limits_{j=1}^\infty 
	\abs{\int_M \Delta u_k e_j(\my) \,dV(\my)}^2 
	\right)^{\frac{1}{2}}=
	\norm{\Delta u_k(\omega,t,\cdot)}_{L^2(M)}, 
\end{align*}
making repeated use of the properties 
of $\Seq{e_j}_{j\in \N}$, $\Seq{\lambda_j}_{j\in \N}$. 
By the Cauchy-Schwarz inequality,
\begin{align*}
	\abs{\eqref{eq:G1-tmp}} 
	& \leq 
	\norm{I_u(\omega,\cdot)}_{L^2((0,T))}
	\norm{I_\mff(\omega,\cdot)}_{L^2(0,T)}
	\\ & \qquad 
	+\eps_k \delta_k^{\frac12} 
	\norm{I_u(\omega,\cdot)}_{L^2((0,T))}
	\norm{I_\Delta(\omega,\cdot)}_{L^2(0,T)}
	=:I_{\sqrt{\delta}}+I_{\eps\sqrt{\delta}},
\end{align*}
where 
\begin{align*}
	& I_{\sqrt{\delta}} \leq 
	\delta_k^{\frac12} \norm{\mff_k'}_{L^\infty}
	\norm{\nabla u_k}^2_{L^2((0,T)\times M)},
	\\ 
	& I_{\eps\sqrt{\delta}} 
	\leq \eps_k \delta_k^{\frac12} 
	\norm{\Delta u_k}_{L^2((0,T)\times M)}
	\norm{\nabla u_k}_{L^2((0,T)\times M)}.
\end{align*}
Hence,
\begin{equation}\label{SJ-1}
	\begin{split}
		\int_0^T\int_M\abs{\bar{G}_k^{(1)}}\, dV(\mx)\, dt 
		& \lesssim
		\Biggl( \, \frac{\delta_k^{\frac12}}{\eps_k}
		\norm{\mff_k'}_{L^\infty} \eps_k
		\norm{\nabla u_k}^2_{L^2((0,T)\times M)}
		\\ & \qquad\qquad \qquad
		+\eps_k^{\frac12} \norm{\nabla u_k}_{L^2((0,T)\times M)} 
		\,\left(\eps_k\delta_k\right)^{\frac12}
		\norm{\Delta u_k}_{L^2((0,T)\times M)}\Biggr).
	\end{split}
\end{equation}
Integrating \eqref{SJ-1} over $\lambda\in L$ 
(for $L\subset\subset \R$), squaring, taking 
the expectation, and then using ($C_f$--2), 
the second part of \eqref{eq:moment-diss} with $p=4$, 
\eqref{prva}, \eqref{cetvrta}, and the 
link \eqref{neps} between the diffusion-capillarity 
parameters $\eps_k,\delta_k$, we deduce that 
$$
E\left[\, \abs{\int_0^T\int_M\int_L\abs{\bar{G}_k^{(1)}}
\,d\lambda \, dV(\mx)\, dt}^2\right]\leq C_L,
$$
where $C_L$ is independent of $k$.

\medskip

\noindent \underline{Estimate of $G_k^{(2)}$:} 
It is simpler to deal with $G_k^{(2)}$. Set 
$$
I_{\Phi}=I_{\Phi}(\omega,t,\mx)
:=\sum_{j=1}^\infty
\frac{1}{1+\delta_k(\lambda_j^2-1)} 
\left(\int_M \Phi(\my,u_k) e_j(\my) \,dV(\my)
\right) e_j(\mx),
$$
and note that $\frac{1}{1+\delta_k(\lambda_j^2-1)}
\le 2$, assuming $\delta_k\le 1/2$ $\forall k$ 
and recalling $0<\lambda_j\to \infty$ as $j\to \infty$.
By the $L^2$ orthonormality of $\Seq{e_j}_{j\in \N}$, 
Bessel's inequality and the 
assumption \eqref{assump-Phi},
$$
\norm{I_\Phi(\omega,t,\cdot)}_{L^2(M)}^2
\leq 2 \norm{\Phi(\cdot,u_k)}_{L^2(M)}^2
\lesssim_{\Phi} \left(1
+\norm{u_k(\omega,t,\cdot)}_{L^2(M)}^2\right).
$$
Using this, for any $\varphi\in 
\cD((0,T)\times M\times \R)$ with $\phi\in \cD_{t,\mx}$ 
and $\rho\in \cD_{\lambda}$, we obtain
\begin{align*}
	\abs{\inn{G_k^{(2)},\varphi}}
	& =\abs{\int_0^T\int_M 
	\varphi(t,\mx,u_k)
	\bigl(I_\Phi(\omega,t,\mx)\bigr)^2\, dV(\mx)\, dt}
	\\ & \leq 
	\norm{\varphi}_{L^\infty((0,T)\times M\times \R)}  
	\int_0^T\int_M
	\bigl(I_\Phi(\omega,t,\mx)\bigr)^2\,dV(\mx)\, dt
	\\ & \lesssim_{\Phi,T,M}
	\norm{\varphi}_{L^\infty((0,T)\times M\times \R)}
	\left(1+\norm{u_k(\omega,\cdot)}_{L^2((0,T)\times M)}^2
	\right).
\end{align*}
Squaring, taking the expectation and then using 
the first part of \eqref{eq:moment-diss} with $p=4$,
\begin{align*}
	E\left[\abs{\inn{G_k^{(2)},\varphi}}^2\right]  
	& \lesssim_{\Phi,T,M}  
	\norm{\varphi}_{L^\infty((0,T)\times M\times \R)}^2,
\end{align*}
which implies part (iii) of the lemma.

\medskip

\noindent \underline{Estimate of $G_k^{(3)}$:} The treatment 
of the third term $G_k^{(3)}$ follows along the 
lines of $G_k^{(2)}$. Set 
$$
\tilde I_{\Phi}=\tilde I_{\Phi}(\omega,t,\mx)
:=\sum_{j=1}^\infty
\frac{\delta_k(\lambda_j^2-1)}{1+\delta_k(\lambda_j^2-1)} 
\left(\int_M \Phi(\my,u_k) e_j(\my) \,dV(\my)
\right) e_j(\mx),
$$
and note that $\abs{\frac{\delta_k(\lambda_j^2-1)}
{1+\delta_k(\lambda_j^2-1)}}\le 1$, assuming 
$\delta_k\le 1/2$ $\forall k$ 
and recalling that $0<\lambda_j\to \infty$ 
as $j\to \infty$. For any 
$\varphi=\phi\otimes \rho \in \cD((0,T)\times M\times \R)$ 
with $\phi\in \cD_{t,\mx}$ and $\rho\in \cD_{\lambda}$, 
by the Cauchy-Schwarz inequality,
\begin{align*}
	\abs{\inn{\frac{1}{2} G_k^{(3)},\phi\otimes \rho}}
	&=\abs{\int_0^T\int_M \varphi(t,\mx,u_k)
	\tilde I_\Phi(\omega,t,\mx)\, dV(\mx)\, dt}
	\\ & \le \norm{\phi}_{L^1((0,T);L^2(M))}
	\norm{\rho}_{L^\infty(\R)}
	\norm{\tilde I_\Phi(\omega,t,\cdot)}_{L^\infty((0,T);L^2(M))}.
\end{align*}
As before,
$$
\norm{\tilde I_\Phi(\omega,t,\cdot)}_{L^2(M)}^2
\lesssim_{\Phi,T,M} 1
+\norm{u_k(\omega,t,\cdot)}_{L^2(M)}^2,
$$
so that
$$
\norm{\tilde I_\Phi(\omega,\cdot,\cdot)}_{L^\infty((0,T);L^2(M))}^2
\lesssim_{\Phi,T,M} 1+
\norm{u_k(\omega,\cdot,\cdot)}_{L^\infty(0,T;L^2(M))}^2.
$$ 
By \eqref{prva}, 
$L^2_{t,\mx}\subset L^1_tL^2_{\mx}$ 
and $H^1(\R)\subset L^\infty(\R)$, we obtain
\begin{align*}
	E\left[\abs{\inn{G_k^{(3)},
	\phi\otimes \rho}}^2\right]  
	& \lesssim_{\Phi,T,M}
	\norm{\phi}_{L^2((0,T)\times M)}^2
	\norm{\rho}_{H^1(\R)}^2.
\end{align*}
As before, this establishes the final 
part (iv) of the lemma.
\end{proof} 

The following lemma exhibits a critical strong temporal 
translation estimate for the solution 
$h_k$ of the kinetic equation \eqref{kinetic} 
linked to the pseudo-parabolic SPDE \eqref{PP-1}. 
To motivate the lemma, let us, for the moment, 
set aside the probability variable $\omega$ 
(and thus $E[\cdot]$). The $k$-uniform 
bound $\norm{h_k}_{L^\infty}\lesssim 1$ implies 
that the $h_k$ is bounded in $L^p(0,T;L^2(M\times L))$ 
for all $p\in [1,\infty]$, $L\subset\subset \R$. 
Ignoring the statistical mean, the next 
lemma shows that $h_k$ is $t$-continuous 
in $L^1(0,T)$ for the norm $H^{-N}(M\times L)$, 
for $N$ large, uniformly in $k$. 
Then, according to \cite{Simon:1987vn}, 
this implies that the velocity averages of $\Seq{h_k}$ 
are compact in the space $L^2(0,T;H^{-1}(M\times L))$ 
(for the strong topology), and also in 
the slightly better space $L^2_t\bigl(L^2_w(M\times L)\bigr)$, 
cf.\eqref{eq:L2tL2x-weak}. The strong temporal compactness supplied 
by $L^2_t\bigl(L^2_w(M\times L)\bigr)$ is essential 
for applying our stochastic velocity 
averaging result, cf.~Theorem \ref{thm-1}. 
In the next section, we will handle 
the probability variable by invoking 
some results of Skorokhod, linked to 
the tightness of measures and a.s.~representations 
of random variables. 

\begin{lemma}[temporal translation estimate 
in $L^1_tH^{-N}_{\mx,\lambda}$]\label{lem:temporal-est}
For each $k\in \N$, let $h_k=h_k(\omega,t,\mx,\lambda)$ 
satisfy the kinetic SPDE \eqref{kinetic}. 
Fix $L\subset\subset \R$, and $N\in \N$ such 
$H^N_0(M\times L)\subset C^2(M\times L)$. 
Then, for any $\vartheta\in (0,T)$,
\begin{equation} \label{eq:temporal-est}
	E \left[\sup_{\tau\in (0,\vartheta)}
	\norm{h_k(\cdot,\cdot+\tau,\cdot,\cdot)
	-h_k}_{L^1(0,T-\tau;H^{-N}(M\times L))}\right]
	\lesssim C_L\left(\vartheta+\vartheta^{1/2}\right),
\end{equation}
for some $k$-independent constant $C_L$.
\end{lemma}

\begin{proof}
Recall that for any $f\in L^1(0,T;H^{-N}(M\times L))$,
$$
\norm{f}_{L^1(0,T;H^{-N}(M\times L))}
=\int_0^T \sup\Seq{\abs{\int_M\int_L 
f \phi \,d\mx \,d\lambda}: 
\phi\in H^N_0(M\times L), \, 
\norm{\phi}_{H_0^1(M\times L)}\leq 1} \,dt.
$$
To prove the lemma, it is sufficient to establish that
\begin{equation}\label{eq:time-estimate1}
	E \left[\, \sup_{\tau\in (0,\vartheta)} \int_0^{T-\tau}
	\sup_\phi \abs{\int_M\int_L \bigl(h_k(t+\tau)-h_k(t)\bigr)
	\phi \, d\mx \,d\lambda}
	\,dt \right]\leq 
	C \left(\vartheta+\vartheta^{1/2}\right),
\end{equation} 
for a constant $C$ independent of $k$. 
In \eqref{eq:time-estimate1}, the inner 
supremum runs over all $\phi\in H^N_0(M\times L)$ 
for which $\norm{\phi}_{H_0^1(M\times L)}\leq 1$. 
First, we use the SPDE \eqref{kinetic} 
and Lemma \ref{ito-eq} to write
\begin{equation}\label{eq:translation-tmp1}
	\begin{split}
		&\abs{\int_M\int_L \bigl( h_k(t+\tau)-h_k(t) \bigr)
		\phi \, d\mx \, d\lambda }
		\\ & \quad 
		\leq \sum_{i=1}^6\int_t^{t+\tau}\int_M\int_L
		\abs{I_i} \, d\mx \, d\lambda \, dt'
		+\sum_{i=7}^{10}\abs{\int_t^{t+\tau}\int_M\int_L
		I_i \, d\mx \, d\lambda \, dW(t')},
	\end{split}
\end{equation}
where 
\begin{align*}
	& I_1=\abs{\mff_k'(\mx,\lambda)}
	\abs{h_k}\abs{\nabla_{\mx}\phi}
	\lesssim 
	\abs{\nabla_{\mx}\phi}
	\lesssim
	\norm{\phi}_{H^N_0(M\times L)},
	\\ & 
	I_2=\abs{\bar{G}_k^{(0,1)}}
	\abs{\nabla_{\mx} \phi}
	\lesssim 
	\abs{\bar{G}_k^{(0,1)}}
	\norm{\phi}_{H^N_0(M\times L)},
	\\ & 
	I_3=\abs{G_k^{(0,1,\lambda)}}
	\abs{\pa_{\lambda}\nabla_{\mx}\phi}
	\lesssim 
	\abs{G_k^{(0,1,\lambda)}}
	\norm{\phi}_{H^N_0(M\times L)},
	\\ & 
	I_4=\abs{G_k^{(0,2)}}\abs{\pa_\lambda \phi}
	\lesssim 
	\abs{G_k^{(0,2)}}
	\norm{\phi}_{H^N_0(M\times L)},
	\\ &
	I_5=\abs{\bar{G}^{(1)}_k}\abs{\pa^2_{\lambda\lambda} \phi}
	\lesssim 
	\abs{\bar{G}^{(1)}_k}\norm{\phi}_{H^N_0(M\times L)},
	\\ & 
	I_6 =\abs{G_k^{(2)}}\abs{\pa_\lambda \phi}
	\lesssim \abs{G_k^{(2)}}
	\norm{\phi}_{H^N_0(M\times L)},
\end{align*}
and
\begin{align*} 
	I_7 =\bar{G}_k^{(0,3)}\phi,
	\quad 
	I_8 = G_k^{(0,3,\lambda)}\pa_\lambda \phi,
	\quad 
	I_9 = \bar{G}_k^{(3)}\phi, 
	\quad 
	I_{10}=G_k^{(3,\lambda)}\pa_\lambda \phi.
\end{align*}
In the estimate of $I_1$, we used ($C_f$--2) 
and that $\abs{h_k}\leq 1$. In addition, for $I_1,\ldots,I_6$, 
we specified the Sobolev index $N$ large enough to 
ensure---via the Sobolev embedding theorem---that the various 
terms involving $\phi$ are bounded by a constant times 
$\norm{\phi}_{H^N_0(M\times L)}$. By assumption, 
$\norm{\phi}_{H^N_0(M\times L)}\leq 1$. 

Observe that we have slightly abused the integral notation for $I_4$ and 
$I_6$ in \eqref{eq:translation-tmp1}, recalling that these 
objects are measure-valued in $\lambda$, see 
\eqref{eq:measure-m-def} for more precise notation.

In \eqref{eq:translation-tmp1}, we 
take the supremum over $\phi$, 
integrate with respect to $t\in (0,T-\tau)$,
take the supremum over $\tau\in (0,\vartheta)$, 
and finally apply the expectation operator $E[\cdot]$, yielding 
\begin{equation}\label{eq:translation-tmp2}
	\begin{split}
		&	E \left[\, \sup_{\tau\in (0,\vartheta)} \int_0^{T-\tau}
		\sup_\phi \abs{\int_M\int_L \bigl( h_k(t+\tau)-h_k(t) \bigr)
		\phi \, d\mx \, d\lambda}\, dt\right]
		\lesssim \vartheta\sum_{i=1}^6 J_i+\sum_{i=7}^{10} J_i,
	\end{split}
\end{equation}
where, using the Cauchy-Schwarz inequality (whenever needed) 
and the bounds derived in Lemma \ref{kinetic-L},
\begin{align*}
	& J_1=1,
	\quad 
	J_2=E\left[\norm{\bar{G}_k^{(0,1)}}_{L^1((0,T)\times M\times L)}\right]
	\lesssim_{T,L} 
	E\left[\norm{\bar{G}_k^{(0,1)}}_{L^2((0,T)\times M\times L)}\right]
	\lesssim 1,
	\\ & 
	J_3=E\left[
	\norm{G_k^{(0,1,\lambda)}}_{L^1((0,T)\times M\times L)}\right]
	\lesssim_{T,L}
	E\left[
	\norm{G_k^{(0,1,\lambda)}}_{L^2((0,T)\times M\times L)}\right]
	\lesssim 1,
	\\ & 
	J_4=E\left[\norm{G_k^{(0,2)}}_{\cM((0,T)\times M\times L)}\right]
	\lesssim 1,
	\quad 
	J_5= E\left[
	\norm{\bar{G}^{(1)}_k}_{L^1((0,T)\times M\times L)}\right]
	\lesssim_L 1,
	\\ &
	J_6=E\left[\norm{G_k^{(2)}}_{\cM((0,T)\times M\times L)}
	\right]\lesssim 1,
\end{align*}
and, for $i=7,8,9,10$,
\begin{align*}
	J_i=E\left[\sup_{\tau\in (0,\vartheta)}
	\int_0^{T-\tau}\sup_\phi\abs{\int_t^{t+\tau}\int_M\int_L
	I_i \, d\mx \, d\lambda \, dW(t')} \, dt\right].
\end{align*}

By the stochastic Fubini theorem and the 
Cauchy-Schwarz inequality,
\begin{align*}
	& \abs{\int_t^{t+\tau}\int_M\int_L
	I_7 \, d\mx \, d\lambda \, dW(t')}=
	\abs{\int_M\int_L \left(\,\int_t^{t+\tau}
	\bar{G}_k^{(0,3)} \, dW(t')\right) 
	\, \phi \, d\mx \, d\lambda}
	\\ & \qquad 
	\leq \norm{\int_t^{t+\tau}
	\bar{G}_k^{(0,3)} \, dW(t')}_{L^2(M\times L)}
	\norm{\phi}_{L^2(M\times L)}
	\lesssim \norm{\int_t^{t+\tau}
	\bar{G}_k^{(0,3)} \, dW(t')}_{L^2(M\times L)},
\end{align*}
where we have used that $\norm{\phi}_{L^2(M\times L)}\lesssim 
\norm{\phi}_{H^N_0(M\times L)}\leq 1$.  
Hence, by the (Hilbert-space valued) BDG inequality 
(see, e.g., \cite[page 174]{DaPrato:2014aa}), 
\begin{align*}
	& E\left[\sup_{\tau\in (0,\vartheta)}
	\int_0^{T-\tau} \sup_\phi\abs{\int_t^{t+\tau}\int_M\int_L
	I_7 \, d\mx \, d\lambda \, dW(t')} \, dt\right]
	\\ & \qquad
	\lesssim\int_0^{T-\vartheta} E\left[\sup_{\tau\in (0,\vartheta)}
	\norm{\int_t^{t+\tau}\bar{G}_k^{(0,3)} \, dW(t')}_{L^2(M\times L)}
	\right]\, dt
	\\ & \qquad 
	\lesssim \int_0^{T-\vartheta} 
	E \left[\left(\int_t^{t+\vartheta}
	\norm{\bar{G}_k^{(0,3)}}_{L^2(M\times L)}^2
	\,dt'\right)^{\frac12}\right]\, dt
	\\ & \qquad \lesssim_T
	\left(E \int_0^{T-\vartheta}\int_t^{t+\vartheta}\int_M\int_L
	\abs{\bar{G}_k^{(0,3)}}^2 \, d\mx \, d\lambda
	\,dt'\, dt \right)^{\frac12}
	\\ & \qquad \leq \vartheta^{\frac12}
	\left(E\left[\int_0^T\int_M\int_L
	\abs{\bar{G}_k^{(0,3)}}^2 \, d\mx \, d\lambda
	\,dt\right]\right)^{\frac12}
	\lesssim \vartheta^{\frac12}.
\end{align*}
Similarly,
\begin{align*}
	E\left[\sup_{\tau\in (0,\vartheta)}
	\int_0^{T-\tau} \sup_\phi\abs{\int_t^{t+\tau}\int_M\int_L
	I_i \, d\mx \, d\lambda \, dW(t')}\right]\, dt
	\lesssim \vartheta^{\frac12}, 
	\quad i=8,9,10.
\end{align*}

Utilising our estimates of $J_1,\ldots,J_{10}$ 
in \eqref{eq:translation-tmp2}, we arrive at the 
sought-after translation 
estimate \eqref{eq:time-estimate1}.
\end{proof}

\subsection{Tightness and Skorokhod-Jakubowski a.s.~representations}
For each $k\in \N$, $u_k$ is a solution to the pseudo-parabolic SPDE 
\eqref{PP-1} with initial data $u_0$, see \eqref{PP-ID}. 
According to Lemma \ref{kinetic-L}, the corresponding process 
$h_k:=\sign(u_k-\lambda)$ solves the kinetic 
SPDE \eqref{kinetic} (weakly in $t,\mx,\lambda$). 
Setting
\begin{align*}
	&
	D^f_k:=\bigl(\mff'-\mff_k'\bigr)h_k
	: \Omega\to \cD_{t,\mx,\lambda}',
	\quad 
	g_k^{(1)}:= \bar{G}_k^{(0,1)} 
	: \Omega\to \cD_{t,\mx,\lambda}',
	\quad 
	g_k^{(2)}:= \bar{G}_k^{(0,1)}
	: \Omega\to \cD_{t,\mx,\lambda}',
	\\ &
	g_k^{(3)}:= G_k^{(0,1,\lambda)} 
	: \Omega\to \cD_{t,\mx,\lambda}',
	\quad 
	g_k^{(4)}:= \bar{G}_k^{(1)}
	: \Omega\to \cD_{t,\mx,\lambda}',
	\quad 
	D_k := G_k^{(2)}-G_k^{(0,2)}
	: \Omega\to \cD_{t,\mx,\lambda}',
	\\
	& \Phi_k^{(1)} := \bar{G}_k^{(0,3)}-\bar{G}_k^{(3)}
	: \Omega\to \cD_{t,\mx,\lambda}',
	\quad
	\Phi_k^{(2)} := G_k^{(0,3,\lambda)}
	-G_k^{(3,\lambda)} 
	: \Omega\to \cD_{t,\mx,\lambda}',
\end{align*}
it follows that $h_k$ satisfies the 
following SPDE in $\cD_{t,\mx,\lambda}'$, 
almost surely: 
\begin{equation}\label{eq:kinetic-new}
	\begin{split}
		d h_k+\Div_\mx \bigl(F h_k\bigr)\, dt
		& =\Div_{\mx} \left(D^f_k+g_k^{(1)}+g_k^{(2)}\right)\,dt
		+\pa_{\lambda}\Div_{\mx} g_k^{(3)}\,dt
		+\pa^2_{\lambda\lambda} g_k^{(4)}\,dt
		+\pa_{\lambda} D_k\,dt
		\\ & \qquad 
		+\left[ \Phi_k^{(1)}+\pa_{\lambda}
		\Phi_k^{(2)}\right]\,dW_t, 
		\qquad h_k\big |_{t=0}=h_0,
	\end{split}
\end{equation} 
where $h_0:=\sign(u_0-\lambda)$, $u_0$ 
comes from \eqref{PP-ID}, and $F:=\mff'
=\partial_\lambda \mff(\mx,\lambda)$, 
cf.~\eqref{conv}.  Recall that $0\leq h_k\leq 1$ 
and $F\in L^2_{\mx,\lambda,\loc}$, cf.~\eqref{conv}, where 
$F\in L^2_{\mx,\lambda,\loc}
:=L^2_{\loc}(M\times \R)$ 
if and only if $F\in L^2_{\mx,\lambda} 
:=L^2(M\times L)$ for all $L\subset\subset \R$ 
(and similarly for the other local spaces).
Moreover, by \eqref{conv},
$$
\norm{D^f_k}_{L^2((0,T)\times 
M\times L)}\lesssim_{T,L}
\norm{\sup\limits_{\lambda \in \R}
\abs{\mff'(\cdot,\lambda)
-\mff_k'(\cdot,\lambda)}}_{L^2(M)}\tok 0, 
\qquad L\subset\subset \R.
$$
We will view the constituents of 
the right-hand side of \eqref{eq:kinetic-new} 
as random mappings of the form
\begin{align*}
	& g_k^{(1)},g_k^{(2)},g_k^{(3)},
	\Phi_k^{(1)},\Phi_k^{(2)}
	:\Omega \to L^2_{t,\mx,\lambda},
	\quad
	D^f_k
	:\Omega \to L^2_{t,\mx,\lambda,\loc},
	\\ &
	g_k^{(4)}:\Omega \to L^1_{t,\mx,\lambda,\loc},
	\quad 
	D_k:\Omega \to \cM_{t,\mx,\lambda},
\end{align*}
where, as before, ``loc" refers to the $\lambda$-variable. 
Given Lemma \ref{kinetic-L}, the 
following $k$-uniform bounds hold:
\begin{equation}\label{eq:basic-kinetic-est}
	\begin{split}
		& E \left[\norm{\left(g_k^{(1)},g_k^{(2)},
		g_k^{(3)}\right)}_{L^2_{t,\mx,\lambda}}^2\right]\tok 0, 
		\quad 
		E \left[\norm{D^f_k}_{L^2_{t,\mx,\lambda,\loc}}^2
		\right]\tok 0, 
		\\ &
		E\left[\norm{g_k^{(4)}}_{L^1_{t,\mx,\lambda,\loc}}^2
		\right]\lesssim 1,
		\quad
		E\left[\norm{D_k}_{\cM_{t,\mx,\lambda}}^2
		\right]\lesssim 1, 
		\quad 
		E\left[\norm{\left(\Phi_k^{(1)},
		\Phi_k^{(2)}\right)}_{L^2_{t,\mx,\lambda}}^2
		\right]\lesssim 1.
	\end{split}
\end{equation}

As for the stochastic part of 
\eqref{eq:kinetic-new}, $W$ is an $\R$-valued 
Wiener process  defined on a 
filtered probability space $\bigl(\Omega,\mcf,\prob,
\Seqb{\mcf_t}\bigr)$ satisfying the usual conditions. 
Moreover, since $u_k$ is $\Seqb{\mcf_t}$--adapted,  
$\innb{h_k,\rho}$ and $\innb{\Phi_k,\rho}$ 
are adapted to $\Seqb{\mcf_t}$, for each $\rho\in C_c(\R)$.
Each $u_k$ is a.s.~continuous in time with 
values in $L^2(M)$, cf.~\eqref{eq:pseudo-spaces}, and thus 
the velocity averages $\innb{h_k,\rho}$ and 
$\innb{\Phi_k,\rho}$ are similarly time-continuous. 
However, Lemma \ref{ito-eq} shows 
that the weak limit of the velocity averages 
$\innb{h_k,\rho}$ is at best weakly right-continuous. 
Generally, $h_k$, $\Phi_k$, $g_k$, $G_k$ do 
not possess temporal continuity. Whenever needed, stochastic objects 
that lack temporal regularity are interpreted as random 
distributions (see Section \ref{sec:prelim}). 

Returning to \eqref{eq:basic-kinetic-est}, we 
have additional estimates for $h_k$. 
To summarise the key estimates 
(for later use), we will introduce a particular 
Banach space. Inspired by \cite{Bensoussan:1995aa}, 
fix $N\in \N$ and two sequences $\Seq{\mu_m}_{m\in \N}$, 
$\Seq{\nu_m}_{m\in \N}$ of positive numbers 
converging to zero as $m\to\infty$. 
For any $L\subset\subset \R$, consider the set
$$
\cZ_{\mu_m,\nu_m}^{(p,q)}(L):=
\Seq{u\in L^p(0,T;L^q(M\times L)):
\sup_{m\in \N} \frac{1}{\nu_m}
\sup_{\tau\in (0,\mu_m)}
\norm{u(\cdot+\tau,\cdot)-u}_{L^1(0,T-\tau;
H^{-N}(M\times L))}<\infty},
$$
for $p,q\in [1,\infty]$, which is endowed 
with the norm
\begin{equation}\label{eq:cZ-norm}
	\norm{u}_{\cZ_{\mu_m,\nu_m}^{(p,q)}(L)}
	:=\norm{u}_{L^p(0,T;L^q(M\times L))}
	+\sup_{m\in \N} \frac{1}{\nu_m}
	\sup_{\tau\in (0,\mu_m)}
	\norm{u(\cdot+\tau,\cdot)
	-u(\cdot,\cdot)}_{L^1(0,T-\tau;H^{-N}(M\times L))}.
\end{equation}
Note that $\cZ_{\mu_m,\nu_m}^{(p,q)}$ is a 
Banach space, which is separable (and thus Polish) 
if $p,q<\infty$. We also need the 
corresponding probabilistic space, for $r\in [1,\infty]$,
\begin{align*}
	L^r_{\prob}\bigl(\cZ_{\mu_m,\nu_m}^{(p,q)}(L)\bigr)
	& :=\Biggl\{u\in L^r_\prob
	\bigl(\Omega;L^p(0,T;L^q(M\times L))\bigr): 
	\\ & \qquad\qquad 
	E\left[\sup_{m\in \N} \frac{1}{\nu_m}
	\sup_{\tau\in (0,\mu_m)}
	\norm{u(\cdot+\tau,\cdot)
	-u}_{L^1(0,T-\tau;H^{-N}(M\times L))}\right]
	<\infty\Biggr\},	
\end{align*}
which is a Banach space when equipped with the norm
\begin{align*}
	\norm{z}_{L^r_{\prob}\bigl(\cZ_{\mu_m,\nu_m}^{(p,q)}
	(L)\bigr)} & :=\left(E\norm{u}_{L^p(0,T;L^q(M\times L))}^r
	\right)^{\frac{1}{r}}
	\\ & \qquad 
	+E\left[\sup_{m\in \N} \frac{1}{\nu_m}
	\sup_{\tau\in (0,\mu_m)}
	\norm{u(\cdot+\tau,\cdot)
	-u}_{L^1(0,T;H^{-N}(M\times L))}\right].
\end{align*}
Suppose the sequences $\Seq{\mu_m}$ and $\Seq{\nu_m}$ 
(which converge to zero) satisfy
\begin{equation}\label{eq:rm-num}
	\sum_{m=1}^\infty \frac{\mu_m^\alpha}{\nu_m}<\infty,
\end{equation}
where $\alpha\in (0,1)$ is given 
by Lemma \ref{lem:temporal-est}. 
Then $h_k$ is bounded in $L^r_{\prob}
\bigl(\cZ_{\mu_m,\nu_m}^{(p,q)}(L)\bigr)$: 
\begin{equation*}%\label{eq:hk-LrZ-bound}
	\norm{h_k}_{L^r_{\prob}
	\bigl(\cZ_{\mu_m,\nu_m}^{(p,q)}(L)\bigr)}
	\leq C_L, \quad \forall p,q,r\in [1,\infty],
\end{equation*}
for some constant $C_L$ independent of $k$, 
which follows from the 
$L^\infty_{\omega,t,\mx,\lambda}$-bound 
on $h_k$ and Lemma \ref{lem:temporal-est}.

The bounds \eqref{eq:basic-kinetic-est} imply 
weak convergences in $(\omega,t,x,\lambda)$ 
of the functions 
\begin{equation}\label{eq:Xk-def}
	\begin{split}
		& X^{(1)}_k:=h_k, \,\,
		X^{(2)}:=D^f_k,\,\,
		X^{(3)}_k:=g_k^{(1)}, \,\,
		X^{(4)}_k:=g_k^{(2)},\,\,
		X^{(5)}_k:=g_k^{(3)},\,\,
		X^{(6)}_k:=g_k^{(4)},\,\,
		\\ & 
		X^{(7)}_k:=D_k, \,\,
		X^{(8)}_k:=\Phi_k^{(1)}, \,\,
		X^{(9)}_k:=\Phi_k^{(2)},\,\,
		X^{(10)}_k:=W, \,\,
		X^{(11)}_k:=h_0,
	\end{split}
\end{equation}
which---according to Lemma \ref{lem:weak-limit-kinetic}---is 
sufficient to pass to the limit in the 
kinetic SPDE \eqref{eq:kinetic-new}. 
To be able to apply the stochastic velocity 
averaging result (Theorem \ref{thm-1}), it 
is essential that we upgrade the weak convergence in 
the probability variable $\omega$ to strong (a.s.) 
convergence. To conclude a.s.~convergence 
we will invoke some results  
linked to tightness of probability measures 
and a.s.~representations of random variables 
in quasi-Polish spaces \cite{Jakubowski:1997aa} 
(see Section \ref{sec:prelim}). Simultaneously, we will 
upgrade the convergence in the temporal variable $t$ 
to strong $L^2$ convergence, which is also a key condition for 
applying the stochastic velocity averaging theorem. 
To this end, we introduce the quasi-Polish path space
$$
\cX=\cX_1\times \cdots \times \cX_{11},
$$
where
\begin{align*}
	& \cX_1 = 
	L^2\Bigl(0,T;\bigl(L_{\loc}^2(M\times \R),\tau_w\bigr)\Bigr)
	=: L^2_t\bigl(L_{\loc,w}^2(M\times \R)\bigr),
	\\ & 
	\cX_2=
	\Bigl(L^2_{\loc}((0,T)\times M\times \R),\tau_w\Bigr)
	=:L^2_{\loc,w}((0,T)\times M\times \R),
	\\ &
	\cX_3=\cX_4=\cX_5=\cX_8=\cX_9
	=\Bigl(L^2((0,T)\times M\times \R),\tau_w\Bigr)
	=: L^2_w((0,T)\times M\times \R),
	\\ & 
	\cX_6=
	\Bigl(\cM_{\loc}((0,T)\times M\times \R),\tau_{w\star}\Bigr)
	=:\cM_{\loc,w\star}((0,T)\times M\times \R)
	\\ &
	\cX_7=\Bigl(\cM((0,T)\times M\times \R),\tau_{w\star}\Bigr)
	=:\cM_{w\star}((0,T)\times M\times \R),
	\\ & 
	\cX_{10}=C([0,T]),
	\quad
	\cX_{11}=\Bigl(L^2_{\loc}(M\times \R),\tau_w\Bigr)
	=:L^2_{\loc,w}(M\times \R).
\end{align*}
The factor space $\cX_{l}$ is Polish for 
$l=10$ and otherwise quasi-Polish.  
The corresponding $\sigma$-algebras are 
denoted by $\cB_{\cX},\cB_1,\ldots,\cB_{11}$, 
where $\cB_l:=\cB(\cX_l)$ is chosen 
as the Borel $\sigma$-algebra for all
$l$ except $l=6,7$. Regarding $\cB_6$ and $\cB_7$, denote 
by $\Seq{f_\ell^{(l)}}_{\ell\in \N}$ 
the separating sequence of the quasi-Polish space 
$\cX_{l}$ and by $\cB_{f^{(l)}}$ the $\sigma$-algebra generated by 
$\Seq{f_\ell^{(l)}}_{\ell\in \N}$, $l=6,7$, 
see Section \ref{sec:prelim}. 
We then specify $\cB_6:=\cB_{f^{(6)}}\subset \cB(\cX_6)$ 
and $\cB_7:=\cB_{f^{(7)}}\subset \cB(\cX_7)$. 
For concreteness, we take $\cB_{\cX}:=\cB(\cX)$. 
Consider the random mappings 
$X_k:\bigl(\Omega,\prob,\F\bigr)
\to \bigl(\cX,\cB_{\cX}\bigr)$ defined by
$$
\Omega\ni \omega \mapsto X_k(\omega)\in \cX, 
\quad X_k(\omega):=\bigl(X_k^{(1)}(\omega),\ldots,
X_k^{(11)}(\omega)\bigr),
$$
with the corresponding joint laws
\begin{equation}\label{eq:Xk-joint-laws}
	\cL_k=\prob\circ X_k^{-1}.
\end{equation}
The marginals of $\cL_k$ are denoted 
by $\cL_k^{(l)}$, $l=1,\ldots,11$. We will verify that 
the laws $\cL_k$ are tight by demonstrating 
that the product measures $\cL_k^{(1)}
\times \cdots \times \cL_k^{(11)}$ are tight (with 
respect to $\cB_1\times\cdots\times\cB_{11}$).

\begin{lemma}[tightness of joint laws]\label{lem:tight}
The sequence $\Seq{\cL_k}_{k\in \N}$ of 
probability measures---defined by 
\eqref{eq:Xk-joint-laws}---is tight on 
the path space $\bigl(\cX,\cB_{\cX}\bigr)$.
\end{lemma}

\begin{proof}
For each $\kappa>0$, we need to produce compact sets
\begin{align*}
	& \cC_{\kappa}^{(1)} \subset 
	L^2_t\bigl(L_{\loc,w}^2(M\times \R)\bigr),
	\quad 
	\cC_{\kappa}^{2}\subset 
	L^2_{\loc,w}((0,T)\times M\times \R),
	\\ &
	\cC_{\kappa}^{(3)},\, \cC_{\kappa}^{4}, \, \cC_{\kappa}^{(5)},\,
	\cC_{\kappa}^{(8)},\, \cC_{\kappa}^{(9)} 
	\subset L^2_w((0,T)\times M\times \R),
	\\ & 
	\cC_{\kappa}^{6} \subset 
	\cM_{\loc,w\star}((0,T)\times M\times \R),
	\quad
	\cC_{\kappa}^{(7)} \subset 
	\cM_{w\star}((0,T)\times M\times \R),
	\\ & 
	\cC_{\kappa}^{(10)} \subset C([0,T]), 
	\quad \text{and} \quad
	\cC_{\kappa}^{(11)} \subset L^2_{\loc,w}(M\times \R),
\end{align*}
such that, uniformly in $k\in \N$,
\begin{equation}\label{eq:tight-marginals}
	\cL_k^{(l)}\bigl(\cC_{\kappa}^{(l)}\bigr)
	>1-\kappa/11 \quad 
	\Longleftrightarrow \quad
	\cL_k^{(l)}\bigl((\cC_{\kappa}^{(l)})^c\bigr)
	=\prob\bigl(X_k^{(l)}\in 
	(\cC_{\kappa}^{(l)})^c\bigr)
	\le \kappa/11, \qquad l=1,\ldots,11.
\end{equation} 
This will immedately imply that 
$$
\cL_k\bigl(\cC_{\kappa}\bigr)
=\prob\bigl(X_k \in \cC_{\kappa}\bigr)
> 1-\kappa,
\quad 
\text{where $\cC_{\kappa}:=\cC_{\kappa}^{(1)}
\times \cdots\times \cC_{\kappa}^{(11)}$ is compact},
$$
which establishes the uniform 
tightness of $\Seq{\cL_k}_{k\in\N}$ 
according to Definition \ref{def:tight-measures}.

\medskip

Let us now prove \eqref{eq:tight-marginals}, starting 
with $l=1$. For $L\subset\subset \R$, set 
$\cZ_{\mu_m,\nu_m}(L):=\cZ_{\mu_m,\nu_m}^{(p,2)}(L)$, for 
some $p>2$, and suppose \eqref{eq:rm-num} holds. 
Clearly, for any $N>1$, 
$$
L^2(M\times L)\subset\subset H^{-1}(M\times L)
\subset\subset H^{-N}(M\times L),
$$ 
where $X\subset\subset Y$ means that $X$ 
is compactly embedded in $Y$. By \cite{Simon:1987vn}, we have that
\begin{equation}\label{eq:cZ-compact-tmp}
	\Seq{u:\norm{u}_{\cZ_{\mu_m,\nu_m}(L)}\leq R}
	\quad \text{is relatively compact in 
	$L^2(0,T;H^{-1}(M\times L))$},
\end{equation} 
for any $R>0$. Indeed, to conclude this we need 
Theorem 6 of \cite{Simon:1987vn} on the partial compactness of 
functions with values in an intermediate space. 
Let $X_1$, $X_0$, $X_{-1}$ be Banach spaces 
with continuous embeddings $X_1\subset X_0\subset X_{-1}$ 
and $X_1$ compactly embedded in $X_0$. 
Then \cite[Theorem 6]{Simon:1987vn} ensures 
that a set $\cZ$ is relatively compact in $L^q(0,T;X_0)$, 
$\forall q\in [1,\bar{q})$, $\bar{q}\in[1,\infty]$, 
if $\cZ$ is bounded in $L^{\bar{q}}(0,T;X_0)\cap L^1(0,T;X_1)$ 
and $\norm{u(\cdot+\tau)-u}_{L^1(0,T;X_{-1})}\to 0$ as 
$\tau\to 0$, uniformly for $u\in \cZ$. 
To conclude \eqref{eq:cZ-compact-tmp}, 
we apply this result with $X_1=L^2(M\times L)$, 
$X_0=H^{-1}(M\times L)$, $X_{-1}= H^{-N}(M\times L)$, $\overline{q}=p$ 
(recall $p>2$), and $N$ given by Lemma \ref{lem:temporal-est}.
Observe that \eqref{eq:cZ-compact-tmp} can be improved to
\begin{equation}\label{eq:cZ-compact}
	\Seq{u:\norm{u}_{\cZ_{\mu_m,\nu_m}(L)}\leq R}
	\quad \text{is relatively compact in
	$L^2_t\bigl(L_w^2(M\times L)\bigr)$},
\end{equation}
which follows from \eqref{eq:cZ-compact-tmp} and the uniform 
$L^2$ boundedness of the involved 
functions, cf.~\eqref{eq:cZ-norm}.

\medskip

Let $\Seq{L_l}_{l\in\N}$ be a sequence of 
sets $L_l\subset\subset \R$ that
monotonically increases towards $\R$ 
as $l\to \infty$, and let 
$\Seqb{R_{\kappa}^{(1)}(l)}_{l\in\N}$ 
be a sequence of positive numbers.
By \eqref{eq:cZ-compact} (and 
a diagonal argument), the set 
$$
\cK_\kappa^{(1)}:=
\Seq{u:\norm{u}_{\cZ_{\mu_m,\nu_m}(L_l)}
\leq R_{\kappa}^{(1)}(l), \, l\in \N}
$$
is a relatively compact subset 
of $L^2_t\bigl(L_{\loc,w}^2(M\times \R)\bigr)$. 
For any $m\ge 1$, we have the obvious bound
$$
\prob\bigl(X_k^{(1)}\in 
(\cK_{\kappa}^{(1)})^c\bigr)
\leq \sum_{l=1}^\infty\prob\left(\Seq{\omega\in
\Omega:\norm{h_k}_{\cZ_{\mu_m,\nu_m}(L_l)}
>R_{\kappa}^{(1)}(l)}\right)\leq A+B_m,
$$
where
\begin{align*}
	& A:=\sum_{l=1}^\infty\prob\left(\Seq{\omega\in
	\Omega:\norm{h_k}_{L^p(0,T;L^2(M\times L_l))}
	>R_{\kappa}^{(1)}(l)}\right),
	\\ & 
	B_m:=\sum_{l=1}^\infty\prob\left(\Seq{\omega\in
	\Omega:\sup_{\tau\in (0,\mu_m)}
	\norm{h_k(\cdot,\cdot+\tau,\cdot,\cdot)
	-h_k}_{L^1(0,T-\tau;H^{-N}(M\times L_l))}
	>\nu_m R_{\kappa}^{(1)}(l)}\right).
\end{align*}
By the Chebyshev inequality 
and $\abs{h_k}\leq 1$, 
$$
A\leq \sum_{l=1}^\infty\frac{1}{R_{\kappa}^{(1)}(l)} 
E\left[\norm{h_k}_{L^p(0,T;L^2(M\times L_l))}\right]
\lesssim \sum_{l=1}^\infty\frac{C_l}{R_{\kappa}^{(1)}(l)},
$$
which can be made $\leq \kappa/22$ by taking 
$R_{\kappa}^{(1)}(l)=2^{-l}C_lR_{\kappa}^{(1)}$ with 
$R_{\kappa}^{(1)}$ large enough (uniformly in $k$). 
Similarly, given \eqref{eq:temporal-est} 
and \eqref{eq:rm-num},
\begin{align*}
	B_m &\leq  \sum_{l,m=1}^\infty
	\frac{1}{\nu_m R_{\kappa}^{(1)}(l)} 
	E \left[\sup_{\tau\in (0,\mu_m)}
	\norm{h_k(\cdot,\cdot+\tau,\cdot,\cdot)
	-h_k}_{L^1(0,T-\tau;H^{-N}(M\times L_l))}\right]
	\\ & \lesssim \sum_{l,m=1}^\infty
	\frac{\mu_m^\alpha}{\nu_m}
	\frac{C_l}{R_{\kappa}^{(1)}(l)}
	\lesssim \sum_{l=1}^\infty
	\frac{\bar C_l}{R_{\kappa}^{(1)}(l)},
\end{align*}
which can be made $\leq \kappa/22$ by taking 
$R_{\kappa}^{(1)}(l)=2^{-l}\bar C_lR_{\kappa}^{(1)}$ with 
$R_{\kappa}^{(1)}$ large enough (uniformly in $k$). 
Hence, we conclude that $\prob\bigl(X_k^{(1)}\in 
(\cK_{\kappa}^{(1)})^c\bigr)\le A+B_m\leq \kappa/11$. 
Consequently, we can take the compact 
set $\cC_\kappa^{(1)}$ in \eqref{eq:tight-marginals}
as the closure of $\cK_\kappa^{(1)}$. 

Next, by the Banach-Alaoglu theorem 
and the reflexivity of $L^2$, the set
$$
\cC_\kappa^{(3)}:=
\Seq{u:\norm{u}_{L^2((0,T)\times M\times \R)}
\leq R_{\kappa}^{(3)}}
$$
is a compact subset of $L^2_w((0,T)\times M\times \R)$, 
for any $R_{\kappa}^{(3)}>0$. By the Chebyshev inequality 
and \eqref{eq:basic-kinetic-est},
\begin{align*}
	\prob\bigl(X_k^{(3)}\in 
	(\cC_{\kappa}^{(3)})^c\bigr)
	& =\prob\left(\Seq{\omega\in 
	\Omega: \norm{g_k^{(1)}}_{L^2((0,T)\times M\times \R)}
	>R_{\kappa}^{(3)}}\right)
	\\ & \leq \frac{1}{\bigl(R_{\kappa}^{(3)}\bigr)^2} 
	E\left[\norm{g_k^{(1)}}_{L^2((0,T)\times M\times \R)}^2\right]
	\lesssim\frac{1}{\bigl(R_{\kappa}^{(3)}\bigr)^2},
\end{align*}
which can be made $\leq \kappa/11$ by taking 
$R_{\kappa}^{(3)}$ large enough (uniformly in $k$), so 
that \eqref{eq:tight-marginals} holds for $l=3$. 
Similarly, we can produce compact subsets 
$\cC_{\kappa}^{(4)}$, $\cC_{\kappa}^{(5)}$, $\cC_{\kappa}^{(8)}$, 
$\cC_{\kappa}^{(9)}$ of $L^2_w((0,T)\times M\times \R)$ 
such that \eqref{eq:tight-marginals} holds 
for $l=4,5,8,9$.

Let $\Seq{L_l}_{l\in\N}$ be a sequence of sets
$L_l\subset\subset \R$ such that $L_l\uparrow \R$ 
as $l\to \infty$, and let 
$\Seqb{R_{\kappa}^{(4)}(l)}_{l\in\N}$ 
be a sequence of positive numbers.
By the Banach-Alaoglu theorem (and 
a diagonal argument), the set 
$$
\cC_\kappa^{(2)}:=
\Seq{u:\norm{u}_{L^2((0,T)\times M\times L_l)}
\leq R_{\kappa}^{(2)}(l), \, l\in \N}
$$
is a compact subset of $L^2_{\loc,w}((0,T)\times M\times \R)$. 
By the Chebyshev inequality 
and \eqref{eq:basic-kinetic-est},
\begin{align*}
	\prob\bigl(X_k^{(2)}\in 
	(\cC_{\kappa}^{(2)})^c\bigr)
	& \leq \sum_{l=1}^\infty
	\prob\left(\Seq{\omega\in 	
	\Omega:\norm{D^f_k}_{L^2((0,T)\times M\times L_l)}
	>R_{\kappa}^{(2)}(l)}\right)
	\\ & \leq \frac{1}{\bigl(R_{\kappa}^{(2)}(l)\bigr)^2} 
	E\left[\norm{D^f_k}_{L^2((0,T)\times M\times L_l)}^2\right]
	\leq \sum_{l=1}^\infty
	\frac{C_l}{\bigl(R_{\kappa}^{(2)}(l)\bigr)^2},
\end{align*}
which can be made $\leq \kappa/11$ by taking 
$\bigl(R_{\kappa}^{(2)}(l)\bigr)^2
=2^{-l}C_lR_{\kappa}^{(2)}$ with 
$R_{\kappa}^{(2)}$ large enough (uniformly in $k$). 
Hence, we conclude that $\prob\bigl(X_k^{(2)}\in 
(\cC_{\kappa}^{(2)})^c\bigr)\le \kappa /11$, 
and so \eqref{eq:tight-marginals} holds for $l=2$. 
Similarly, we can find a compact subset 
$\cC_{\kappa}^{(11)}$ of $L^2_{\loc,w}(M\times \R)$ 
such that \eqref{eq:tight-marginals} holds for $l=11$. 

By the Banach-Alaoglu theorem, the set
$\cC_\kappa^{(7)}:=\Seq{u:\norm{u}_{\cM((0,T)\times M\times \R)}
\leq R_{\kappa}^{(7)}}$ is a compact subset of 
$\cM_{w\star}((0,T)\times M \times \R)$, 
for any $R_{\kappa}^{(7)}>0$. By the Chebyshev inequality 
and \eqref{eq:basic-kinetic-est},
\begin{align*}
	\prob\bigl(X_k^{(7)}\in 
	(\cC_{\kappa}^{(7)})^c\bigr)
	& =\prob\left(\Seq{\omega\in 
	\Omega: \norm{D_k}_{\cM((0,T)\times M\times \R)}
	>R_{\kappa}^{(7)}}\right)
	\\ & \leq 
	\frac{1}{\bigl(R_{\kappa}^{(7)}\bigr)^2} 
	E\left[\norm{D_k}_{\cM((0,T)\times M\times \R)}^2\right]
	\lesssim\frac{1}{\bigl(R_{\kappa}^{(7)}\bigr)^2},
\end{align*}
which can be made $\leq \kappa/11$ by taking 
$R_{\kappa}^{(7)}$ large (uniformly in $k$). 
Thus \eqref{eq:tight-marginals} holds for $l=7$. 
Similarly, we can supply a compact subset 
$\cC_{\kappa}^{(6)}$ of $\cM_{\loc,w\star}(M\times \R)$ 
such that \eqref{eq:tight-marginals} holds for $l=6$.

Finally, since the law of $W$ is tight as a 
Radon measure on the Polish space $C([0, T])$, 
we can find a compact subset $\cC_{\kappa}^{(10)}$ 
of $C([0,T])$ such that 
$\prob\bigl(W\in (\cC_{\kappa}^{(10)})^c\bigr)
\leq \kappa/11$, thereby 
verifying \eqref{eq:tight-marginals} for $l=10$.
\end{proof}

\begin{lemma}[Skorokhod-Jakubowski 
a.s.~representations]\label{lem:as-repr}
(i) By passing to a subsequence (not relabelled), there 
exist a new complete probability space 
$\bigl(\tilde\Omega,\tilde \mcf,\tilde\prob\bigr)$ 
and random variables 	
\begin{equation}\label{eq:tilde-Xk}
	\begin{split}
		& \tilde X_k=\bigl(\tilde X_k^{(1)},\ldots, 
		\tilde X_k^{(11)}\bigr),\quad k\in \N, 
		\quad \text{with joint laws $\tilde\cL_k=\cL_k$},
		\quad \text{where},
		\\ & \tilde X^{(1)}_k:=\tilde h_k, \,\,
		\tilde X^{(2)}:= \tilde D^f_k,\,\,
		\tilde X^{(3)}_k:=\tilde g_k^{(1)},\,\,
		\tilde X^{(4)}_k:=\tilde g_k^{(2)},\,\,
		\tilde X^{(5)}_k:=\tilde g_k^{(3)},\,\,
		\tilde X^{(6)}_k:=\tilde g_k^{(4)},\,\,
		\\ & 
		\tilde X^{(7)}_k:=\tilde D_k, \,\,
		\tilde X^{(8)}_k:=\tilde \Phi_k^{(1)}, \,\,
		\tilde X^{(9)}_k:=\tilde \Phi_k^{(2)},\,\,
		\tilde X^{(10)}_k:=\tilde W_k, \,\,
		\tilde X^{(11)}_k:=\tilde h_{0,k},
	\end{split}
\end{equation}
and
\begin{equation}\label{eq:tilde-X}
	\begin{split}
		& \tilde X=\bigl(\tilde X^{(1)},\ldots, 
		\tilde X^{(11)}\bigr), 
		\quad \text{where},
		\\ & \tilde X^{(1)}:=\tilde h, \,\,
		\tilde X^{(2)}:= \tilde D^f,\,\, 
		\tilde X^{(3)}:=\tilde g^{(1)},\,\,
		\tilde X^{(4)}:=\tilde g^{(2)},\,\,
		\tilde X^{(5)}:=\tilde g^{(3)},\,\,
		\tilde X^{(6)}:=\tilde g^{(4)},\,\,
		\\ & 
		\tilde X^{(7)}:=\tilde D, \,\,
		\tilde X^{(8)}:=\tilde \Phi^{(1)}, \,\,
		\tilde X^{(9)}:=\tilde \Phi^{(2)},\,\,
		\tilde X^{(10)}:=\tilde W, \,\,
		\tilde X^{(11)}:=\tilde h_{0},
	\end{split}
\end{equation}
such that $\tilde{X}_k\to X$ 
almost surely (in the topology of $\cX$), 
i.e., as $k\to \infty$,
\begin{equation}\label{eq:tilde-conv}
	\begin{split}
		& \tilde h_k
		\to \tilde h 
		\quad \text{in $L^2_t\bigl(L^2_{\loc,w}(M\times \R)\bigr)$, 
		$\tilde \prob$--a.s.},
		\\ & 
		\tilde g_k^{(1)}	
		\weak 
		\tilde g^{(1)}, \,\,
		\tilde g_k^{(2)}
		\weak \tilde g^{(2)},\,\,
		\tilde g_k^{(3)}	
		\weak \tilde g^{(3)}
		\quad \text{in $L^2((0,T)\times M\times \R)$, 
		$\tilde \prob$--a.s.},
		\\ & 
		\tilde \Phi_k^{(1)}
		\weak 
		\tilde \Phi^{(1)},\,\,
		\tilde \Phi_k^{(2)}
		\weak 
		\tilde \Phi^{(2)}
		\quad \text{in $L^2((0,T)\times M\times \R)$, 
		$\tilde \prob$--a.s.},
		\\ &
		\tilde D_k^f	
		\weak \tilde D^f
		\quad \text{in $L^2_{\loc}((0,T)\times M\times \R)$, 
		$\tilde \prob$--a.s.},
		\\ & 
		\tilde g_k^{(4)}
		\weak \tilde g^{(4)}
		\quad \text{in $\cM_{\loc}((0,T)\times M\times \R)$, 
		$\tilde \prob$--a.s.},
		\\ & 
		\tilde D_k	
		\weakstar \tilde D \quad 
		\text{in $\cM((0,T)\times M\times \R)$, 
		$\tilde \prob$--a.s.},
		\\ &
		\tilde W_k	
		\to \tilde W \quad 
		\text{in $C([0,T])$, 
		$\tilde \prob$--a.s.},
		\quad 
		\tilde h_{0,k}
		\weak \tilde h_0 \quad 
		\text{in $L^2_{\loc}(M\times \R)$, 
		$\tilde \prob$--a.s.}
	\end{split}
\end{equation}

(ii) We have the following 
$k$-independent estimates:
\begin{equation}\label{eq:tilde-kinetic-est}
	\begin{split}
		& \tilde E \left[\norm{\left(\tilde g_k^{(1)},
		\tilde g_k^{(2)}, \tilde g_k^{(3)}\right)}_{L^2_{t,\mx,
		\lambda}}^2\right]\tok 0, 
		\quad 
		\tilde E \left[\norm{\tilde D^f_k}_{L^2_{t,\mx,\lambda,\loc}}^2
		\right]\tok 0, 
		\\ & 
		\tilde E \left[\norm{\tilde g_k^{(4)}}_{\cM_{t,\mx,\lambda,\loc}}^2
		\right]\lesssim 1, \quad 
		\tilde E\left[\norm{\tilde D_k}_{\cM_{t,\mx,\lambda}}^2
		\right]\lesssim 1,\quad 
		\tilde E\left[\norm{\left(\tilde \Phi_k^{(1)},
		\tilde \Phi_k^{(2)}\right)}_{L^2_{t,\mx,\lambda}}^2
		\right]\lesssim 1,
	\end{split}
\end{equation}
where $\tilde E$ denotes the expectation 
on $\bigl(\tilde\Omega,\tilde \mcf,\tilde\prob\bigr)$. 
In particular, $\tilde D^f=0$ and 
$\tilde g^{(l)}=0$ for $l=1,2,3$.

(iii) For each $k\in \N$, there exist Borel measurable 
functions $\tilde u_k=\tilde u_k(\omega,t,\mx)$, 
$\tilde u_{0,k}=\tilde u_{0,k}(\omega,\mx)$ such that
\begin{equation}\label{eq:tilde-sign-functions}
	\begin{split}
		& \tilde h_k(\omega,t,\mx,\lambda) 
		=\sign\bigl(\tilde u_k(\omega,t,\mx)
		-\lambda\bigr),
		\quad \text{for a.e.~$(\omega,t,\mx,\lambda)
		\in \tilde \Omega\times (0,T)\times M\times \R$.}
		\\ &
		\tilde h_{0,k}(\omega,\mx,\lambda) 
		=\sign\bigl(\tilde u_{0,k}(\omega,\mx)
		-\lambda\bigr),
		\quad \text{for a.e.~$(\omega,\mx,\lambda)
		\in \tilde \Omega\times M\times \R$.} 
	\end{split}
\end{equation}
Denote by $\cL(u_k)$, $\cL(\tilde u_k)$ 
the laws of $u_k$, $\tilde u_k$, respectively, 
where $u_k$ is a solution to the 
pseudo-parabolic SPDE \eqref{PP-1}--\eqref{PP-ID}. 
Viewing $\cL(u_k)$, $\cL(\tilde u_k)$ as probability 
measures on $L^2_t\bigl(L^2_w(M)\bigr)$, 
the laws $\cL(u_k)$, $\cL(\tilde u_k)$ coincide. 
Further, $\tilde u_k$ takes values 
in $C([0,T];L^2(M))\cap L^\infty(0,T;H^1(M))
\cap L^2(0,T;H^2(M))$, $\tilde u_k$ 
exhibits the $k$-dependent estimates
\begin{equation}\label{eq:tilde-uk-kdep-est}
	\tilde E \left [ 
	\norm{\tilde u_k}_{L^\infty(0,T;H^1(M))}^2\right],
	\, \tilde E \left [ 
	\norm{\tilde u_k}_{L^2(0,T;H^2(M))}^2\right]\lesssim_k 1,
\end{equation}
and the $k$-independent estimate
\begin{equation}\label{eq:tilde-uk-Linf-L2-est}
	\tilde E \left [
	\norm{\tilde u_k}_{L^\infty(0,T;L^2(M))}^2
	\right]\lesssim 1.
\end{equation}
Similarly, viewing the laws $\cL(u_0)$ of $u_0$ 
and $\cL(\tilde u_{k,0})$ of $\tilde u_{0,k}$ 
as probability measures on $\bigl(L^2(M),\tau_w\bigr)$, 
$\cL(u_0)$ and $\cL\bigl(\tilde u_{k,0}\bigr)$ 
coincide. Besides, $\tilde u_{k,0}$ takes values 
in $H^2(M)$ and satisfies the two estimates 
\begin{equation}\label{eq:tilde-u0k-est}
	\tilde E \left [
	\norm{\tilde u_{0,k}}_{H^2(M)}^2\right]\lesssim_k 1, 
	\quad 
	\tilde E \left [\norm{\tilde u_{0,k}}_{L^2(M)}^2
	\right]\lesssim 1.
\end{equation} 
\end{lemma}	

\begin{proof}
The first part (i) of the lemma---in particular 
\eqref{eq:tilde-conv}---is an immediate consequence of 
applying Lemma \ref{lem:tight} (tightness) and 
Theorem \ref{thm:Jak-Skor} (Jakubowski-Skorokhod) 
to the mappings $X_k=\bigl(X_k^{(1)},\ldots,
X_k^{(11)}\bigr)$ and their joint laws $\cL_k$, 
defined respectively in \eqref{eq:Xk-def} 
and \eqref{eq:Xk-joint-laws}.  Regarding part (ii), the 
estimates in \eqref{eq:tilde-kinetic-est} follow from 
the equality of laws ($\tilde \cL_k=\cL_k$) and 
the corresponding estimates in \eqref{eq:basic-kinetic-est}.  

Next, let us prove \eqref{eq:tilde-sign-functions}. Since the 
laws of $h_k=\sign\bigl(u_k-\lambda\bigr)$ 
and $\tilde{h}_k$ are the same, we claim that
\begin{equation}\label{=}
	\prob\bigl(\left\{a<h_k(\cdot,t,\mx,\lambda)<b\right\}\bigr)
	=\tilde{\prob}\left(\left\{a<\tilde{h}_k(\cdot,t,\mx,\lambda)
	<b\right\}\right)
\end{equation} 
for every $(a,b) \subset\R$ and 
a.e.~$(t,\mx,\lambda)\in (0,T)\times M \times \R$. 
Indeed, if there exist $(a,b)\subset \R$ 
and $I\subset (0,T)\times M\times L$, with $|I|>0$, such that
$$
\prob\bigl(\left\{a<h_k(\cdot,t,\mx,\lambda)<b\right\}\bigr)
\, >\, 
\tilde{\prob}\left(\left\{a<\tilde{h}_k(\cdot,t,\mx,\lambda)
<b\right\}\right),
\qquad (t,\mx,\lambda) \in I, 
$$
then it follows that
\begin{align*}
	& \prob\Bigl(\Bigl\{a \, |I| \, 
	<\iiint_{(0,T)\times M \times \R}  
	h_k(\cdot,t,\mx,\lambda)
	\chi_I(t,\mx,\lambda) \,d\lambda \,d\mx\, dt \, 
	< b\,|I|\Bigr\}\Bigr) 
	\\ & \qquad
	> \, \tilde{\prob}
	\Bigl(\Bigl\{a\, |I|\,
	< \iiint_{(0,T)\times M \times \R} 
	\tilde{h}_k(\cdot,t,\mx,\lambda) 
	\chi_I(t,\mx,\lambda) \, d\lambda\, d\mx\, dt 
	< \,b \, |I| \Bigr\}\Bigr),
\end{align*}
where $\chi_{I}$ denotes the characteristic function 
of the set $I$; without loss of generality, we 
may assume that the measure of $I$ is finite.
This strict inequality ($>$) contradicts the equality of laws 
of $h_k$ and $\tilde{h}_k$ as random variables 
taking values in the space $L^2_t\bigl(L^2_{\loc,w}(M\times\R)\bigr)$. 
Similarly, a contradiction arises when ``$>$" is replaced by ``$<$". 
This proves the claim \eqref{=}.

The relation \eqref{=} implies that the range 
of $\tilde{h}_k$ consists of the two values 
$\Seq{-1,1}$, and also that $\tilde{h}_k$ is nonincreasing 
in $\lambda$. To see this, consider any interval $(a,b)$ 
not containing $-1$ or $1$. Then, cf.~\eqref{=},
$$
0=\prob\bigl(\left\{a<h_k<b\right\}\bigr)
=\tilde{\prob}\left(\left\{a<\tilde{h}_k<b\right\}\right),
$$ 
and thus it follows that $\tilde{h}_k$ is 
almost surely equal to $-1$ or $1$. Similarly, since $h_k$ 
is nonincreasing with respect to $\lambda$, its 
$\lambda$-derivative is a nonpositive measure. 
Thus, for any nonnegative $\rho \in C_c(\R)$, we have
$$
1=\prob\bigl(\left\{\left\langle\pa_\lambda h_k,\rho
\right \rangle\leq 0\right\}\bigr)
=\tilde{\prob}\left(\left\{\left\langle 
\pa_\lambda \tilde{h}_k,\rho \right\rangle\leq 0\right\}\right),
$$
which clearly implies that $\tilde{h}_k$ 
is nonincreasing with respect to $\lambda$.
 
Combining our findings, there exists a measurable function 
$\tilde{u}_k$ such that 
\begin{equation}\label{sol-form-tilde}
	\tilde{h}_k(\omega,t,\mx,\lambda)
	=\sign\bigl(\tilde{u}_k(\omega,t,\mx)-\lambda\bigr).
\end{equation}
This proves the first part of \eqref{eq:tilde-sign-functions}. 
The second part follows in the same way.

Viewing $u_k$ and $\tilde u_k$ as random variables 
taking values in $L^2_t\bigl(L^2_w(M)\bigr)$, we claim that $u_k$ 
and $\tilde{u}_k$ have the same laws. To verify this, denote   
by $\chi_L(\lambda)$ the characteristic function of the 
interval $(-L,L)$, $L\in \R$. We also need the 
truncation function (at level $L>0$):
\begin{equation}\label{eq:truncate-L}
	\cT_L:\R\to\R, 
	\quad \cT_L(u):=\max\bigl(-L,\min(u,L)\bigr)=
	\begin{cases} 
		-L, & \text{if $u<-L$} \\
		u, & \text{if $-L\leq u \leq L$} \\
		L, & \text{if $u>L$}
	\end{cases}.
\end{equation}
As the laws of $h_k$ and $\tilde{h}_k$ coincide, 
we obtain, for a.e.~$(t,\mx)\in (0,T)\times M$, 
\begin{align*}
	& \prob\left(\left\{
	a<
	\cT_L\bigl(u_k(\cdot,t,\mx)\bigr) 
	<b\right\}\right)
	=\prob\left(\left\{
	a<
	\frac{1}{2}\int_\R h_k(\cdot,t,\mx,\lambda)
	\chi_L(\lambda)\, d\lambda 
	<b\right\}\right)
	\\ & \quad
	= \tilde{\prob}\left(\left\{
	a<
	\frac{1}{2}\int_\R 
	\tilde{h}_k(\cdot,t,\mx,\lambda)
	\chi_L(\lambda)\, d\lambda 
	<b\right\}\right)
	=\tilde{\prob}\left(\left\{
	a<
	\cT_L\bigl(\tilde{u}_k(\cdot,t,\mx)\bigr) 
	<b\right\}\right),
\end{align*} 
where \eqref{sol-form-tilde} is used 
to obtain the last equality. In other words, the laws 
of $\cT_L\bigl(u_k\bigr)$ 
and $\cT_L\bigl(\tilde{u}_k\bigr)$ 
are the same for \textit{every} $L>0$, which implies that 
the laws of $u_k$ and $\tilde{u}_k$ coincide. 
Similarly, we can prove that the laws of $u_0$ and 
$\tilde{u}_{0,k}$ coincide.

Let us argue that $\tilde u_k$ takes 
values in the indicated spaces. By Theorem 
\ref{unique_sol_par}, $u_k$ takes values in these spaces.
First, $C([0,T];L^2(M))$ is a Polish space 
that is continuously embedded 
in the quasi-Polish space $L^2_t\bigl(L^2_w(M)\bigr)$. 
Thus, by the Lusin-Suslin-Kuratowski theorem, 
$C([0,T];L^2(M))$ is a Borel set in $L^2_t\bigl(L^2_w(M)\bigr)$ 
and so, by the equality of laws, $\tilde u_k$ takes values 
in $C([0,T];L^2(M))$. Similarly, $\tilde u_k$ takes values 
in $L^2(0,T;H^2(M))$. We can also argue that 
$\tilde{u}_k$ takes values in $L^\infty(0,T;H^1(M))$.
Clearly, by \eqref{eq:pseudo-spaces}, 
$E\norm{u_k}_{L^\infty([0,T];L^2(M))}^2$ 
is bounded uniformly in $k$. Therefore, 
by the equality of laws, we conclude that 
\eqref{eq:tilde-uk-Linf-L2-est} holds. 
Similarly, we verify that 
\eqref{eq:tilde-uk-kdep-est} holds. 
The claims about $\tilde u_{0,k}$ can 
be proved in the same way.
\end{proof}

We need to introduce a filtration 
$\Seqb{\tilde \F_t^k}$  
for the new mappings $\tilde X_k$, 
cf.~\eqref{eq:tilde-Xk}. Set 
\begin{equation}\label{eq:tilde-Ftk}
	\tilde \F_t^k=\sigma\left(
	\sigma \Bigl(\tilde X_k\big |_{[0,t]}\Bigr)
	\cup \Seq{N \in \tilde \F: \tilde P (N)=0}
	\right).
\end{equation}
The filtration $\Seqb{\tilde \F_t^k}$ is the smallest 
one that makes all the components of 
$\tilde X_k=\bigl(\tilde X_k^{(1)},\ldots,
\tilde X_k^{(11)} \bigr)$ adapted.  
Recall that a Wiener process is fully 
determined by its law. Since $\tilde W_k$ and $W$ 
have the same laws, L\'{e}vy's martingale 
characterization of a Wiener 
process (see e.g.~\cite[Theorem 4.6]{DaPrato:2014aa}, 
implies that each $\tilde W_k$ is a Wiener 
process with respect to its canonical filtration. 
Besides, each $\tilde W_k$ is a 
Wiener process relative to the filtration 
$\Seqb{\tilde \F_t^k}$ defined by \eqref{eq:tilde-Ftk}. 
The proof of this claim is standard, verifying 
that $\tilde W_k(t)$ is $\Seqb{\tilde \F_t^k}$ 
measurable and that $\tilde W_k(t)-\tilde W_k(s)$ 
is independent of $\tilde \F_s^k$, for all $s< t$. 
These properties are consequences of 
the facts that $\tilde W_k,W$ share 
the same law, $W(t)$ is $\F_t$ 
measurable, and $W(t)-W(s)$ is 
independent of $\F_s$.

It remains to demonstrate that the tilde variables 
$\tilde h_k,\ldots,\tilde W_k,\tilde h_{0,k}$ satisfy 
kinetic SPDEs like \eqref{eq:kinetic-new} 
on the new probability space 
$\bigl(\tilde\Omega,\tilde \mcf,\tilde\prob\bigr)$. 
There exist several approaches for doing this, see 
for example \cite{Bensoussan:1995aa,Brzezniak:2011aa,Ondrejat:2010aa}.
In the framework of random distributions, 
\cite[Theorem 2.9.1]{Breit:2018aa} supplies a 
result of this type for a general class of SPDEs (the 
proof of which can be applied here).

\begin{lemma}[SPDE on new probability space]
\label{lem:SPDE-on-new-prob-space}
(i) On the filtered probability space 
$\bigl(\tilde\Omega,\tilde \mcf,\tilde\prob,
\Seqb{\tilde \F_t^k}\bigr)$ given by 
Lemma \ref{lem:as-repr} and \eqref{eq:tilde-Ftk}, 
the Skorokhod-Jakubowski 
representations $\tilde X_k=\bigl(\tilde h_k,\ldots,
\tilde W_k,\tilde h_{0,k}\bigr)$ satisfy 
the following SPDE in $\cD_{t,\mx,\lambda}'$, 
$\tilde \prob$--a.s.: 
\begin{equation}\label{eq:tilde-kinetic}
	\begin{split}
		d \tilde h_k+\Div_\mx \bigl(F \tilde h_k\bigr)\, dt
		& =\Div_{\mx} \left(\tilde D^f_k
		+\tilde g_k^{(1)}+\tilde g_k^{(2)}\right)\,dt
		+\pa_{\lambda}\Div_{\mx} \tilde g_k^{(3)}\,dt
		+\pa^2_{\lambda\lambda} \tilde g_k^{(4)}\,dt
		+\pa_{\lambda} \tilde D_k\, dt
		\\ & \qquad 
		+\left[\tilde \Phi_k^{(1)}
		+\pa_{\lambda}\tilde \Phi_k^{(2)}\right]\,d\tilde W_k(t), 
		\qquad \tilde h_k\big |_{t=0}=\tilde h_{0,k},
	\end{split}
\end{equation} 
that is, for every $\varphi\in 
\cD([0,T)\times M\times \R)$,
\begin{equation}\label{eq:tilde-kinetic-weak}
	\begin{split}
		&\int_0^T\int_M\int_{\R}\tilde h_k 
		\pa_t\varphi \, d\lambda \, dV(\mx) \, dt
		+\int_{\R^d} \int_{\R}\tilde h_{0,k} 
		\varphi(0,\mx)\, d\lambda\, dV(\mx)
		+\int_0^T\int_{\R^d}\int_{\R} \bigl(F\tilde h_k\bigr) 
		\cdot \nabla_{\mx}\varphi \, d\lambda\, dV(\mx)\,dt
		\\ & \qquad =\int_0^T \int_{\R^d}\int_{\R}
		\left(\tilde D_k^f+\tilde g_k^{(1)}+ \tilde g_k^{(2)}\right)
		\cdot \nabla_{\mx}\varphi \, d\lambda \, dV(\mx)\, dt
		+\int_0^T \int_{\R^d}\int_{\R}
		\tilde D_k\pa_{\lambda}\varphi \, d\lambda \, dV(\mx)\, dt
		\\ & \qquad \qquad 
		-\int_0^T \int_{\R^d}\int_{\R}
		\tilde g_k^{(3)}\cdot \bigl(\pa_{\lambda}\nabla_{\mx}
		\varphi\bigr) \, d\lambda \, dV(\mx)\, dt
		-\int_0^T \int_{\R^d}\int_{\R}
		\tilde g_k^{(4)}\pa^2_{\lambda\lambda}
		\varphi \, d\lambda \, dV(\mx)\, dt
		\\ & \qquad \qquad\qquad
		-\int_0^T\int_{\R^d}\int_{\R}\left( 
		\tilde \Phi_k^{(1)}\varphi 
		- \tilde \Phi_k^{(2)} \pa_{\lambda}\varphi 
		\right)\, d\lambda \,dV(\mx) \, d\tilde W_k(t), 
		\quad \text{$\tilde \prob$--almost surely}.
	\end{split}
\end{equation}
where the $k$-dependent quantities in \eqref{eq:tilde-kinetic} 
satisfy the properties and bounds listed 
in \eqref{eq:tilde-kinetic-est}, 
\eqref{eq:tilde-sign-functions}.

(ii) The process $\tilde u_k \in 
L_{\tilde \prob}^2\left(\tilde \Omega;
L^\infty(0,T;H^1(M))\cap C([0,T];L^2(M))\cap 
L^2(0,T;H^2(M))\right)$---defined via 
\eqref{eq:tilde-sign-functions}---is 
adapted to $\Seqb{\tilde \F_t^k}$ and satisfies 
the following SPDE in the $\mx$-weak 
sense of \eqref{mild-ppde} on $(0,T)\times M$: 
\begin{equation}\label{PP-1-SJ}
	\begin{split}
		& d \tilde u_k 
		+\Div_{\mx} \mff_k(\mx, \tilde u_k) 
		\, dt =\eps_k \Delta \tilde u_k \, dt
		+\delta_k  \, d \Delta  \tilde u_k 
		+\Phi(\mx, \tilde u_k)\, d\tilde W_k(t), 
		\quad 
		\tilde u_k\big|_{t=0}=\tilde u_{0,k}\in H^2(M).
	\end{split}
\end{equation}
In addition, $\tilde u_k$, $\tilde u_{0,k}$ satisfy the 
estimates \eqref{eq:tilde-uk-kdep-est}, 
\eqref{eq:tilde-uk-Linf-L2-est}, and 
\eqref{eq:tilde-u0k-est}.
\end{lemma}

\begin{proof}
For any fixed test function $\varphi 
\in \cD([0,T)\times M\times \R)$, consider the function 
$F_\varphi: \bigl(\cX,\cB_\cX\bigr) \to \R$, 
$\cX \ni X\mapsto F_\varphi(X)\in \R$, defined such that  
the claim \eqref{eq:tilde-kinetic-weak} takes the form 
$F_\varphi\bigl(\tilde X_k\bigr)=0$, 
$\tilde\prob$-almost surely. Recall that the 
law $\tilde \cL_k$ of $\tilde X_k$ 
and the law $\cL_k$ of $X_k$ coincide, see Lemma 
\ref{lem:as-repr}. Then, using the same procedure 
as in the proof of \cite[Theorem 2.9.1]{Breit:2018aa} (see also 
\cite{Bensoussan:1995aa}), we conclude that
$\cL_{\R}\bigl(F_\varphi\bigl(\tilde X_k\bigr)
\bigr)= \cL_{\R}\bigl(F_\varphi(X_k)\bigr)$. 
Combining this with the knowledge that 
$X_k$ satisfies the kinetic 
SPDE \eqref{kinetic} in $\cD'_{t,\mx,\lambda}$, 
$\prob$--a.s., the claim \eqref{eq:tilde-kinetic-weak} follows. 
Similarly, since $\tilde u_k\sim u_k$, 
$\tilde u_{0,k}\sim u_0$ and $u_k$ satisfies 
the SPDE \eqref{PP-1} with initial data $u_0$, it follows 
that $\tilde u_k, \tilde u_{0,k}$ satisfy 
\eqref{PP-1-SJ}. Here we have also used that 
$\tilde u_k$ is $\Seqb{\tilde \F_t^k}$--adapted. 
To see this, recall that $\tilde u_k$ is defined via 
\eqref{eq:tilde-sign-functions} and that 
$\tilde h_k$ is $\Seqb{\tilde \F_t^k}$--adapted.
\end{proof}

\begin{remark}\label{rem:compare-kinetic}
The kinetic SPDE \eqref{eq:tilde-kinetic} can be written 
in the form \eqref{eq-1} (with $m=1$, $\lambda\in \R$) used 
by the stochastic velocity averaging result (Corollary 
\ref{cor:velocity-average-M}). 
Indeed, for any $\rho\in C^\infty_c(\R)$, we make 
the following identifications:
\begin{align*}
	g_{k,\rho} & = \action{g_{k}}{\rho}
	:=\Div \left( \action{\tilde D_k^f
	+\tilde{g}^{(1)}_k+\tilde{g}^{(2)}_k}{\rho}
	-\action{\tilde{g}^{(3)}_k}{\rho'}
	\right),
	\\ 
	G_{k,\rho} & = \action{G_{k}}{\rho}:=
	\action{\tilde{g}_k^{(4)}}{\rho''}
	+\action{\tilde{D}_k}{\rho'},
	\\ 
	\Phi_{k,\rho} & =\action{\Phi_{k}}{\rho}:=
	\action{\tilde{\Phi}_k^{(1)}}{\rho}
	+\action{\tilde{\Phi}_k^{(2)}}{\rho'}.
\end{align*}
In view of \eqref{eq:tilde-conv}, as $k\to \infty$,
$$
G_{k,\rho} \overset{\star}{\rightharpoonup} 
G_\rho=\action{G}{\rho}:= \action{\tilde{g}^{(4)}}{\rho''}
+\action{\tilde{D}}{\rho'}
\quad \text{in $\cM((0,T)\times M)$, 
$\tilde \prob$--a.s.}
$$
Besides, by \eqref{eq:tilde-kinetic-est}, 
$E\left[\norm{G_{k,\rho}}_{\cM((0,T)\times M)}^2\right]
\lesssim 1$. This verifies the assumption (iii) 
of Lemma \ref{lem:weak-limit-kinetic} (with 
$\R^d$ replaced by $M$). 

Similarly, by \eqref{eq:tilde-conv}, 
$$
g_{k,\rho} \rightharpoonup 
g_\rho =\action{g}{\rho}
:=\Div \left( \action{\tilde D^f
+\tilde{g}^{(1)}+\tilde{g}^{(2)}}{\rho}
-\action{\tilde{g}^{(3)}}{\rho'}
\right)
$$
in $L^2(0,T;W^{-1,2}(M))$, $\tilde \prob$--a.s., as $k\to \infty$. 
Unfortunately, this weak convergence is not enough to 
deliver the assumption (ii) with $r=2$. However, 
by the estimates in \eqref{eq:tilde-kinetic-est}, 
this convergence is, in fact, strong and the 
limit $g$ is zero. Indeed, the first part 
of \eqref{eq:tilde-kinetic-est} implies
$$
E\left[\norm{g_{k,\rho}}_{L^2(0,T;W^{-1,2}(M))}^2\right]
\tok 0,
$$
and so we conclude $g\equiv 0$; in other words, $g_{k,\rho} \tok 
\action{g}{\rho}=0$ in $L^2(0,T;W^{-1,2}(M))$, $\tilde \prob$--a.s., 
for any $\rho\in C^\infty_c(\R)$. This shows that the assumption (ii) 
of Lemma \ref{lem:weak-limit-kinetic} holds (on the manifold $M$).

Next, by \eqref{eq:tilde-conv} and 
\eqref{eq:tilde-kinetic-est},
\begin{align*}
	&\Phi_{k,\rho}\rightharpoonup 
	\Phi_{\rho}=\action{\Phi}{\rho}
	:=\action{\tilde{\Phi}^{(1)}}{\rho}
	+\action{\tilde{\Phi}^{(2)}}{\rho'}
	\quad 
	\text{in $L^2((0,T)\times M)$, $\tilde \prob$--a.s.},
\end{align*} 
and $E\left[\norm{\Phi_{k,\rho}}_{L^2((0,T)\times M)}^2\right]
\lesssim 1$; Due to the first convergence given in 
\eqref{eq:tilde-conv} for $\tilde h_k$, coupled with the 
specific form of $\Phi_{k,\rho}$ that lets us express it using 
$\tilde h_k$, we are able to enhance the convergence 
of $\Phi_{k,\rho}$ to $\Phi_{\rho}$ in the space 
$L^2_t\bigl(L^2_{\loc,w}(\R^d)\bigr)$. 
Thus, (iv) of Lemma \ref{lem:weak-limit-kinetic} holds (on $M$). 

Finally, by the first part of \eqref{eq:tilde-conv}, the assumption 
(i) of Lemma \ref{lem:weak-limit-kinetic} is fulfilled (on $M$). 
Now, since all the assumptions (i)--(iv) hold, we may apply 
Corollary \ref{cor:velocity-average-M} to 
the kinetic SPDE \eqref{eq:tilde-kinetic} (see 
next subsection). 
\end{remark}

\subsection{Strong convergence and proof 
of Theorem \ref{main-thm}} 

A filtered probability space 
$\bigl(\tilde\Omega,\tilde \mcf,\tilde\prob,
\Seqb{\tilde \F_t}\bigr)$ is needed for the limit $\tilde X$
of $\Seqb{\tilde X_k}$, cf.~\eqref{eq:tilde-Xk} and 
\eqref{eq:tilde-X}. Here we specify 
\begin{equation}\label{eq:tilde-Ft}
	\tilde \F_t :=\sigma\left(
	\sigma \Bigl(\tilde X\big |_{[0,t]}\Bigr)
	\cup \Seq{N \in \tilde \F: \tilde P (N)=0}
	\right).
\end{equation}
The filtration $\Seqb{\tilde \F_t}$ is the smallest 
one that makes all the components of 
$\tilde X =\bigl(\tilde h,\ldots,\tilde W,
\tilde h_0\bigr)$ adapted. Recall that 
$\tilde W_k$ is a Wiener process with 
respect to $\bigl(\tilde\Omega,\tilde \mcf,
\tilde\prob,\Seqb{\tilde \F_t^k}\bigr)$.
Since $\tilde W_k\to \tilde W$ in $C([0,T])$, $\tilde \prob$--a.s., 
it is easy to see that also the limit 
$\tilde W$ is a Wiener process with respect to $\bigl(\tilde\Omega,
\tilde \mcf,\tilde\prob,\Seqb{\tilde \F_t^k}\bigr)$, 
see for example \cite[Lemma 9.9]{Ondrejat:2010aa} or 
\cite[Lemma 2.9.3]{Breit:2018aa}.

\begin{theorem}[strong convergence]\label{thm:strong-conv}
Suppose conditions ($C_f$--1)--($C_f$--4) 
and ($C_\Phi$--1)--($C_\Phi$--2) hold, and that the 
flux $\mff$ is non-degenerate in 
the sense of Definition \ref{def-non-deg}. 
Suppose the diffusion-capillarity 
parameters $\eps_k,\delta_k$ satisfy the 
relation \eqref{neps}. 

(i) Consider the sequence $\Seqb{\tilde{h}_k}$ of 
Skorokhod-Jakubowski representations of $\Seq{h_k}$, 
cf.~\eqref{sol-form} and Lemma \ref{lem:as-repr}. Then 
the averaged quantities 
$$
\Seq{\inn{\tilde h_k,\rho}(\omega,t,\mx)
=\int_{\R} \tilde{h}_k(\omega,t,\mx,\lambda) 
\rho(\lambda) \, d\lambda}_{k\in \N}, 
\quad \rho\in L^2_{\loc}(\R), 
$$ 
converge strongly in 
$L^2\bigl(\tilde \Omega\times (0,T)\times M)\bigr)$, 
along a non-relabelled subsequence. More precisely, 
\begin{equation}\label{eq:thk-strong-conv}
	\tilde h_k \tok \tilde h 
	\quad
	\quad \text{in $L^2_{\omega,t,\mx}
	\bigl(L^2_{\loc,w}(\R)\bigr)
	:=L^2\bigl(\tilde \Omega\times (0,T)
	\times M;L^2_{\loc,w}(\R)\bigr)$},
\end{equation}
where $\tilde h$ is the limit identified 
in \eqref{eq:tilde-conv}.

(ii) The corresponding sequence $\Seq{\tilde u_k}$ 
of processes $\tilde u_k$ defined via 
\eqref{eq:tilde-sign-functions}, see also 
part (ii) of Lemma \ref{lem:SPDE-on-new-prob-space}, 
converges strongly in $L^1\bigl(\tilde \Omega\times (0,T)
\times M\bigr)$ to a limit $\tilde u$:
\begin{equation}\label{eq:tuk-strong-conv}
	\tilde E\left[\norm{\tilde u_k
	-\tilde u}_{L^1((0,T)\times M)}\right]\tok 0, 
	\quad \text{where} \quad
	\tilde E \left[ 
	\norm{\tilde u}_{L^\infty(0,T;L^2(M))}^2
	\right]\lesssim 1.
\end{equation}
The limit $\tilde u$ is adapted to the 
filtration $\Seqb{\tilde \F_t}$ defined 
in $\eqref{eq:tilde-Ft}$. 
\end{theorem}

\begin{proof} 
The convergence statement in \eqref{eq:thk-strong-conv} 
follows directly from Theorem \ref{thm-1} (stochastic 
velocity averaging), in the form 
of Corollary \ref{cor:velocity-average-M}, which 
can be applied in view of Lemmas \ref{lem:as-repr}, 
\ref{lem:SPDE-on-new-prob-space} and Remark \ref{rem:compare-kinetic}. 
The bound on $\tilde u$ in \eqref{eq:tuk-strong-conv} 
is a consequence of the $k$-independent estimate 
\eqref{eq:tilde-uk-Linf-L2-est}. 

Given \eqref{eq:tilde-sign-functions}, we may 
assume that $\norm{\bigl(\tilde h_k,
\tilde h\bigr)}_{L^\infty(\tilde \Omega
\times (0,T)\times M)}\leq 1$. 
Therefore, by \eqref{eq:thk-strong-conv} and a simple 
approximation argument, for any 
$\rho \in L^\infty_c(\R)$,
\begin{equation}\label{conv-hk}
	\int_{\R} \left(\tilde{h}_k-\tilde{h}\right) 
	\rho(\lambda)\, d\lambda 
	\tok 0 \quad \text{strongly in $L^2(\tilde \Omega 
	\times (0,T)\times M)$}.
\end{equation}
We will use this to prove that $\Seqb{\tilde{u}_k}$ converges 
in $L^1\bigl(\tilde{\Omega}\times (0,T)\times M\bigr)$. 
To this end, we need the truncation function 
\eqref{eq:truncate-L}. Using \eqref{conv-hk} with 
$\rho=\chi_{[-L,L]}(\lambda)$, for any 
arbitrary $L>0$, we obtain 
\begin{equation}\label{eq:L1-conv-trunc}
	\cT_L(\tilde u_k)
	\overset{\eqref{eq:tilde-sign-functions}}{=}
	\frac12\int^L_{-L}\tilde h_k\, d\lambda
	\tok \frac12\int_{-L}^L \tilde h\, d\lambda
	=: \tilde u_L 
	\quad \text{in $L^2(\tilde \Omega \times (0,T)\times M)$};
\end{equation}
in particular, $\cT_L(\tilde u_k)\tok \tilde u_L$ 
in $L^1_{\omega,t,\mx}$. Next, 
\begin{equation}\label{eq:trunc-control}
	\begin{split}	
		& \tilde{E}\left[\int_0^T\int_M 
		\abs{\tilde u_k-\cT_L(\tilde u_k)}
		\, dV(\mx)\, dt\right]
		\\ & \qquad \leq 
		\tilde{E}\left[\int_0^T\int_{\Seq{\abs{\tilde u_k>L}}} 
		\abs{L}\, dV(\mx) \, dt\right] 
		\leq \frac{1}{L}\tilde{E}\left[\int_0^T\int_M 
		\abs{\tilde u_k}^2 \, dV(\mx) \, dt\right]\toL 0,
	\end{split}
\end{equation}
uniformly in $k$, thanks to the 
$k$-independent $L^2$ estimate 
\eqref{eq:tilde-uk-Linf-L2-est}. 
We can use this and \eqref{eq:L1-conv-trunc} 
to show that $\Seqb{\tilde u_L}_{L>0}$ is a 
Cauchy sequence in $L^1(\tilde \Omega 
\times (0,T)\times M)$. Indeed, for any $L_1,L_2>0$ 
and $k\in \N$,
$$
\tilde{E}\left[\int_0^T\int_M 
\abs{\tilde u_{L_1}-\tilde u_{L_2}}\, dV(\mx)\,dt\right]
\leq A(L_1,k)+B(L_1,L_2,k)+C(L_2,k),
$$
where 
\begin{align*}
	& A(L_1,k):=\tilde{E}\left[\int_0^T\int_M 
	\abs{\tilde u_{L_1}-\cT_{L_1}(\tilde u_k)}
	\, dV(\mx) \,dt\right],
	\\ & 
	B(L_1,L_2,k):=\tilde{E}\left[\int_0^T\int_M 
	\abs{\cT_{L_1}(\tilde u_k)-\cT_{L_2}(\tilde u_k)}
	\, dV(\mx) \,dt\right],
	\\ & 
	C(L_2,k):=\tilde{E}\left[\int_0^T\int_M 
	\abs{\cT_{L_2}(\tilde u_k)-\tilde u_{L_2}}
	\, dV(\mx)\,dt\right].
\end{align*}
Fix any small $\kappa>0$. First, given 
\eqref{eq:trunc-control}, there exists 
$L_0=L_0(\kappa)>0$ such that for any $L_1,L_2\ge L_0$, 
we have $B_1(L_1,L_2,k)<\kappa/3$, uniformly in $k$. 
On the other hand, by \eqref{eq:L1-conv-trunc}, for 
fixed $L_1,L_2,\ge L_0$, there exists 
$k_0=k_0(\kappa,L_1,L_2)$ such that 
$A(L_1,k)<\kappa/3$ and $C(L_2,k)<\kappa/3$ 
for all $k\ge k_0$. Consequently, 
$$
\lim_{L_1,L_2\to \infty}\tilde{E}
\left[\int_0^T\int_M 
\abs{\tilde u_{L_1}-\tilde u_{L_2}}
\, dV(\mx)\,dt \right]=0.
$$
In other words, $\Seqb{\tilde u_L}$ is a 
Cauchy sequence in the Banach space 
$L^1\bigl(\tilde \Omega\times (0,T)\times M\bigr)$,
with measure $d\tilde \prob\times dt\times dV$. 
This implies the existence of 
$\tilde u \in L^1\bigl(\tilde \Omega \times 
(0,T)\times M\bigr)$ such that
\begin{equation}\label{eq:uL-limit}
	\lim_{L\to \infty} 
	\tilde E\left[\norm{\tilde u
	-\tilde u_{L}}_{L^1((0,T)\times M)}\right]=0.
\end{equation}

Arguing as in \eqref{eq:trunc-control}, using 
the $L^2_{\omega,t,\mx}$ bound on $\tilde u$ 
coming from \eqref{eq:tuk-strong-conv}, it follows that 
\begin{equation}\label{eq:trunc-control2}
	\cT_L(\tilde u)\toL \tilde u 
	\quad \text{in $L^1_{\omega,t,\mx}$}, 
\end{equation}	
which---via \eqref{eq:L1-conv-trunc} and 
\eqref{eq:uL-limit}---implies that 
\begin{equation}\label{eq:TL-uk-conv}
	\cT_L(\tilde u_k)
	\tok \cT_L(\tilde u) 
	\quad \text{in $L^1_{\omega,t,\mx}$}.
\end{equation}
As a result, for any $L>0$,
$$
\tilde E \left[\norm{\tilde u
-\tilde u_k}_{L^1((0,T)\times M)}\right]
\leq I_1(L)+I_2(L,k)+I_3(k),
$$
where
\begin{align*}
	& I_1(L) :=\tilde E \left[\norm{\tilde u
	-\cT_L(\tilde u)}_{L^1((0,T)\times M)}\right]
	\toL 0, \quad \text{by \eqref{eq:trunc-control2}},
	\\ 
	& I_2(L,k):= \tilde E \left[\norm{\cT_L(\tilde u)
	-\cT_L(\tilde u_k)}_{L^1((0,T)\times M)}\right]
	\tok 0, \quad \text{by \eqref{eq:TL-uk-conv}, for 
	fixed $L$},
	\\ & 
	I_3(L,k):=\tilde E \left[\norm{\cT_L(\tilde u_k)
	-\tilde u_k}_{L^1((0,T)\times M)}\right]\toL 0, 
	\quad \text{by \eqref{eq:trunc-control}, 
	uniformly in $k$}.
\end{align*}
Therefore, $\tilde u_k\tok \tilde u$ in 
$L^1_{\omega,t,\mx}$, which brings to an end 
the proof of \eqref{eq:tuk-strong-conv}. 

Finally, the limit $\tilde{u}$ is adapted to the 
filtration $\Seqb{\tilde \F_t}$ defined 
in \eqref{eq:tilde-Ft}, since it is 
the (strong $L^1_{\omega,t,x}$) limit 
of the $\Seqb{\tilde \F_t^k}$--adapted processes 
$\Seqb{\tilde u_k}$, cf.~part (ii) of 
Lemma \ref{lem:SPDE-on-new-prob-space}.
\end{proof}

In view of the $k$-independent $L^2$ estimate 
in \eqref{eq:tilde-u0k-est}, there is no loss of generality 
in assuming that 
\begin{equation}\label{eq:new-tilde-data}
	\tilde u_{0,k} \weak \tilde u_0
	\quad \text{in $L^2(\tilde \Omega \times M)$ as 
	$k\to \infty$}, 
\end{equation}
for some $\tilde u_0\in L^2(\tilde \Omega\times M)$.  
Given Theorem \ref{thm:strong-conv}, 
we can now conclude the proof of our main result, 
namely Theorem \ref{main-thm}, by 
appealing to

\begin{lemma}[weak solution]\label{lem:weak-sol}
In addition to the conditions listed in 
Theorem \ref{thm:strong-conv}, suppose
($C_\Phi$--3) holds. Let $\tilde u$ be the 
strong $L^1_{\omega,t,\mx}$ limit identified 
in \eqref{eq:tuk-strong-conv}. Then $\tilde{u}$ is a 
weak martingale solution to the stochastic conservation law 
\eqref{cl-1} with $\mx$-discontinuous flux 
$\mff(\mx,\cdot)$ and initial data $\tilde u_0\in 
L^2_{\tilde \prob}\bigl(\tilde \Omega;L^2(M)\bigr)$, 
cf.~\eqref{eq:new-tilde-data}.
\end{lemma}

\begin{proof}
Given the convergences \eqref{eq:tuk-strong-conv} 
and \eqref{eq:new-tilde-data}, we will argue along 
the lines of the proof of Lemma 
\ref{lem:weak-limit-kinetic}. Introduce 
the random variable 
\begin{align*}
	\cI(\omega,t) &:= \int_M \tilde u(t)\, \varphi\, dV\, dt
	-\int_M \tilde u_0\, \varphi\, dV 
	-\int_0^T \int_M \mff(\mx, \tilde u)\cdot 
	\nabla \varphi \, dV \, dt' 
	\\ & \qquad
	-\int_0^t\int_M \Phi(\mx,\tilde u) 
	\varphi(\mx) \, dV \, d\tilde W_{t'}
	\\ & =:\cI^{(1)}(\omega,t)-\cI^{(2)}(\omega)
	-\cI^{(3)}(\omega,t)-\cI^{(4)}(\omega,t), 
	\qquad \varphi \in C_c^2(M).
\end{align*}
By our assumptions on $\mff,\Phi$ 
and the $L^2$ bound in \eqref{eq:tuk-strong-conv}, 
$\cI\in L^2(\tilde \Omega\times (0,T))$. 

Now the proof of the lemma consists of showing that 
\begin{equation}\label{eq:I=0}
	\norm{\cI}_{L^2(\tilde \Omega\times (0,T))}^2
	=\tilde E \left[\int_0^T 	\bigl(\cI(\omega,t)
	\bigr)^2\,dt\right]=0,
\end{equation}
which implies that $\cI=0$ for 
$d\tilde \prob\times dt$--a.e.~$(\omega,t)\in 
\tilde \Omega\times (0,T)$. By the Fubini theorem, 
this shows that $\cI(\omega,t)=0$ for a.e.~$t\in (0,T)$, 
$\tilde \prob$--a.s.

Let us establish \eqref{eq:I=0}. Since the simple functions 
are dense in $L^2$, it is enough to establish that 
$\tilde E \left[\int_0^T \chi_Z\cI\,dt\right]=0$, where 
$\chi_Z=\chi_Z(\omega,t)\in L^\infty_{\omega,t}$ 
is the characteristic function of an arbitrary 
measurable set $Z\subset \tilde \Omega \times (0,T)$. 
Each $\tilde u_k=\tilde u_k(\omega,t)=
\tilde u_k(\omega,t,\mx)$ is a weak 
solution of \eqref{PP-1-SJ}, so that
\begin{equation}\label{mild-ppde-SJ}
	\begin{split}
		\cI_k(\omega,t) & :=
		\underbrace{\int_M  \tilde u_k(\omega,t) 
		\varphi \, dV(\mx)}_{=:\cI_k^{(1)}}
		-\underbrace{\int_M \tilde u_{0,k}(\omega) 
		\varphi \, dV(\mx)}_{=:\cI_k^{(2)}} 
		\\ & \quad 
		-\underbrace{\int_0^t \int_M \mff_k(\mx, \tilde u_k) 
		\cdot \nabla \varphi \, dV(\mx) \, dt'}_{=:\cI_k^{(3)}}
 		-\underbrace{\int_0^t\int_M \Phi(\mx,\tilde u_k)
 		\varphi \,dV(\mx) \, d\tilde{W}_{t'}}_{=:\cI_k^{(4)}}
 		\\ & \quad  
 		-\underbrace{\eps_k\int_0^t \int_M \tilde u_k 
 		\Delta \varphi \,dV(\mx)\, dt'}_{=:\cI_k^{(5)}}
 		-\underbrace{\delta_k \int_M \tilde u_k(\omega,t)
 		\Delta \varphi \, dV(\mx)}_{=:\cI_k^{(6)}}
 		\\ & \quad 
 		-\underbrace{\delta_k\int_M \tilde{u}_{0,k}(\omega)
 		\Delta \varphi \, dV(\mx)}_{=:\cI_k^{(7)}} =0, 
 		\quad \text{for $d\tilde{\prob}\times dt$ 
 		a.e.~$(\omega,t)\in \tilde \Omega\times (0,T)$}.
	\end{split}
\end{equation} 

Let us compute the limit of each
$\tilde E \left[\int_0^T\chi_Z \cI_k^{(l)}\,dt\right]$ as 
$k\to \infty$, $l=1,\ldots,7$. 
By the $k$-independent $L^2$ estimates 
\eqref{eq:tilde-uk-Linf-L2-est}, \eqref{eq:tilde-u0k-est}, 
and also $\chi_Z \Delta \varphi \in L^2_{\omega,t,\mx}$, 
\begin{equation}\label{eq:weak-sol-tmp1}
	\lim_{k\to\infty}
	\tilde E \left[\int_0^T\chi_Z \cI_k^{(l)}\,dt\right]=0, 
	\quad l=5,6,7.
\end{equation}

Next, by the convergences 
\eqref{eq:tuk-strong-conv}, \eqref{eq:new-tilde-data} 
and also $\chi_Z\varphi\in 
L^\infty_{\omega,t,\mx}\subset L^2_{\omega,t,\mx}$,
\begin{equation}\label{eq:weak-sol-tmp2}
	\begin{split}
		&\lim_{k\to\infty}
		\tilde E \left[\int_0^T\chi_Z\cI_k^{(1)}\,dt\right]
		=\tilde E\left[\int_0^T\int_M \chi_Z
		\tilde u(t) \varphi \, dV(\mx) \, dt\right]
		= \tilde E \left[\int_0^T\chi_Z\cI^{(1)}\,dt\right],
		\\ &
		\lim_{k\to\infty}
		\tilde E \left[\int_0^T\chi_Z\cI_k^{(2)}\,dt\right]
		=\tilde E\left[\int_0^T\int_M \chi_Z
		\tilde u_0 \varphi \, dV(\mx) \, dt\right]
		=\tilde E\left[
		\int_0^T\chi_Z\cI^{(2)}\,dt\right].
	\end{split}
\end{equation}

Regarding the flux term $\cI_k^{(3)}$, note that
\begin{align*}
	\abs{\mff_k(\mx, \tilde u_k)-\mff(\mx, \tilde u)}
	& \leq \abs{\mff_k(\mx, \tilde u_k)-\mff_k(\mx, \tilde u)}
	+\abs{\mff_k(\mx, \tilde u)-\mff(\mx,\tilde u)}
	\\ & \leq 
	\sup\limits_{\lambda \in \R}
	\abs{\mff_k'(\mx,\lambda)}
	\abs{\tilde u_k-\tilde u}
	+\sup\limits_{\lambda \in \R}
	\abs{\mff(\mx,\lambda)-\mff_k(\mx,\lambda)}.
\end{align*}
In view of \eqref{eq:tilde-uk-Linf-L2-est} and 
\eqref{eq:tuk-strong-conv}, we may as well assume 
that $\tilde u_k\to \tilde u$ in $L^p_{\omega,t,\mx}$ 
for all $p\in [1,2)$. Hence, by the assumptions 
($C_f$--1) and ($C_f$--2),
$$
\tilde E \left[\int_0^T\int_0^t \int_M 
\abs{\mff_k(\mx, \tilde u_k)
-\mff(\mx, \tilde u)}\,dV(\mx)
\, dt' \,dt \right] 
\tok 0.
$$
Since $\chi_Z \nabla \varphi\in L^\infty_{\omega,t,\mx}$, 
this implies that
\begin{equation}\label{eq:weak-sol-tmp3}
	\lim_{k\to\infty}
	\tilde E \left[\int_0^T\chi_Z\cI_k^{(3)}\,dt\right]
	=\tilde E\left[\int_0^T\int_0^t\int_M 
	\chi_Z\mff(\mx, \tilde u)\cdot 
	\nabla \varphi \, dV(\mx) \, dt'\, dt\right]
	=\tilde E \left[\int_0^T\chi_Z\cI^{(3)}\,dt\right].
\end{equation}

Finally, we consider the stochastic integral term 
$\cI_k^{(4)}$. Using the BDG inequality \eqref{eq:BDG} 
with $p=1$ and then the Cauchy-Schwarz inequality (twice), 
\begin{align*}
	&\abs{\tilde E \left[\int_0^T \chi_Z \cI^{(4)}\,dt\right]
	-\tilde E \left[\int_0^T\chi_Z\cI_k^{(4)}\,dt\right]}
	\leq \int_0^T\tilde E \left[\, \abs{\int_0^t
	\int_M \bigl(\Phi(\mx,\tilde u)
	-\Phi(\mx, \tilde u_k)\bigr)
 	\varphi \,dV\, d\tilde W_{t'}}\, \right]\,dt
 	\\ & \qquad 
 	\lesssim_T \tilde{E}\left[\left(\int_0^T
 	\left(\int_M \bigl(\Phi(\mx, \tilde u)
 	-\Phi(\mx, \tilde u_k) \bigr)\varphi \, dV\, 
 	\right)^2 \, dt'\right)^{\frac12}\right]
 	\\ & \qquad
 	\leq \left(\tilde E \left[\int_0^T \int_M 
 	\left(\Phi(\mx,\tilde u)
 	-\Phi(\mx,\tilde u_k)\right)^2\,dV(\mx)\,dt\right] 
 	\right)^{\frac12}\norm{\varphi}_{L^2(M)}.
\end{align*} 
By \eqref{eq:tuk-strong-conv}, passing to a 
subsequence if necessary, there is no loss 
of generality in assuming that 
$$
\tilde u_k(\omega,t,\mx) \tok 
\tilde u(\omega,t,\mx) \quad 
\text{for $d\tilde\prob \times dt\times dV(\mx)$ 
a.e.~$(\omega,t,\mx)\in \tilde \Omega
\times (0,T)\times M$},
$$
so that $I_k(\omega,t,\mx):=\left(\Phi(\mx,\tilde u)
-\Phi(\mx,\tilde u_k)\right)^2 \tok 0$ 
for a.e.~$(\omega,t,\mx)\in \tilde 
\Omega\times (0,T)\times M$. Besides, 
by $({\rm C}_{\Phi}-3)$, we have a dominating function
$\abs{I_k(\omega,t,\mx)}\leq 4\sup\limits_{\lambda \in \R}
\abs{\Phi(\mx,\lambda)}^2\in L^1(\tilde \Omega
\times(0,T)\times M)$. Thus, by the dominated 
convergence theorem,
\begin{equation}\label{eq:weak-sol-tmp4}
	\abs{\tilde E \left[\int_0^T \chi_Z \cI^{(4)}\,dt\right]
	-\tilde E \left[\int_0^T\chi_Z\cI_k^{(4)}\,dt\right]}\tok 0.
\end{equation}

Summarising, given \eqref{mild-ppde-SJ} 
and the convergences \eqref{eq:weak-sol-tmp1}, 
\eqref{eq:weak-sol-tmp2}, \eqref{eq:weak-sol-tmp3} 
and \eqref{eq:weak-sol-tmp4},
$$
\tilde E \left[\int_0^T \chi_Z\left(\cI^{(1)}
-\cI^{(2)}-\cI^{(3)}-\cI^{(4)}\right)\,dt\right]=0, 
\quad \text{for any measurable set 
$Z\subset \tilde \Omega \times (0,T)$}.
$$ 
This concludes the proof that 
$\cI= \cI^{(1)}-\cI^{(2)}-\cI^{(3)}-\cI^{(4)}$ 
satisfies \eqref{eq:I=0}, and thus we have 
shown that the limit $\tilde u$ 
satisfies \eqref{mild-cl}, $\tilde \prob$--a.s., for 
a.e.~$t\in [0,T]$. Finally, the right-hand 
side of \eqref{mild-cl} clearly defines 
a continuous stochastic process. Therefore, 
$\int_M \tilde u(t)\, \varphi\,dV(\mx)$ has 
a continuous modification such that \eqref{mild-cl} 
holds for \textit{all} $t$.
\end{proof}

\begin{remark}
We note that since the flux $\mff$ generally 
is discontinuous in $\mx$, we have not 
discussed the admissibility of the obtained weak solution. 
Admissibility (entropy) conditions for stochastic 
conservation laws with discontinuous 
flux are little studied, 
see \cite{Andreianov:2010fk} for an overview 
of different conditions in the deterministic case.
Having said that, the entropy condition in 
\cite{Karlsen:2003lz} has a straightforward 
adaption to the present stochastic 
problem. The weak solution constructed 
herein can be easily shown to satisfy this condition.
On the other hand, if $\mff$ is a $C^1$ vector field, then 
the admissibility condition from 
\cite{Galimberti:2018aa} can be used, in which 
case we can prove that $\Seq{\tilde{u}_k}$ 
converges towards the entropy solution. 
This entropy solution is unique and 
one can in this case upgrade the obtained ``martingale" 
solution to a probabilistic strong 
solution \textit{\`ala} Yamada-Watanabe.
\end{remark}

\section{Acknowledgement} 
This work was supported by 
project P35508 of the Austrian Science Fund FWF.

%\bibliographystyle{abbrv}
%\bibliography{Bibliography}

\end{document}